\documentclass[12pt,a4paper,twoside]{article}
\usepackage[colorlinks=true,linkcolor=blue,citecolor=blue]{hyperref}
\usepackage{amsmath,mathrsfs,multicol,amsfonts,mathtools,amsthm, bm,fancyhdr,lastpage,xcolor,titlesec,cleveref}
\usepackage[english]{babel}
\usepackage{nicematrix}
\usepackage{amssymb}
\usepackage[bottom,multiple]{footmisc}
\usepackage[top=2.5cm,bottom=2.5cm,left=2.5cm,right=2.5cm]{geometry}
\definecolor{myeditcolor}{named}{blue}
\titleformat{\section}[block]{\bfseries\large}{\thesection. }{2pt}{}
\theoremstyle{definition}
\newtheorem{theorem}{Theorem}[section]
\newtheorem{definition}[theorem]{Definition}

\newtheorem{remark}[theorem]{Remark}

\begin{document}
\thispagestyle{empty}
\begin{center}
\noindent\textbf{{\Large Canonical Connections and Algebraic Ricci Solitons on Lorentzian 4-Dimensional Nilpotent Lie Groups}}
\end{center}
\begin{center} \noindent Youssef Ayad\footnote{corresponding author: youssef.ayad@edu.umi.ac.ma}\\
\small{\textit{$^1$Faculty of sciences, Moulay Ismail University of Mekn\`{e}s}}\\
\small{\textit{B.P. 11201, Zitoune, Mekn\`{e}s}, Morocco}
\end{center}
\begin{abstract}
In this paper, we compute the canonical and Kobayashi-Nomizu connections, together with their curvature, on Lorentzian four-dimensional nilpotent Lie groups endowed with a product structure. We also classify the algebraic Ricci solitons associated with these connections.
\end{abstract}
\begin{center}
\end{center}
\begin{center}
\end{center}
\section{Introduction}
\hspace*{3mm}    An algebraic Ricci soliton is a Lie group endowed with a left-invariant Riemannian metric whose Ricci operator can be expressed as the sum of a homothety and a derivation of its Lie algebra. This notion was introduced by Lauret in \cite{lauret2001ricci}, where he established a fundamental connection between solvsolitons and Ricci solitons on homogeneous Riemannian manifolds. More precisely, he proved that any Riemannian solvsoliton metric is a Ricci soliton. In a later work \cite{lauret2011ricci}, Lauret extended his study of Ricci solitons on solvable Lie groups by investigating the structure of solsolitons. He showed that every solsoliton arises, up to isometry, from a nilsoliton together with an abelian Lie algebra of symmetric derivations of the associated metric Lie algebra. He also proved the uniqueness of solsolitons on a fixed solvable Lie group up to isometry and scaling, and classified all solsolitons of dimension at most four. In \cite{jablonski2011concerning,jablonski2014homogeneous,jablonski2015homogeneous}, Jablonski significantly extended Lauret's work by proving that homogeneous Ricci solitons are algebraic, characterizing solvable Lie groups admitting Ricci soliton metrics, and establishing several structural and rigidity results for homogeneous Ricci solitons.

The notion of an algebraic Ricci soliton was extended to the pseudo-Riemannian setting in \cite{batat2017algebraic}, where Batat and Onda investigated algebraic Ricci solitons on three-dimensional Lorentzian Lie groups. They obtained a complete classification of such solitons and showed that, in contrast to the Riemannian case, Lorentzian Ricci solitons are not necessarily algebraic. Earlier, Onda \cite{onda2014examples} initiated the study of algebraic Ricci solitons in the pseudo-Riemannian setting by providing several examples. The classification of Lorentzian algebraic Ricci solitons on four-dimensional nilpotent Lie groups was carried out in \cite{aitbenhaddou2026lorentzian}, building on the classification of left-invariant Lorentzian metrics on four-dimensional nilpotent Lie groups established in \cite{bokan2015lorentz}. More generally, a complete study of four-dimensional Lorentzian algebraic Ricci solitons was given in \cite{garcia}. As a consequence, the authors proved that, in sharp contrast to the Riemannian setting, every connected and simply connected four-dimensional Lie group admits a left-invariant Lorentzian metric that is a Ricci soliton.

In \cite{etayo2016distinguished}, Etayo and Santamar{\'{\i}}a studied several affine connections on manifolds endowed with a product or a complex structure. In particular, they investigated the canonical connection and the Kobayashi-Nomizu connection associated with a product structure. Later, Wang \cite{wang2022canonical} classified algebraic Ricci solitons associated with the canonical and Kobayashi-Nomizu connections on three-dimensional Lorentzian Lie groups equipped with a product structure.

In this paper, we introduce a product structure on Lorentzian four-dimensional nilpotent Lie groups and compute the corresponding canonical and Kobayashi-Nomizu connections, together with their curvature. We then define algebraic Ricci solitons associated with these connections and provide a complete classification of such solitons on Lorentzian four-dimensional nilpotent Lie groups endowed with this product structure.
\section{Preliminaries}
Let $(G, g)$ be a Lorentzian Lie group and let $\left(\mathfrak{g}, \langle ., .\rangle\right)$ be its associated Lorentzian Lie algebra of dimension $n$. Consider $\nabla$ to be the Levi-Civita connection associated to $\left(\mathfrak{g}, \langle \cdot, \cdot\rangle\right)$; It is well known that $\nabla$ is characterized by the following Koszul formula
$$2\langle \nabla_u v, w\rangle = \langle [u, v], w\rangle + \langle [w, u], v\rangle + \langle [w, v], u\rangle, \quad \forall u, v, w \in \mathfrak{g}.$$
The Riemann curvature tensor $R$ of $\left(\mathfrak{g}, \langle \cdot, \cdot\rangle\right)$ assigns to each pair $u, v \in \mathfrak{g}$ the linear transformation
$$R_{uv} = \nabla_{[u, v]} - [\nabla_u, \nabla_v].$$
More precisely, we have
$$R_{uv}w = \nabla_{[u, v]}w - \nabla_u\nabla_v w + \nabla_v\nabla_u w, \quad \forall u, v ,w \in \mathfrak{g}.$$
Let $\mathscr{B} = \left\lbrace v_1, \ldots, v_n\right\rbrace$ be an orthonormal basis of $\left(\mathfrak{g}, \langle \cdot, \cdot\rangle\right)$, that is, it satisfies 
$$\langle v_i, v_j \rangle = 0 \quad \forall i \neq j, \quad \langle v_i, v_i \rangle = 1 \quad \forall i = 1, ..., n - 1, \quad \langle v_n, v_n \rangle = -1.$$
The structure constants $\xi_{ijk}$ of $\left(\mathfrak{g}, \langle ., .\rangle\right)$ are defined by $\xi_{ijk} = \langle [v_i, v_j], v_k\rangle$. By the Koszul formula, the Levi-Civita connection of $\left(\mathfrak{g}, \langle \cdot, \cdot\rangle\right)$ is described by
$$\nabla_{v_i}v_j = \displaystyle{\sum_{k = 1}^{n} \langle v_k, v_k\rangle \frac{1}{2}\left(\xi_{ijk} - \xi_{jki} + \xi_{kij}\right)v_k}.$$
The Ricci tensor $\rho$ of $\left(\mathfrak{g}, \langle \cdot, \cdot\rangle\right)$ is defined by
$$\rho(u, v) = \displaystyle{\sum_{k = 1}^{n} \langle v_k, v_k\rangle}\langle R_{v_ku}v_k, v\rangle.$$
The Ricci operator $\operatorname{Ric}$ of $\left(\mathfrak{g}, \langle \cdot, \cdot\rangle\right)$ is defined by
$$\rho(u, v) = \langle \operatorname{Ric}(u), v\rangle,$$
or equivalently
$$\operatorname{Ric}(u) = \displaystyle{\sum_{k = 1}^{n} \langle v_k, v_k\rangle R_{v_ku}v_k}.$$
\subsection{Nilpotent Lie groups of dimension four}
It is well known that when a Lie group $G$ with Lie algebra $\mathfrak{g}$ is nilpotent, then the exponential map $\operatorname{exp} : \mathfrak{g} \longmapsto G$ is a diffeomorphism. In this case $G$ is diffeomorphic to $\mathbb{R}^n$ and hence simply connected. It is well known also that when $G$ is simply connected then $\operatorname{Aut}(G) \backsimeq \operatorname{Aut}(\mathfrak{g})$ \cite{warner1983foundations}. This is a useful fact in the classification of metrics on simply connected Lie groups. According to \cite{magnin1986algebres,goze1996nilpotent}, there are only two non-Abelian nilpotent Lie algebras of dimension $4$: $\mathfrak{g}_4$ and $\mathfrak{h}_3 \oplus \mathbb{R}$ with corresponding Lie groups $G_4$ and $H_3 \times \mathbb{R}$.\\
The Lie algebra $\mathfrak{g}_4$ has a basis $\mathcal{B} = \left\lbrace e_1, e_2, e_3, e_4\right\rbrace$ with non-zero Lie brackets
$$[e_1, e_2] = e_3, \qquad [e_1, e_3] = e_4.$$
It is $3$-step nilpotent, with a one-dimensional center $\mathcal{Z}(\mathfrak{g}_4) = \operatorname{Span}(e_4)$ and a two-dimensional commutator subalgebra $\mathscr{D}(\mathfrak{g}_4) = \operatorname{Span}(e_3, e_4)$.\\
In contrast, $\mathfrak{h}_3 \oplus \mathbb{R}$ has a basis $\mathcal{B} = \left\lbrace e_1, e_2, e_3, e_4\right\rbrace$ with the only non-zero Lie bracket: $[e_1, e_2] = e_3$.
This Lie algebra is $2$-step nilpotent, with a two-dimensional center $\mathcal{Z}\left(\mathfrak{h}_3 \oplus \mathbb{R}\right) = \operatorname{Span}(e_3, e_4)$ and a one-dimensional commutator subalgebra, spanned by $e_3$.

Let $\mathscr{B} = \left\lbrace v_1, v_2, v_3, v_4\right\rbrace$ be an othonormal basis of $\left( \mathfrak{g}_4, \langle \cdot, \cdot\rangle\right)$ or $\left( \mathfrak{h}_3 \oplus \mathbb{R}, \langle \cdot, \cdot\rangle\right)$. In this case, the Levi-Civita connection is described by
\begin{equation} \label{Levi}
\nabla_{v_i}v_j = \displaystyle{\sum_{k = 1}^{4} \langle v_k, v_k\rangle \frac{1}{2}\left(\xi_{ijk} - \xi_{jki} + \xi_{kij}\right)v_k}.
\end{equation}
The Ricci tensor $\rho$ is given by
\begin{equation} \label{riccitensor}
\rho(u, v) = \displaystyle{\sum_{k = 1}^{4} \langle v_k, v_k\rangle\langle R_{v_ku}v_k, v\rangle}.
\end{equation}
The Ricci operator $\operatorname{Ric}$ is given by
\begin{equation} \label{riccioperator}
\operatorname{Ric}(u) = \displaystyle{\sum_{k = 1}^{4} \langle v_k, v_k\rangle R_{v_ku}v_k}.
\end{equation}
We define a product structure $J$ on $\left( \mathfrak{g}_4, \langle \cdot, \cdot\rangle\right)$ and $\left( \mathfrak{h}_3 \oplus \mathbb{R}, \langle \cdot, \cdot\rangle\right)$ by
$$Jv_1 = v_1, \qquad Jv_2 = v_2, \qquad Jv_3 = v_3, \qquad Jv_4 = -v_4,$$
then $J^2 = \operatorname{Id}$ and $\langle Jv_k, Jv_k\rangle = \langle v_k, v_k\rangle$. By \cite{etayo2016distinguished}, we define the canonical connection $\nabla^0$ and the Kobayashi-Nomizu connection $\nabla^1$ as follows:
\begin{align}
\nabla^0_u v
&= \nabla_u v - \frac{1}{2}(\nabla_u J)Jv,
\label{canonical}\\
\nabla^1_u v
&= \nabla^0_u v - \frac{1}{4}\left[
(\nabla_v J)Ju - (\nabla_{Jv} J)u
\right],
\label{kobayashinomizu}
\end{align}
where
\begin{equation}\label{leibnizrule}
(\nabla_u J)v = \nabla_u (Jv) - J(\nabla_u v),
\end{equation}
which follows from the Leibniz rule
$$\nabla_X (JY) = (\nabla_X J)Y + J(\nabla_X Y).$$
We will use the following notation $\nabla_u (J)v := (\nabla_u J)v$. Next, we define the curvature associated with $\nabla^0$ and $\nabla^1$ as follows
\begin{equation}\label{canonicalcurvature}
R^0_{uv}w = \nabla^0_{[u, v]}w - \nabla^0_u\nabla^0_v w + \nabla^0_v\nabla^0_u w, \quad \forall u, v ,w \in \mathfrak{g}.
\end{equation}
\begin{equation}\label{kobayashicurvature}
R^1_{uv}w = \nabla^1_{[u, v]}w - \nabla^1_u\nabla^1_v w + \nabla^1_v\nabla^1_u w, \quad \forall u, v ,w \in \mathfrak{g}.
\end{equation}
The Ricci tensors associated to the canonical connection and the Kobayashi-Nomizu connection are defined by
\begin{equation} \label{canonicalriccitensors}
\rho^0(u, v) = \displaystyle{\sum_{k = 1}^{4} \langle v_k, v_k\rangle\langle R^0_{v_ku}v_k, v\rangle}, \qquad \rho^1(u, v) = \displaystyle{\sum_{k = 1}^{4} \langle v_k, v_k\rangle\langle R^1_{v_ku}v_k, v\rangle}.
\end{equation}
The Ricci operators $\operatorname{Ric}^0$ and $\operatorname{Ric}^1$ are given by
\begin{equation} \label{canonicalriccioperators}
\operatorname{Ric}^0(u) = \displaystyle{\sum_{k = 1}^{4} \langle v_k, v_k\rangle R^0_{v_ku}v_k}, \qquad \operatorname{Ric}^1(u) = \displaystyle{\sum_{k = 1}^{4} \langle v_k, v_k\rangle R^1_{v_ku}v_k}.
\end{equation}
Let 
\begin{equation}\label{tilderho}
\widetilde{\rho}^0(u, v) = \frac{\rho^0(u, v) + \rho^0(v, u)}{2}, \qquad \widetilde{\rho}^1(u, v) = \frac{\rho^1(u, v) + \rho^1(v, u)}{2}
\end{equation}
and 
\begin{equation}\label{tildeRic}
\widetilde{\rho}^0(u, v) = \langle\widetilde{\operatorname{Ric}}^0(u), v\rangle, \qquad \widetilde{\rho}^1(u, v) = \langle\widetilde{\operatorname{Ric}}^1(u), v\rangle.
\end{equation}
Let $\operatorname{Ric}^0 = \left(\operatorname{Ric}^0_{ij}\right)_{1 \leq i, j \leq 4}$ and $\widetilde{\operatorname{Ric}}^0 = \left(\widetilde{\operatorname{Ric}}^0_{ij}\right)_{1 \leq i, j \leq 4}$ be the matrices of $\operatorname{Ric}^0$ and $\widetilde{\operatorname{Ric}}^0$ in the orthonormal basis $\left\lbrace v_1, v_2, v_3, v_4\right\rbrace$, respectively. On the one hand, we have
$$\widetilde{\rho}^0(v_4, v_1) = \langle  \widetilde{\operatorname{Ric}}^0(v_4), v_1\rangle = \widetilde{\operatorname{Ric}}^0_{14}$$
On the other hand, we have
$$\widetilde{\rho}^0(v_4, v_1) = \frac{\rho^0(v_4, v_1) + \rho^0(v_1, v_4)}{2} = \frac{\langle  \operatorname{Ric}^0(v_4), v_1\rangle + \langle  \operatorname{Ric}^0(v_1), v_4\rangle}{2} = \frac{\operatorname{Ric}^0_{14} - \operatorname{Ric}^0_{41}}{2}$$
Hence $\widetilde{\operatorname{Ric}}^0_{14} = \frac{\operatorname{Ric}^0_{14} - \operatorname{Ric}^0_{41}}{2}$. In fact, we have
{\small$$\widetilde{\operatorname{Ric}}^0_{i4} = -\widetilde{\operatorname{Ric}}^0_{4i} = \frac{\operatorname{Ric}^0_{i4} - \operatorname{Ric}^0_{4i}}{2},\; \forall i = 1, 2, 3, \quad \widetilde{\operatorname{Ric}}^0_{ij} = \widetilde{\operatorname{Ric}}^0_{ji}  = \frac{\operatorname{Ric}^0_{ij} + \operatorname{Ric}^0_{ji}}{2}, \; \forall i, j \in \left\lbrace 1, 2, 3\right\rbrace.$$}
The same is true for $\operatorname{Ric}^1$.
\begin{definition}
\begin{enumerate}
\item A Lorentzian four-dimensional Lie algebra $\left( \mathfrak{g}, \langle \cdot, \cdot\rangle, J\right)$ is called flat with respect to $\nabla^0$ (resp. $\nabla^1$), if it satifies $R^0 = 0$ (resp. $R^1 = 0$).
\item A Lorentzian four-dimensional Lie algebra $\left( \mathfrak{g}, \langle \cdot, \cdot\rangle, J\right)$ is called Ricci-flat with respect to $\nabla^0$ (resp. $\nabla^1$), if it satifies $\operatorname{Ric}^0 = 0$ (resp. $\operatorname{Ric}^1 = 0$).
\item A Lorentzian four-dimensional Lie algebra $\left( \mathfrak{g}, \langle \cdot, \cdot\rangle, J\right)$ is called the first (resp. second) kind algebraic Ricci soliton associated to $\nabla^0$, if it satifies 
$$\operatorname{Ric}^0 = c\operatorname{Id} + D \quad (\text{resp.}\; \widetilde{\operatorname{Ric}}^0 = c\operatorname{Id} + D),$$
where $c$ is a real number, and $D$ is a derivation of $\mathfrak{g}$, that is
$$[D(u), v] + [u, D(v)] = D[u, v] \quad \forall u, v \in \mathfrak{g}.$$
\item A Lorentzian four-dimensional Lie algebra $\left( \mathfrak{g}, \langle \cdot, \cdot\rangle, J\right)$ is called the first (resp. second) kind algebraic Ricci soliton associated to $\nabla^1$, if it satifies 
$$\operatorname{Ric}^1 = c\operatorname{Id} + D \quad (\text{resp.}\; \widetilde{\operatorname{Ric}}^1 = c\operatorname{Id} + D).$$
\end{enumerate}
\end{definition}
\section{Lorentzian algebraic Ricci solitons on $\mathfrak{g}_4$}
\begin{theorem} \cite{bokan2015lorentz}
The set $\mathcal{S}^L(\mathfrak{g}_4)$ of nonequivalent Lorentz inner product on algebra $\mathfrak{g}_4$ in basis $\left\lbrace e_1, e_2, e_3, e_4\right\rbrace$ is represented by the following matrices
\begin{multicols}{3}
$S_A^{\pm} = \begin{bmatrix}
\pm1 & 0 & 0 & 0\\
0 & \mp1 & 0 & 0\\
0 & 0 & a & b\\
0 & 0 & b & c
\end{bmatrix}$\par
$S_A = \begin{bmatrix}
1 & 0 & 0 & 0\\
0 & 1 & 0 & 0\\
0 & 0 & a & b\\
0 & 0 & b & c
\end{bmatrix}$\par
$S_1^{\lambda} = \begin{bmatrix}
1 & 0 & 0 & 0\\
0 & 0 & 0 & 1\\
0 & 0 & \lambda & 0\\
0 & 1 & 0 & 0
\end{bmatrix}$\par
$S_2^{\lambda} = \begin{bmatrix}
0 & 0 & 0 & 1\\
0 & 1 & 0 & 0\\
0 & 0 & \lambda & 0\\
1 & 0 & 0 & 0
\end{bmatrix}$\par
$S_3^{\lambda} = \begin{bmatrix}
0 & 0 & 1 & 0\\
0 & 1 & 0 & 0\\
1 & 0 & 0 & 0\\
0 & 0 & 0 & \lambda
\end{bmatrix}$\par
$S_4^{\lambda} = \begin{bmatrix}
1 & 0 & 0 & 0\\
0 & 0 & 1 & 0\\
0 & 1 & 0 & 0\\
0 & 0 & 0 & \lambda
\end{bmatrix}$
\end{multicols}
with $\lambda > 0$, $b \geq 0$. Matrix $A = \begin{bmatrix}
a & b\\
b & c
\end{bmatrix}$ satisfies $\operatorname{det}(A) > 0$ in $S_A^{\pm}$ and $\operatorname{det}(A) < 0$ in $S_A$.
\end{theorem}
\begin{remark}
Note that as a vector space, $\mathfrak{g}_4$ is identified with $\mathbb{R}^4$ and then the elements $e_1$, $e_2$, $e_3$ and $e_4$ are the canonical basis elements of $\mathbb{R}^4$.
\end{remark}
\subsection{The Lorentzian Lie algebra $\left(\mathfrak{g}_4, S_1^{\lambda}\right)$}
Put
$$v_1 = e_1, \qquad v_2 = \frac{1}{\sqrt{\lambda}}e_3, \qquad v_3 = \frac{1}{\sqrt{2}}(e_2 + e_4), \qquad v_4 = \frac{1}{\sqrt{2}}(e_2 - e_4).$$
Then $\mathscr{B} = \left\lbrace v_1, v_2, v_3, v_4\right\rbrace$ is an orthonormal basis of the Lorentzian Lie algebra $\left(\mathfrak{g}_4, S_1^{\lambda}\right)$. The bracket in $\mathscr{B}$ is described by
$$[v_1, v_2] = \frac{1}{\sqrt{2\lambda}}(v_3 - v_4), \qquad [v_1, v_3] = \sqrt{\frac{\lambda}{2}}v_2, \qquad [v_1, v_4] = \sqrt{\frac{\lambda}{2}}v_2.$$

By \cite{aitbenhaddou2026lorentzian}, the Levi-Civita connection of $\left(\mathfrak{g}_4, S_1^{\lambda}\right)$ is described by the following formulas
\begin{multicols}{2}
$\nabla_{v_i}v_i=0,\; \forall i=1,\ldots,4$\par
$\nabla_{v_1}v_2=\frac{1-\lambda}{2\sqrt{2\lambda}}v_3-\frac{1-\lambda}{2\sqrt{2\lambda}}v_4$\par
$\nabla_{v_2}v_1=\frac{-(\lambda+1)}{2\sqrt{2\lambda}}v_3+\frac{1+\lambda}{2\sqrt{2\lambda}}v_4$\par
$\nabla_{v_1}v_3=\nabla_{v_1}v_4=\frac{\lambda-1}{2\sqrt{2\lambda}}v_2$\par
$\nabla_{v_3}v_1=\nabla_{v_4}v_1=\frac{-(1+\lambda)}{2\sqrt{2\lambda}}v_2$\par
$\nabla_{v_2}v_3=\nabla_{v_3}v_2=\frac{\lambda+1}{2\sqrt{2\lambda}}v_1$\par
$\nabla_{v_2}v_4=\nabla_{v_4}v_2=\frac{\lambda+1}{2\sqrt{2\lambda}}v_1$\par
$\nabla_{v_3}v_4=\nabla_{v_4}v_3=0$
\end{multicols}
Using the formula (\ref{leibnizrule}), we obtain the following equalities
\begin{multicols}{2}
$\nabla_{v_1} (J)v_1 = \nabla_{v_1} (J)v_3 = 0$\par
$\nabla_{v_1} (J)v_2 = \frac{\lambda - 1}{\sqrt{2\lambda}}v_4$\par
$\nabla_{v_1} (J)v_4 = \frac{1 - \lambda}{\sqrt{2\lambda}}v_2$\par
$\nabla_{v_2} (J)v_1 = \frac{1 + \lambda}{\sqrt{2\lambda}}v_4$\par
$\nabla_{v_2} (J)v_2 = \nabla_{v_2} (J)v_3 = 0$\par
$\nabla_{v_2} (J)v_4 = \frac{-\lambda - 1}{\sqrt{2\lambda}}v_1$\par
$\nabla_{v_3} (J)v_i = 0, \;\forall i= 1,..., 4$\par
$\nabla_{v_4} (J)v_i = 0, \;\forall i= 1,..., 4$
\end{multicols}

By (\ref{canonical}), the canonical connection $\nabla^0$ of $\left(\mathfrak{g}_4, S_1^{\lambda}\right)$ is given by
\begin{multicols}{3}
$\nabla^0_{v_i} v_i = 0, \;\forall i= 1,\ldots, 4$\par
$\nabla^0_{v_1} v_2 = \frac{1 - \lambda}{2\sqrt{2\lambda}}v_3$\par
$\nabla^0_{v_1} v_3 = \frac{\lambda - 1}{2\sqrt{2\lambda}}v_2$\par
$\nabla^0_{v_1} v_4 = 0$\par
$\nabla^0_{v_2} v_1 = -\frac{\lambda + 1}{2\sqrt{2\lambda}}v_3$\par
$\nabla^0_{v_2} v_3 = \frac{\lambda + 1}{2\sqrt{2\lambda}}v_1$\par
$\nabla^0_{v_2} v_4 = 0$\par
$\nabla^0_{v_3} v_1 = -\frac{1 + \lambda}{2\sqrt{2\lambda}}v_2$\par
$\nabla^0_{v_3} v_2 = \frac{\lambda + 1}{2\sqrt{2\lambda}}v_1$\par
$\nabla^0_{v_3} v_4 = 0$\par
$\nabla^0_{v_4} v_1 = -\frac{1 + \lambda}{2\sqrt{2\lambda}}v_2$\par
$\nabla^0_{v_4} v_2 = \frac{\lambda + 1}{2\sqrt{2\lambda}}v_1$\par
$\nabla^0_{v_4} v_3 = 0$
\end{multicols}

By (\ref{canonicalcurvature}), the curvature $R^0$ of the canonical connection $\nabla^0$ of $\left(\mathfrak{g}_4, S_1^{\lambda}\right)$ is described by
\begin{multicols}{3}
$R^0_{v_1v_2}v_1 = \frac{\lambda^2 - 1}{8\lambda}v_2$\par
$R^0_{v_1v_2}v_2 = \frac{1 - \lambda^2}{8\lambda}v_1$\par
$R^0_{v_1v_2}v_3 = 0$\par
$R^0_{v_1v_2}v_4 = 0$\par
$R^0_{v_1v_3}v_1 = \frac{-3\lambda^2 - 2\lambda + 1}{8\lambda}v_3$\par
$R^0_{v_1v_3}v_2 = 0$\par
$R^0_{v_1v_3}v_3 = \frac{3\lambda^2 + 2\lambda - 1}{8\lambda}v_1$\par
$R^0_{v_1v_3}v_4 = 0$\par
$R^0_{v_1v_4}v_1 = \frac{-3\lambda^2 - 2\lambda + 1}{8\lambda}v_3$\par
$R^0_{v_1v_4}v_2 = 0$\par
$R^0_{v_1v_4}v_3 = \frac{3\lambda^2 + 2\lambda - 1}{8\lambda}v_1$\par
$R^0_{v_1v_4}v_4 = 0$\par
$R^0_{v_2v_3}v_1 = 0$\par
$R^0_{v_2v_3}v_2 = \frac{\left(\lambda + 1\right)^2}{8\lambda}v_3$\par
$R^0_{v_2v_3}v_3 = -\frac{\left(\lambda + 1\right)^2}{8\lambda}v_2$\par
$R^0_{v_2v_3}v_4 = 0$\par
$R^0_{v_2v_4}v_1 = 0$\par
$R^0_{v_2v_4}v_2 = \frac{\left(\lambda + 1\right)^2}{8\lambda}v_3$\par
$R^0_{v_2v_4}v_3 = -\frac{\left(\lambda + 1\right)^2}{8\lambda}v_2$\par
$R^0_{v_2v_4}v_4 = 0$\par
$R^0_{v_3v_4}v_i = 0, i = 1, \ldots, 4$\par
\end{multicols}

Applying (\ref{canonicalriccioperators}) and the skew-symmetry of $R^0_{v_iv_j}$, we find that the Ricci operator $\operatorname{Ric}^0$ of $\nabla^0$ is represented in the basis $\mathscr{B}$ by the following matrix
$$\operatorname{Ric}^0 = \begin{bmatrix}
\frac{-\lambda - 1}{4} & 0 & 0 & 0\\
0 & \frac{\lambda + 1}{4} & 0 & 0\\
0 & 0 & \frac{1 - \lambda^2}{4\lambda} & \frac{1 - \lambda^2}{4\lambda}\\
0 & 0 & 0 & 0
\end{bmatrix}.$$
\begin{theorem}
$\left(\mathfrak{g}_4, S_1^{\lambda}, J\right)$ is a first kind algebraic Ricci soliton associated to the connection $\nabla^0$ if and only if $c = 0$ and $\lambda = \frac{1}{3}$. In particular, we have 
$$\operatorname{Ric}^0 = D = \begin{bmatrix}
\frac{-1}{3} & 0 & 0 & 0\\
0 & \frac{1}{3} & 0 & 0\\
0 & 0 & \frac{2}{3} & \frac{2}{3}\\
0 & 0 & 0 & 0
\end{bmatrix}.$$
\end{theorem}
\begin{proof}
Assume that $\left(\mathfrak{g}_4, S_1^{\lambda}, J\right)$ is a first kind algebraic Ricci soliton with respect to $\nabla^0$, then there exists a real number $c$ and a derivation $D$ of $\mathfrak{g}_4$ such that $\operatorname{Ric}^0 = cI_4 + D$. From this, it follows that 
$$D = \operatorname{Ric}^0 - cI_4 = \begin{bmatrix}
\frac{-\lambda - 1}{4} - c & 0 & 0 & 0\\
0 & \frac{\lambda + 1}{4} - c & 0 & 0\\
0 & 0 & \frac{1 - \lambda^2}{4\lambda} - c & \frac{1 - \lambda^2}{4\lambda}\\
0 & 0 & 0 & -c
\end{bmatrix}.$$
Since $D$ is a derivation of  $\mathfrak{g}_4$, the condition below holds
\begin{eqnarray*}
& [D(v_1), v_3] + [v_1, D(v_3)] = D[v_1, v_3] \\ \Leftrightarrow & \left[ \left(\frac{-\lambda - 1}{4} - c\right)v_1, v_3\right] + \left[ v_1, \left(\frac{1 - \lambda^2}{4\lambda} - c\right)v_3\right] = \sqrt{\frac{\lambda}{2}}D(v_2)\\
\Leftrightarrow& \left( \frac{-2\lambda^2 - \lambda + 1}{4\lambda} - 2c\right)[v_1, v_3] = \sqrt{\frac{\lambda}{2}}D(v_2) \\
\Leftrightarrow& \left( \frac{-2\lambda^2 - \lambda + 1}{4\lambda} - 2c\right)\sqrt{\frac{\lambda}{2}}v_2 = \left(\frac{\lambda + 1}{4} - c\right)\sqrt{\frac{\lambda}{2}}v_2\\
\Leftrightarrow& \frac{-2\lambda^2 - \lambda + 1}{4\lambda} - 2c = \frac{\lambda + 1}{4} - c \Leftrightarrow c = \frac{-3\lambda^2 - 2\lambda + 1}{4\lambda}
\end{eqnarray*}
The condition $[D(v_1), v_4] + [v_1, D(v_4)] = D[v_1, v_4]$ is similar to the previous one. Also, $D$ satisfies the following condition
\begin{eqnarray*}
& [D(v_1), v_2] + [v_1, D(v_2)] = D[v_1, v_2] \\ \Leftrightarrow & \left[ \left(\frac{-\lambda - 1}{4} - c\right)v_1, v_2\right] + \left[ v_1, \left(\frac{\lambda + 1}{4} - c\right)v_2\right] = \frac{1}{\sqrt{2\lambda}}\left(D(v_3) - D(v_4)\right)\\
\Leftrightarrow& -2c[v_1, v_2] = \frac{1}{\sqrt{2\lambda}}\left(D(v_3) - D(v_4)\right) \\
\Leftrightarrow& \frac{-2c}{\sqrt{2\lambda}}\left(v_3 - v_4\right) = \frac{1}{\sqrt{2\lambda}}\left(-cv_3 + cv_4\right)
\end{eqnarray*}
Hence $2c = c$ and then $c = 0$. This implies that $-3\lambda^2 - 2\lambda + 1 = 0$; the two solutions are $\lambda = \frac{1}{3}$ and $\lambda = -1$. Since $\lambda > 0$, then $\lambda = \frac{1}{3}$. Thus
$$\operatorname{Ric}^0 = D = \begin{bmatrix}
\frac{-1}{3} & 0 & 0 & 0\\
0 & \frac{1}{3} & 0 & 0\\
0 & 0 & \frac{2}{3} & \frac{2}{3}\\
0 & 0 & 0 & 0
\end{bmatrix}.$$
One can easily check that $D$ satisfies the condition
$$[D(v_i), v_j] + [v_i, D(v_j)] = D[v_i, v_j]$$
for all vanishing brackets.
\end{proof}
According to (\ref{tilderho}) and (\ref{tildeRic}), we find that $\widetilde{\operatorname{Ric}}^0$ is represented in the basis $\mathscr{B}$ by the following matrix
$$\widetilde{\operatorname{Ric}}^0 = \begin{bmatrix}
\frac{-\lambda - 1}{4} & 0 & 0 & 0\\
0 & \frac{\lambda + 1}{4} & 0 & 0\\
0 & 0 & \frac{1 - \lambda^2}{4\lambda} & \frac{1 - \lambda^2}{8\lambda}\\
0 & 0 & \frac{\lambda^2 - 1}{8\lambda} & 0
\end{bmatrix}.$$
\begin{theorem}
$\left(\mathfrak{g}_4, S_1^{\lambda}, J\right)$ is a second kind algebraic Ricci soliton associated to the connection $\nabla^0$ if and only if $c = \frac{-1}{3}$ and $\lambda = \frac{1}{3}$. In particular, we have 
$$\widetilde{\operatorname{Ric}}^0 = \begin{bmatrix}
\frac{-1}{3} & 0 & 0 & 0\\
0 & \frac{1}{3} & 0 & 0\\
0 & 0 & \frac{2}{3} & \frac{1}{3}\\
0 & 0 & \frac{-1}{3} & 0
\end{bmatrix}, \qquad D = \begin{bmatrix}
0 & 0 & 0 & 0\\
0 & \frac{2}{3} & 0 & 0\\
0 & 0 & 1 & \frac{1}{3}\\
0 & 0 & \frac{-1}{3} & \frac{1}{3}
\end{bmatrix}.$$
\end{theorem}
\begin{proof}
Suppose that $\left(\mathfrak{g}_4, S_1^{\lambda}, J\right)$ is a second kind algebraic Ricci soliton with respect to $\nabla^0$, then there exists a real number $c$ and a derivation $D$ of $\mathfrak{g}_4$ such that $\widetilde{\operatorname{Ric}}^0 = cI_4 + D$. Thus, we obtain that that 
$$D = \widetilde{\operatorname{Ric}}^0 - cI_4 = \begin{bmatrix}
\frac{-\lambda - 1}{4} - c & 0 & 0 & 0\\
0 & \frac{\lambda + 1}{4} - c & 0 & 0\\
0 & 0 & \frac{1 - \lambda^2}{4\lambda} - c & \frac{1 - \lambda^2}{8\lambda}\\
0 & 0 & \frac{\lambda^2 - 1}{8\lambda} & - c
\end{bmatrix}.$$
As $D$ is a derivation of  $\mathfrak{g}_4$, it satisfies the condition
\begin{eqnarray*}
& [D(v_1), v_3] + [v_1, D(v_3)] = D[v_1, v_3] \\ \Leftrightarrow & \left[ \left(\frac{-\lambda - 1}{4} - c\right)v_1, v_3\right] + \left[ v_1, \left(\frac{1 - \lambda^2}{4\lambda} - c\right)v_3 + \frac{\lambda^2 - 1}{8\lambda}v_4\right] = \sqrt{\frac{\lambda}{2}}D(v_2)\\
\Leftrightarrow& \left( \frac{-3\lambda^2 - 2\lambda + 1}{8\lambda} - 2c\right)\sqrt{\frac{\lambda}{2}}v_2 = \left(\frac{\lambda^2 + \lambda}{4\lambda} - c\right)\sqrt{\frac{\lambda}{2}}v_2\\
\Leftrightarrow& c = \frac{-5\lambda^2 - 4\lambda + 1}{8\lambda}
\end{eqnarray*}
The condition $[D(v_1), v_4] + [v_1, D(v_4)] = D[v_1, v_4]$ is similar to the previous one. Also, $D$ satisfies the condition
\begin{eqnarray*}
& [D(v_1), v_2] + [v_1, D(v_2)] = D[v_1, v_2] \\ \Leftrightarrow & \left[ \left(\frac{-\lambda - 1}{4} - c\right)v_1, v_2\right] + \left[ v_1, \left(\frac{\lambda + 1}{4} - c\right)v_2\right] = \frac{1}{\sqrt{2\lambda}}\left(D(v_3) - D(v_4)\right)\\
\Leftrightarrow& -2c[v_1, v_2] = \frac{1}{\sqrt{2\lambda}}\left(D(v_3) - D(v_4)\right) \\
\Leftrightarrow& \frac{-2c}{\sqrt{2\lambda}}\left(v_3 - v_4\right) = \frac{1}{\sqrt{2\lambda}}\left(\left(\frac{1 - \lambda^2}{8\lambda} - c\right)v_3 + \left(\frac{\lambda^2 - 1}{8\lambda} + c\right)v_4\right)
\end{eqnarray*}
By identifying the components of $v_3$ and $v_4$ we obtain
$$-2c = \frac{1 - \lambda^2}{8\lambda} - c \Longrightarrow c = \frac{\lambda^2 - 1}{8\lambda}$$
Hence $-5\lambda^2 - 4\lambda + 1 = \lambda^2 - 1$, which implies that $-3\lambda^2 - 2\lambda + 1 = 0$. Thus 
$$\lambda = \frac{1}{3}, \qquad c = \frac{-1}{3}.$$
Under these conditions we have
$$\widetilde{\operatorname{Ric}}^0 = \begin{bmatrix}
\frac{-1}{3} & 0 & 0 & 0\\
0 & \frac{1}{3} & 0 & 0\\
0 & 0 & \frac{2}{3} & \frac{1}{3}\\
0 & 0 & \frac{-1}{3} & 0
\end{bmatrix}, \qquad D = \begin{bmatrix}
0 & 0 & 0 & 0\\
0 & \frac{2}{3} & 0 & 0\\
0 & 0 & 1 & \frac{1}{3}\\
0 & 0 & \frac{-1}{3} & \frac{1}{3}
\end{bmatrix}.$$
It is straightforward to check that $D$ satisfies the condition
$$[D(v_i), v_j] + [v_i, D(v_j)] = D[v_i, v_j]$$
for all vanishing brackets.
\end{proof}
By (\ref{kobayashinomizu}), we obtain that the Kobayashi-Nomizu connection $\nabla^1$ of $(\mathfrak{g}_4, S_1^{\lambda}, J)$ is given by
\begin{multicols}{3}
$\nabla^1_{v_i} v_i = 0 \,\,\forall i= 1,\ldots, 4$\par
$\nabla^1_{v_1} v_2 = \frac{1 - \lambda}{2\sqrt{2\lambda}}v_3$\par
$\nabla^1_{v_1} v_3 = \frac{\lambda - 1}{2\sqrt{2\lambda}}v_2$\par
$\nabla^1_{v_1} v_4 = 0$\par
$\nabla^1_{v_2} v_1 = -\frac{\lambda + 1}{2\sqrt{2\lambda}}v_3$\par
$\nabla^1_{v_2} v_3 = \frac{\lambda + 1}{2\sqrt{2\lambda}}v_1$\par
$\nabla^1_{v_2} v_4 = 0$\par
$\nabla^1_{v_3} v_1 = -\frac{1 + \lambda}{2\sqrt{2\lambda}}v_2$\par
$\nabla^1_{v_3} v_2 = \frac{\lambda + 1}{2\sqrt{2\lambda}}v_1$\par
$\nabla^1_{v_3} v_4 = 0$\par
$\nabla^1_{v_4} v_1 = \frac{-\lambda}{\sqrt{2\lambda}}v_2$\par
$\nabla^1_{v_4} v_2 = 0$\par
$\nabla^1_{v_4} v_3 = 0$
\end{multicols}

Applying (\ref{kobayashicurvature}), we find that the curvature $R^1$ of the Kobayashi-Nomizu connection $\nabla^1$ of $(\mathfrak{g}_4, S_1^{\lambda}, J)$ is described by
\begin{multicols}{3}
$R^1_{v_1v_2}v_1 = \frac{\lambda^2 + 2\lambda - 3}{8\lambda}v_2$\par
$R^1_{v_1v_2}v_2 = \frac{-\lambda^2 + 2\lambda + 3}{8\lambda}v_1$\par
$R^1_{v_1v_2}v_3 = 0$\par
$R^1_{v_1v_2}v_4 = 0$\par
$R^1_{v_1v_3}v_1 = \frac{-3\lambda^2 - 2\lambda + 1}{8\lambda}v_3$\par
$R^1_{v_1v_3}v_2 = 0$\par
$R^1_{v_1v_3}v_3 = \frac{3\lambda^2 + 2\lambda - 1}{8\lambda}v_1$\par
$R^1_{v_1v_3}v_4 = 0$\par
$R^1_{v_1v_4}v_1 = \frac{-\lambda}{2}v_3$\par
$R^1_{v_1v_4}v_2 = 0$\par
$R^1_{v_1v_4}v_3 = \frac{\lambda + 1}{4}v_1$\par
$R^1_{v_1v_4}v_4 = 0$\par
$R^1_{v_2v_3}v_1 = 0$\par
$R^1_{v_2v_3}v_2 = \frac{\left(\lambda + 1\right)^2}{8\lambda}v_3$\par
$R^1_{v_2v_3}v_3 = -\frac{\left(\lambda + 1\right)^2}{8\lambda}v_2$\par
$R^1_{v_2v_3}v_4 = 0$\par
$R^1_{v_2v_4}v_1 = 0$\par
$R^1_{v_2v_4}v_2 = 0$\par
$R^1_{v_2v_4}v_3 = \frac{-\lambda - 1}{4}v_2$\par
$R^1_{v_2v_4}v_4 = 0$\par
$R^1_{v_3v_4}v_1 = \frac{\lambda + 1}{4}v_1$\par
$R^1_{v_3v_4}v_2 = \frac{-\lambda - 1}{4}v_2$\par
$R^1_{v_3v_4}v_3 = R^1_{v_3v_4}v_4 = 0 $\par
\end{multicols}

In view of (\ref{canonicalriccioperators}), the Ricci operator associated to $\nabla^1$ is represented in the basis $\mathscr{B}$ by 
$$\operatorname{Ric}^1 = \begin{bmatrix}
\frac{-(\lambda + 1)^2}{4\lambda} & 0 & 0 & 0\\
\\
0 & \frac{\lambda^2 + 2\lambda - 1}{4\lambda} & 0 & 0\\
\\
0 & 0 & \frac{1 - \lambda^2}{4\lambda} & \frac{-\lambda}{2}\\
\\
0 & 0 & 0 & 0
\end{bmatrix}.$$
\begin{theorem}
$\left(\mathfrak{g}_4, S_1^{\lambda}, J\right)$ does not exhibit the properties of a first kind algebraic Ricci soliton associated to the connection $\nabla^1$.
\end{theorem}
\begin{proof}
Assume that $\left(\mathfrak{g}_4, S_1^{\lambda}, J\right)$ is a first kind algebraic Ricci soliton associated to $\nabla^1$, then $\operatorname{Ric}^1 = cI_4 + D$, where $c \in \mathbb{R}$ and $D$ is a derivation of $\mathfrak{g}_4$. Hence
$$D = \operatorname{Ric}^1 - cI_4 = \begin{bmatrix}
\frac{-(\lambda + 1)^2}{4\lambda} - c & 0 & 0 & 0\\
\\
0 & \frac{\lambda^2 + 2\lambda - 1}{4\lambda} - c & 0 & 0\\
\\
0 & 0 & \frac{1 - \lambda^2}{4\lambda} - c & \frac{-\lambda}{2}\\
\\
0 & 0 & 0 & -c
\end{bmatrix}.$$
Given that $D$ is a derivation of  $\mathfrak{g}_4$, the condition below is satisfied
\begin{eqnarray*}
& [D(v_1), v_3] + [v_1, D(v_3)] = D[v_1, v_3] \\ \Leftrightarrow & \left[ \left(\frac{-(\lambda + 1)^2}{4\lambda} - c\right)v_1, v_3\right] + \left[ v_1, \left(\frac{1 - \lambda^2}{4\lambda} - c\right)v_3\right] = \sqrt{\frac{\lambda}{2}}D(v_2)\\
\Leftrightarrow& \left( \frac{-2\lambda^2 - 2\lambda}{4\lambda} - 2c\right)\sqrt{\frac{\lambda}{2}}v_2 = \sqrt{\frac{\lambda}{2}}\left(\frac{\lambda^2 + 2\lambda - 1}{4\lambda} - c\right)v_2\\
\Leftrightarrow& \frac{-2\lambda^2 - 2\lambda}{4\lambda} - 2c = \frac{\lambda^2 + 2\lambda - 1}{4\lambda} - c \Leftrightarrow c = \frac{-3\lambda^2 - 4\lambda + 1}{4\lambda}
\end{eqnarray*}
Similarly, the condition below is satisfied
\begin{eqnarray*}
& [D(v_1), v_4] + [v_1, D(v_4)] = D[v_1, v_4] \\ \Leftrightarrow & \left[ \left(\frac{-(\lambda + 1)^2}{4\lambda} - c\right)v_1, v_4\right] + \left[ v_1, \frac{-\lambda}{2}v_3 - cv_4\right] = \sqrt{\frac{\lambda}{2}}D(v_2)\\
\Leftrightarrow& \left( \frac{-3\lambda^2 - 2\lambda - 1}{4\lambda} - 2c\right)\sqrt{\frac{\lambda}{2}}v_2 = \sqrt{\frac{\lambda}{2}}\left(\frac{\lambda^2 + 2\lambda - 1}{4\lambda} - c\right)v_2\\
\Leftrightarrow& \frac{-3\lambda^2 - 2\lambda - 1}{4\lambda} - 2c = \frac{\lambda^2 + 2\lambda - 1}{4\lambda} - c \Leftrightarrow c = \frac{-4\lambda^2 - 4\lambda}{4\lambda}
\end{eqnarray*}
Hence 
$$c = \frac{-3\lambda^2 - 4\lambda + 1}{4\lambda} = \frac{-4\lambda^2 - 4\lambda}{4\lambda} \Leftrightarrow -3\lambda^2 - 4\lambda + 1 = -4\lambda^2 - 4\lambda \Leftrightarrow \lambda^2 + 1 = 0,$$
this is a contradiction.
\end{proof}

From (\ref{tilderho}) and (\ref{tildeRic}), we have that $\widetilde{\operatorname{Ric}}^1$ is represented in the basis $\mathscr{B}$ by 
$$\widetilde{\operatorname{Ric}}^1 = \begin{bmatrix}
\frac{-(\lambda + 1)^2}{4\lambda} & 0 & 0 & 0\\
\\
0 & \frac{\lambda^2 + 2\lambda - 1}{4\lambda} & 0 & 0\\
\\
0 & 0 & \frac{1 - \lambda^2}{4\lambda} & \frac{-\lambda}{4}\\
\\
0 & 0 & \frac{\lambda}{4} & 0
\end{bmatrix}.$$
\begin{theorem}
$\left(\mathfrak{g}_4, S_1^{\lambda}, J\right)$ does not exhibit the properties of a second kind algebraic Ricci soliton associated to the connection $\nabla^1$.
\end{theorem}
\begin{proof}
Suppose that $\left(\mathfrak{g}_4, S_1^{\lambda}, J\right)$ is a second kind algebraic Ricci soliton associated to $\nabla^1$, then $\widetilde{\operatorname{Ric}}^1 = cI_4 + D$, where $c \in \mathbb{R}$ and $D$ is a derivation of $\mathfrak{g}_4$. From this, it follows that 
$$D = \widetilde{\operatorname{Ric}}^1 - cI_4 = \begin{bmatrix}
\frac{-(\lambda + 1)^2}{4\lambda} - c & 0 & 0 & 0\\
\\
0 & \frac{\lambda^2 + 2\lambda - 1}{4\lambda} - c & 0 & 0\\
\\
0 & 0 & \frac{1 - \lambda^2}{4\lambda} - c & \frac{-\lambda}{4}\\
\\
0 & 0 & \frac{\lambda}{4} & -c
\end{bmatrix}.$$
Since $D$ is a derivation of  $\mathfrak{g}_4$, it satisfies the condition
\begin{eqnarray*}
& [D(v_1), v_3] + [v_1, D(v_3)] = D[v_1, v_3] \\ \Leftrightarrow & \left[ \left(\frac{-(\lambda + 1)^2}{4\lambda} - c\right)v_1, v_3\right] + \left[ v_1, \left(\frac{1 - \lambda^2}{4\lambda} - c\right)v_3 + \frac{\lambda}{4}v_4\right] = \sqrt{\frac{\lambda}{2}}D(v_2)\\
\Leftrightarrow& \left( \frac{-\lambda^2 - 2\lambda}{4\lambda} - 2c\right)\sqrt{\frac{\lambda}{2}}v_2 = \sqrt{\frac{\lambda}{2}}\left(\frac{\lambda^2 + 2\lambda - 1}{4\lambda} - c\right)v_2\\
\Leftrightarrow& \frac{-\lambda^2 - 2\lambda}{4\lambda} - 2c = \frac{\lambda^2 + 2\lambda - 1}{4\lambda} - c \Leftrightarrow c = \frac{-2\lambda^2 - 4\lambda + 1}{4\lambda}
\end{eqnarray*}
Similarly, $D$ satisfies the condition
\begin{eqnarray*}
& [D(v_1), v_4] + [v_1, D(v_4)] = D[v_1, v_4] \\ \Leftrightarrow & \left[ \left(\frac{-(\lambda + 1)^2}{4\lambda} - c\right)v_1, v_4\right] + \left[ v_1, \frac{-\lambda}{4}v_3 - cv_4\right] = \sqrt{\frac{\lambda}{2}}D(v_2)\\
\Leftrightarrow& \left( \frac{-2\lambda^2 - 2\lambda - 1}{4\lambda} - 2c\right)\sqrt{\frac{\lambda}{2}}v_2 = \sqrt{\frac{\lambda}{2}}\left(\frac{\lambda^2 + 2\lambda - 1}{4\lambda} - c\right)v_2\\
\Leftrightarrow& \frac{-2\lambda^2 - 2\lambda - 1}{4\lambda} - 2c = \frac{\lambda^2 + 2\lambda - 1}{4\lambda} - c \Leftrightarrow c = \frac{-3\lambda^2 - 4\lambda}{4\lambda}
\end{eqnarray*}
Hence 
$$c = \frac{-2\lambda^2 - 4\lambda + 1}{4\lambda} = \frac{-3\lambda^2 - 4\lambda}{4\lambda} \Leftrightarrow -2\lambda^2 - 4\lambda + 1 = -3\lambda^2 - 4\lambda \Leftrightarrow \lambda^2 + 1 = 0,$$
this is a contradiction.
\end{proof}
\subsection{The Lorentzian Lie algebra $\left(\mathfrak{g}_4, S_2^{\lambda}\right)$}
Consider the following elements
$$v_1 = e_2, \qquad v_2 = \frac{1}{\sqrt{\lambda}}e_3, \qquad v_3 = \frac{1}{\sqrt{2}}(e_1 + e_4), \qquad v_4 = \frac{1}{\sqrt{2}}(e_1 - e_4).$$
Then $\mathscr{B} = \left\lbrace v_1, v_2, v_3, v_4\right\rbrace$ is an orthonormal basis of $\left(\mathfrak{g}_4, S_2^{\lambda}\right)$. The nonzero commutators in the basis $\mathscr{B}$ are
\begin{align*}
[v_1, v_3] &= -\sqrt{\frac{\lambda}{2}}v_2,  &\quad [v_1, v_4] &= -\sqrt{\frac{\lambda}{2}}v_2, \\
[v_2, v_3] &= \frac{-1}{2\sqrt{\lambda}}(v_3 - v_4), &\quad [v_2, v_4] &= \frac{-1}{2\sqrt{\lambda}}(v_3 - v_4).
\end{align*}
By \cite{aitbenhaddou2026lorentzian}, the Levi-Civita connection of $\left(\mathfrak{g}_4, S_2^{\lambda}\right)$ is given by
\begin{multicols}{2}
$\nabla_{v_1} v_1 = \nabla_{v_2} v_2 = 0 $\par
$\nabla_{v_3} v_3 = \nabla_{v_4} v_4 = \frac{-1}{2\sqrt{\lambda}}v_2 $\par
$\nabla_{v_1} v_2 = \nabla_{v_2} v_1 = \frac{1}{2}\sqrt{\frac{\lambda}{2}}v_3 - \frac{1}{2}\sqrt{\frac{\lambda}{2}}v_4$\par
$\nabla_{v_1} v_3 = \nabla_{v_1} v_4 = \frac{-1}{2}\sqrt{\frac{\lambda}{2}}v_2$\par
$\nabla_{v_3} v_1 = \nabla_{v_4} v_1 = \frac{1}{2}\sqrt{\frac{\lambda}{2}}v_2$\par
$\nabla_{v_2} v_3 = \nabla_{v_2} v_4 = \frac{-1}{2}\sqrt{\frac{\lambda}{2}}v_1$\par
$\nabla_{v_3} v_2 = \frac{-1}{2}\sqrt{\frac{\lambda}{2}}v_1 + \frac{1}{2\sqrt{\lambda}}(v_ 3 - v_4)$\par
$\nabla_{v_4} v_2 = \frac{-1}{2}\sqrt{\frac{\lambda}{2}}v_1 + \frac{1}{2\sqrt{\lambda}}(v_ 3 - v_4)$\par
$\nabla_{v_3} v_4 = \nabla_{v_4} v_3 = \frac{-1}{2\sqrt{\lambda}}v_2$
\end{multicols}
Applying (\ref{leibnizrule}), we find the following formulas
\begin{multicols}{3}
$\nabla_{v_1} (J)v_1 = 0$\par
$\nabla_{v_1} (J)v_2 = -\sqrt{\frac{\lambda}{2}}v_4$\par
$\nabla_{v_1} (J)v_3 = 0$\par
$\nabla_{v_1} (J)v_4 = \sqrt{\frac{\lambda}{2}}v_2$\par
$\nabla_{v_2} (J)v_1 = -\sqrt{\frac{\lambda}{2}}v_4$\par
$\nabla_{v_2} (J)v_2 = 0$\par
$\nabla_{v_2} (J)v_3 = 0$\par
$\nabla_{v_2} (J)v_4 = \sqrt{\frac{\lambda}{2}}v_1$\par
$\nabla_{v_3} (J)v_1 = 0$\par
$\nabla_{v_3} (J)v_2 = \frac{-1}{\sqrt{\lambda}}v_4$\par
$\nabla_{v_3} (J)v_3 = 0$\par
$\nabla_{v_3} (J)v_4 = \frac{1}{\sqrt{\lambda}}v_2$\par
$\nabla_{v_4} (J)v_1 = 0$\par
$\nabla_{v_4} (J)v_2 = \frac{-1}{\sqrt{\lambda}}v_4$\par
$\nabla_{v_4} (J)v_3 = 0$\par
$\nabla_{v_4} (J)v_4 = \frac{1}{\sqrt{\lambda}}v_2$\par
\end{multicols}

Using (\ref{canonical}), we obtain that the canonical connection $\nabla^0$ of $\left(\mathfrak{g}_4, S_2^{\lambda}\right)$ is given by
\begin{multicols}{3}
$\nabla^0_{v_1} v_1 = 0$\par
$\nabla^0_{v_1} v_2 = \frac{1}{2}\sqrt{\frac{\lambda}{2}}v_3$\par
$\nabla^0_{v_1} v_3 = \frac{-1}{2}\sqrt{\frac{\lambda}{2}}v_2$\par
$\nabla^0_{v_1} v_4 = 0$\par
$\nabla^0_{v_2} v_1 = \frac{1}{2}\sqrt{\frac{\lambda}{2}}v_3$\par
$\nabla^0_{v_2} v_2 = 0$\par
$\nabla^0_{v_2} v_3 = \frac{-1}{2}\sqrt{\frac{\lambda}{2}}v_1$\par
$\nabla^0_{v_2} v_4 = 0$\par
$\nabla^0_{v_3} v_1 = \frac{1}{2}\sqrt{\frac{\lambda}{2}}v_2$\par
$\nabla^0_{v_3} v_2 = \frac{-1}{2}\sqrt{\frac{\lambda}{2}}v_1 + \frac{1}{2\sqrt{\lambda}}v_3$\par
$\nabla^0_{v_3} v_3 = \frac{-1}{2\sqrt{\lambda}}v_2$\par
$\nabla^0_{v_3} v_4 = 0$\par
$\nabla^0_{v_4} v_1 = \frac{1}{2}\sqrt{\frac{\lambda}{2}}v_2$\par
$\nabla^0_{v_4} v_2 = \frac{-1}{2}\sqrt{\frac{\lambda}{2}}v_1 + \frac{1}{2\sqrt{\lambda}}v_3$\par
$\nabla^0_{v_4} v_3 = \frac{-1}{2\sqrt{\lambda}}v_2$\par
$\nabla^0_{v_4} v_4 = 0$
\end{multicols}

According to (\ref{canonicalcurvature}), the curvature $R^0$ of $\nabla^0$ is described by
\begin{multicols}{3}
$R^0_{v_1v_2}v_1 = \frac{\lambda}{8}v_2$\par
$R^0_{v_1v_2}v_2 = \frac{-\lambda}{8}v_1$\par
$R^0_{v_1v_2}v_3 = 0$\par
$R^0_{v_1v_2}v_4 = 0$\par
$R^0_{v_1v_3}v_1 = \frac{-3\lambda}{8}v_3$\par
$R^0_{v_1v_3}v_2 = 0$\par
$R^0_{v_1v_3}v_3 = \frac{3\lambda}{8}v_1$\par
$R^0_{v_1v_3}v_4 = 0$\par
$R^0_{v_1v_4}v_1 = \frac{-3\lambda}{8}v_3$\par
$R^0_{v_1v_4}v_2 = 0$\par
$R^0_{v_1v_4}v_3 = \frac{3\lambda}{8}v_1$\par
$R^0_{v_1v_4}v_4 = 0$\par
$R^0_{v_2v_3}v_1 = \frac{-1}{4\sqrt{2}}v_2$\par
$R^0_{v_2v_3}v_2 = \frac{\lambda}{8}v_3 + \frac{1}{4\sqrt{2}}v_1$\par
$R^0_{v_2v_3}v_3 = \frac{-\lambda}{8}v_2$\par
$R^0_{v_2v_3}v_4 = 0$\par
$R^0_{v_2v_4}v_1 = \frac{-1}{4\sqrt{2}}v_2$\par
$R^0_{v_2v_4}v_2 = \frac{\lambda}{8}v_3 + \frac{1}{4\sqrt{2}}v_1$\par
$R^0_{v_2v_4}v_3 = \frac{-\lambda}{8}v_2$\par
$R^0_{v_2v_4}v_4 = 0$\par
$R^0_{v_3v_4}v_i = 0, i = 1, \ldots, 4$\par
\end{multicols}
By (\ref{canonicalriccioperators}), we find that the Ricci operator $\operatorname{Ric}^0$ of $\nabla^0$ is represented in the basis $\mathscr{B}$ by the following matrix
$$\operatorname{Ric}^0 = \begin{bmatrix}
\frac{-\lambda}{4} & 0 & \frac{1}{4\sqrt{2}} & \frac{1}{4\sqrt{2}}\\
0 & \frac{\lambda}{4} & 0 & 0\\
0 & 0 & \frac{-\lambda}{4} & \frac{-\lambda}{4}\\
0 & 0 & 0 & 0
\end{bmatrix}.$$
\begin{theorem}
$\left(\mathfrak{g}_4, S_2^{\lambda}, J\right)$ does not exhibit the properties of a first kind algebraic Ricci soliton associated to the connection $\nabla^0$.
\end{theorem}
\begin{proof}
Assume that $\left(\mathfrak{g}_4, S_2^{\lambda}, J\right)$ is a first kind algebraic Ricci soliton, then there exists a real number $c$ and a derivation $D$ of $\mathfrak{g}_4$ such that $\operatorname{Ric}^0 = cI_4 + D$. From this, it follows that 
$$D = \operatorname{Ric}^0 - cI_4 = \begin{bmatrix}
\frac{-\lambda}{4} - c & 0 & \frac{1}{4\sqrt{2}} & \frac{1}{4\sqrt{2}}\\
0 & \frac{\lambda}{4} - c & 0 & 0\\
0 & 0 & \frac{-\lambda}{4} - c & \frac{-\lambda}{4}\\
0 & 0 & 0 & - c
\end{bmatrix}.$$
From the fact that $D$ is a derivation of  $\mathfrak{g}_4$, it follows that the following condition holds
\begin{eqnarray*}
& [D(v_1), v_3] + [v_1, D(v_3)] = D[v_1, v_3] \\ \Leftrightarrow & \left[ \left(\frac{-\lambda}{4} - c\right)v_1, v_3\right] + \left[ v_1, \frac{1}{4\sqrt{2}}v_1 + \left(\frac{-\lambda}{4} - c\right)v_3\right] = -\sqrt{\frac{\lambda}{2}}D(v_2)\\
\Leftrightarrow& \left( \frac{-\lambda}{2} - 2c\right)[v_1, v_3] = -\sqrt{\frac{\lambda}{2}}D(v_2) \\
\Leftrightarrow& \left( \frac{-\lambda}{2} - 2c\right)\left(-\sqrt{\frac{\lambda}{2}}\right)v_2 = -\sqrt{\frac{\lambda}{2}}\left(\frac{\lambda}{4} - c\right)v_2 \\
\Leftrightarrow& \frac{-\lambda}{2} - 2c = \frac{\lambda}{4} - c \Leftrightarrow c = \frac{-3\lambda}{4}
\end{eqnarray*}
Also, $D$ satisfies the following condition
\begin{eqnarray*}
& [D(v_2), v_3] + [v_2, D(v_3)] = D[v_2, v_3] \\ \Leftrightarrow & \left[ \left(\frac{\lambda}{4} - c\right)v_2, v_3\right] + \left[ v_2, \frac{1}{4\sqrt{2}}v_1 + \left(\frac{-\lambda}{4} - c\right)v_3\right] = \frac{-1}{2\sqrt{\lambda}}\left(D(v_3) - D(v_4)\right)\\
\Leftrightarrow& -2c[v_2, v_3] = \frac{-1}{2\sqrt{\lambda}}\left(D(v_3) - D(v_4)\right) \\
\Leftrightarrow& \frac{2c}{2\sqrt{\lambda}}\left(v_3 - v_4\right) = \frac{c}{2\sqrt{\lambda}}\left(v_3 - v_4\right)
\end{eqnarray*}
Hence $2c = c$ and then $c = 0$. Thus, $c = \frac{-3\lambda}{4} = 0$, which is a contradiction.
\end{proof}
In view of (\ref{tilderho}) and (\ref{tildeRic}), we find that $\widetilde{\operatorname{Ric}}^0$ is represented in the basis $\mathscr{B}$ by the following matrix
$$\widetilde{\operatorname{Ric}}^0 = \begin{bmatrix}
\frac{-\lambda}{4} & 0 & \frac{1}{8\sqrt{2}} & \frac{1}{8\sqrt{2}}\\
0 & \frac{\lambda}{4} & 0 & 0\\
\frac{1}{8\sqrt{2}} & 0 & \frac{-\lambda}{4} & \frac{-\lambda}{8}\\
\frac{-1}{8\sqrt{2}} & 0 & \frac{\lambda}{8} & 0
\end{bmatrix}.$$
\begin{theorem}
$\left(\mathfrak{g}_4, S_2^{\lambda}, J\right)$ does not exhibit the properties of a second kind algebraic Ricci soliton associated to the connection $\nabla^0$.
\end{theorem}
\begin{proof}
Assume that $\left(\mathfrak{g}_4, S_2^{\lambda}, J\right)$ is a second kind algebraic Ricci soliton, this signifies that $\widetilde{\operatorname{Ric}}^0 = cI_4 + D$, where $c \in \mathbb{R}$ and $D$ is a derivation of $\mathfrak{g}_4$. From this, it follows that 
$$D = \widetilde{\operatorname{Ric}}^0 - cI_4 = \begin{bmatrix}
\frac{-\lambda}{4} - c & 0 & \frac{1}{8\sqrt{2}} & \frac{1}{8\sqrt{2}}\\
0 & \frac{\lambda}{4} - c & 0 & 0\\
\frac{1}{8\sqrt{2}} & 0 & \frac{-\lambda}{4} - c & \frac{-\lambda}{8}\\
\frac{-1}{8\sqrt{2}} & 0 & \frac{\lambda}{8} & - c
\end{bmatrix}.$$
Since $D$ is a derivation of  $\mathfrak{g}_4$, the following condition is satisfied
\begin{eqnarray*}
& [D(v_1), v_3] + [v_1, D(v_3)] = D[v_1, v_3] \\ \Leftrightarrow & \left[ \left(\frac{-\lambda}{4} - c\right)v_1 + \frac{1}{8\sqrt{2}}v_3 - \frac{1}{8\sqrt{2}}v_4, v_3\right] + \left[ v_1, \frac{1}{8\sqrt{2}}v_1 + \left(\frac{-\lambda}{4} - c\right)v_3 + \frac{\lambda}{8}v_4\right] = -\sqrt{\frac{\lambda}{2}}D(v_2)\\
\Leftrightarrow& \left( \frac{-\lambda}{2} - 2c\right)[v_1, v_3] + \frac{\lambda}{8}[v_1, v_4] = -\sqrt{\frac{\lambda}{2}}D(v_2) \\
\Leftrightarrow& \left( \frac{-3\lambda}{8} - 2c\right)\left(-\sqrt{\frac{\lambda}{2}}\right)v_2 = -\sqrt{\frac{\lambda}{2}}\left(\frac{\lambda}{4} - c\right)v_2 \\
\Leftrightarrow& \frac{-3\lambda}{8} - 2c = \frac{\lambda}{4} - c \Leftrightarrow c = \frac{-5\lambda}{8}
\end{eqnarray*}
Again, the condition below is satisfied
\begin{eqnarray*}
& [D(v_2), v_3] + [v_2, D(v_3)] = D[v_2, v_3] \\ \Leftrightarrow & \left[ \left(\frac{\lambda}{4} - c\right)v_2, v_3\right] + \left[ v_2, \frac{1}{8\sqrt{2}}v_1 + \left(\frac{-\lambda}{4} - c\right)v_3 + \frac{\lambda}{8}v_4\right] = \frac{-1}{2\sqrt{\lambda}}\left(D(v_3) - D(v_4)\right)\\
\Leftrightarrow& -2c[v_2, v_3] + \frac{\lambda}{8}[v_2, v_4] = \frac{-1}{2\sqrt{\lambda}}\left(D(v_3) - D(v_4)\right) \\
\Leftrightarrow& \left(\frac{\lambda}{8} - 2c\right)\frac{-1}{2\sqrt{\lambda}}\left(v_3 - v_4\right) = \left(\frac{-\lambda}{8} - c\right)\frac{-1}{2\sqrt{\lambda}}\left(v_3 - v_4\right)\\
\Leftrightarrow& \frac{\lambda}{8} - 2c = \frac{-\lambda}{8} - c \Leftrightarrow c = \frac{\lambda}{4}
\end{eqnarray*}
Hence $c = \frac{\lambda}{4} = \frac{-5\lambda}{8}$, which is a contradiction.
\end{proof}
By (\ref{kobayashinomizu}), the Kobayashi-Nomizu connection $\nabla^1$ of $(\mathfrak{g}_4, S_2^{\lambda}, J)$ is given by
\begin{multicols}{3}
$\nabla^1_{v_1} v_1 = 0$\par
$\nabla^1_{v_1} v_2 = \frac{1}{2}\sqrt{\frac{\lambda}{2}}v_3$\par
$\nabla^1_{v_1} v_3 = \frac{-1}{2}\sqrt{\frac{\lambda}{2}}v_2$\par
$\nabla^1_{v_1} v_4 = 0$\par
$\nabla^1_{v_2} v_1 = \frac{1}{2}\sqrt{\frac{\lambda}{2}}v_3$\par
$\nabla^1_{v_2} v_2 = 0$\par
$\nabla^1_{v_2} v_3 = \frac{-1}{2}\sqrt{\frac{\lambda}{2}}v_1$\par
$\nabla^1_{v_2} v_4 = \frac{1}{2\sqrt{\lambda}}v_4$\par
$\nabla^1_{v_3} v_1 = \frac{1}{2}\sqrt{\frac{\lambda}{2}}v_2$\par
$\nabla^1_{v_3} v_2 = \frac{-1}{2}\sqrt{\frac{\lambda}{2}}v_1 + \frac{1}{2\sqrt{\lambda}}v_3$\par
$\nabla^1_{v_3} v_3 = \frac{-1}{2\sqrt{\lambda}}v_2$\par
$\nabla^1_{v_3} v_4 = 0$\par
$\nabla^1_{v_4} v_1 = \sqrt{\frac{\lambda}{2}}v_2$\par
$\nabla^1_{v_4} v_2 = \frac{1}{2\sqrt{\lambda}}v_3$\par
$\nabla^1_{v_4} v_3 = 0$\par
$\nabla^1_{v_4} v_4 = 0$
\end{multicols}

Applying (\ref{kobayashicurvature}), we obtain that the curvature $R^1$ of $\nabla^1$ is described by
\begin{multicols}{3}
$R^1_{v_1v_2}v_1 = \frac{\lambda}{8}v_2$\par
$R^1_{v_1v_2}v_2 = \frac{-\lambda}{8}v_1$\par
$R^1_{v_1v_2}v_3 = 0$\par
$R^1_{v_1v_2}v_4 = 0$\par
$R^1_{v_1v_3}v_1 = \frac{-3\lambda}{8}v_3$\par
$R^1_{v_1v_3}v_2 = 0$\par
$R^1_{v_1v_3}v_3 = \frac{3\lambda}{8}v_1$\par
$R^1_{v_1v_3}v_4 = \frac{-1}{2\sqrt{2}}v_4$\par
$R^1_{v_1v_4}v_1 = \frac{-\lambda}{2}v_3$\par
$R^1_{v_1v_4}v_2 = \frac{1}{4\sqrt{2}}v_2$\par
$R^1_{v_1v_4}v_3 = \frac{\lambda}{4}v_1 - \frac{1}{4\sqrt{2}}v_3$\par
$R^1_{v_1v_4}v_4 = \frac{-1}{2\sqrt{2}}v_4$\par
$R^1_{v_2v_3}v_1 = 0$\par
$R^1_{v_2v_3}v_2 = \frac{\lambda}{8}v_3 + \frac{1}{2\sqrt{2}}v_1$\par
$R^1_{v_2v_3}v_3 = \frac{1}{4\lambda}v_2 - \frac{\lambda}{8}v_2$\par
$R^1_{v_2v_3}v_4 = 0$\par
$R^1_{v_2v_4}v_1 = \frac{1}{4\sqrt{2}}v_2$\par
$R^1_{v_2v_4}v_2 = \frac{1}{2\sqrt{2}}v_1$\par
$R^1_{v_2v_4}v_3 = \frac{1}{4\lambda}v_2 - \frac{\lambda}{4}v_2$\par
$R^1_{v_2v_4}v_4 = 0$\par
$R^1_{v_3v_4}v_1 = \frac{\lambda}{4}v_1 - \frac{1}{4\sqrt{2}}v_3$\par
$R^1_{v_3v_4}v_2 = \frac{1}{4\lambda}v_2 - \frac{\lambda}{4}v_2$\par
$R^1_{v_3v_4}v_3 = \frac{-1}{4\lambda}v_3$\par
$R^1_{v_3v_4}v_4 = 0$\par
\end{multicols}

According to (\ref{canonicalriccioperators}), the Ricci operator $\operatorname{Ric}^1$ associated with $\nabla^1$ is represented in the basis $\mathscr{B}$ by 
$$\operatorname{Ric}^1 = \begin{bmatrix}
\frac{-\lambda}{4} & 0 & \frac{1}{2\sqrt{2}} & \frac{1}{2\sqrt{2}}\\
\\
0 & \frac{\lambda^2 - 1}{4\lambda} & 0 & 0\\
\\
0 & 0 & \frac{-\lambda}{4} & \frac{-2\lambda^2 - 1}{4\lambda}\\
\\
\frac{-1}{2\sqrt{2}} & 0 & 0 & 0
\end{bmatrix}.$$
\begin{theorem}
$\left(\mathfrak{g}_4, S_2^{\lambda}, J\right)$ does not exhibit the properties of a first kind algebraic Ricci soliton associated to the connection $\nabla^1$.
\end{theorem}
\begin{proof}
Assume that $\left(\mathfrak{g}_4, S_2^{\lambda}, J\right)$ is a first kind algebraic Ricci soliton associated to $\nabla^1$, then there exists a real number $c$ and a derivation $D$ of $\mathfrak{g}_4$ such that $\operatorname{Ric}^1 = cI_4 + D$. We arrive at the result that 
$$D = \operatorname{Ric}^1 - cI_4 = \begin{bmatrix}
\frac{-\lambda}{4} - c & 0 & \frac{1}{2\sqrt{2}} & \frac{1}{2\sqrt{2}}\\
\\
0 & \frac{\lambda^2 - 1}{4\lambda} - c & 0 & 0\\
\\
0 & 0 & \frac{-\lambda}{4} - c & \frac{-2\lambda^2 - 1}{4\lambda}\\
\\
\frac{-1}{2\sqrt{2}} & 0 & 0 & - c
\end{bmatrix}.$$
Given that $D$ is a derivation of  $\mathfrak{g}_4$, it satisfies the condition
\begin{eqnarray*}
& [D(v_1), v_2] + [v_1, D(v_2)] = D[v_1, v_2] \\ \Leftrightarrow & [D(v_1), v_2] + [v_1, D(v_2)] = 0\\
\Leftrightarrow& \left[ \left(\frac{-\lambda}{4} - c\right)v_1 - \frac{1}{2\sqrt{2}}v_4, v_2\right] + \left[ v_1, \left(\frac{\lambda^2 - 1}{4\lambda} - c\right)v_2\right] = 0\\
\Leftrightarrow& \frac{-1}{2\sqrt{2}}[v_4, v_2] = 0 \Leftrightarrow \frac{-1}{2\sqrt{2}} \times \frac{1}{2\sqrt{\lambda}}(v_3 - v_4) = 0 \Leftrightarrow \frac{-1}{4\sqrt{2\lambda}} = 0
\end{eqnarray*}
We see that the last equality as above is a contradiction.
\end{proof}
By (\ref{tilderho}) and (\ref{tildeRic}), we obtain that $\widetilde{\operatorname{Ric}}^1$ is represented in the basis $\mathscr{B}$ by
$$\widetilde{\operatorname{Ric}}^1 = \begin{bmatrix}
\frac{-\lambda}{4} & 0 & \frac{1}{4\sqrt{2}} & \frac{1}{2\sqrt{2}}\\
\\
0 & \frac{\lambda^2 - 1}{4\lambda} & 0 & 0\\
\\
\frac{1}{4\sqrt{2}} & 0 & \frac{-\lambda}{4} & \frac{-2\lambda^2 - 1}{8\lambda}\\
\\
\frac{-1}{2\sqrt{2}} & 0 & \frac{2\lambda^2 + 1}{8\lambda} & 0
\end{bmatrix}.$$
\begin{theorem}
$\left(\mathfrak{g}_4, S_2^{\lambda}, J\right)$ does not exhibit the properties of a second kind algebraic Ricci soliton associated to the connection $\nabla^1$.
\end{theorem}
\begin{proof}
We are looking for a real number $c$ and a derivation $D$ of $\mathfrak{g}_4$ such that $\widetilde{\operatorname{Ric}}^1 = cI_4 + D$. From this, it follows that 
$$D = \widetilde{\operatorname{Ric}}^1 - cI_4 = \begin{bmatrix}
\frac{-\lambda}{4} - c & 0 & \frac{1}{4\sqrt{2}} & \frac{1}{2\sqrt{2}}\\
\\
0 & \frac{\lambda^2 - 1}{4\lambda} - c & 0 & 0\\
\\
\frac{1}{4\sqrt{2}} & 0 & \frac{-\lambda}{4} - c & \frac{-2\lambda^2 - 1}{8\lambda}\\
\\
\frac{-1}{2\sqrt{2}} & 0 & \frac{2\lambda^2 + 1}{8\lambda} & - c
\end{bmatrix}.$$
Since $D$ is a derivation of  $\mathfrak{g}_4$, the rule below is satisfied
\begin{eqnarray*}
& [D(v_1), v_2] + [v_1, D(v_2)] = D[v_1, v_2] \\ \Leftrightarrow & [D(v_1), v_2] + [v_1, D(v_2)] = 0\\
\Leftrightarrow& \left[ \left(\frac{-\lambda}{4} - c\right)v_1 + \frac{1}{4\sqrt{2}}v_3 - \frac{1}{2\sqrt{2}}v_4, v_2\right] + \left[ v_1, \left(\frac{\lambda^2 - 1}{4\lambda} - c\right)v_2\right] = 0\\
\Leftrightarrow& \frac{1}{4\sqrt{2}}[v_3, v_2] - \frac{1}{2\sqrt{2}}[v_4, v_2] = 0\\ 
\Leftrightarrow& \frac{-1}{4\sqrt{2}} \times \frac{1}{2\sqrt{\lambda}}(v_3 - v_4) = 0 \Leftrightarrow \frac{-1}{8\sqrt{2\lambda}} = 0
\end{eqnarray*}
We see that the last equality as above is a contradiction.
\end{proof}
\subsection{The Lorentzian Lie algebra $\left(\mathfrak{g}_4, S_3^{\lambda}\right)$}
Define the following elements 
$$v_1 = e_2, \qquad v_2 = \frac{1}{\sqrt{\lambda}}e_4, \qquad v_3 = \frac{1}{\sqrt{2}}(e_1 + e_3), \qquad v_4 = \frac{1}{\sqrt{2}}(e_1 - e_3).$$
Then $\mathscr{B} = \left\lbrace v_1, v_2, v_3, v_4\right\rbrace$ is an orthonormal basis of $\left(\mathfrak{g}_4, S_3^{\lambda}\right)$. The bracket is given in $\mathscr{B}$ by 
$$[v_1, v_3] = \frac{1}{2}(v_4 - v_3), \qquad [v_1, v_4] = \frac{1}{2}(v_4 - v_3), \qquad [v_3, v_4] = -\sqrt{\lambda}v_2.$$
By \cite{aitbenhaddou2026lorentzian}, the Levi-Civita connection of $\left(\mathfrak{g}_4, S_3^{\lambda}\right)$ is characterized by the following formulas
\begin{multicols}{2}
$\nabla_{v_1} v_1 = \nabla_{v_2} v_2 = 0 $\par
$\nabla_{v_3} v_3 = \nabla_{v_4} v_4 = \frac{-1}{2}v_1 $\par
$\nabla_{v_1} v_2 = \nabla_{v_2} v_1 = 0$\par
$\nabla_{v_1} v_3 = \nabla_{v_1} v_4 = 0$\par
$\nabla_{v_3} v_1 = \nabla_{v_4} v_1 = \frac{1}{2}v_3 - \frac{1}{2}v_4$\par
$\nabla_{v_2} v_3 = \nabla_{v_3} v_2 = \frac{-\sqrt{\lambda}}{2}v_4$\par
$\nabla_{v_2} v_4 = \nabla_{v_4} v_2 = \frac{-\sqrt{\lambda}}{2}v_3$\par
$\nabla_{v_3} v_4 = \frac{-1}{2}v_1 - \frac{\sqrt{\lambda}}{2}v_2$\par
$\nabla_{v_4} v_3 = \frac{-1}{2}v_1 + \frac{\sqrt{\lambda}}{2}v_2$
\end{multicols}
The nonzero terms in $\nabla_{v_i} (J)v_j$ are
\begin{multicols}{3}
	$\nabla_{v_2} (J)v_3 = -\sqrt{\lambda}v_4$\par
	$\nabla_{v_2} (J)v_4 = \sqrt{\lambda}v_3$\par
	$\nabla_{v_3} (J)v_1 = -v_4$\par
	$\nabla_{v_3} (J)v_2 = -\sqrt{\lambda}v_4$\par
	$\nabla_{v_3} (J)v_4 = v_1 + \sqrt{\lambda}v_2$\par
	$\nabla_{v_4} (J)v_1 = -v_4$\par
	$\nabla_{v_4} (J)v_4 = v_1$
\end{multicols}

The nonzero terms in the canonical connection $\nabla^0$ of $\left(\mathfrak{g}_4, S_3^{\lambda}\right)$ are
\begin{multicols}{3}
	$\nabla^0_{v_3} v_1 = \frac{1}{2}v_3$\par
	$\nabla^0_{v_3} v_3 = \frac{-1}{2}v_1$\par
	$\nabla^0_{v_4} v_1 = \frac{1}{2}v_3$\par
	$\nabla^0_{v_4} v_2 = \frac{-\sqrt{\lambda}}{2}v_3$\par
	$\nabla^0_{v_4} v_3 = \frac{-1}{2}v_1 + \frac{\sqrt{\lambda}}{2}v_2$
\end{multicols}

To obtain the Ricci operator associated with $\nabla^0$, we only need to describe the following curvatures
\begin{multicols}{4}
	$R^0_{v_2v_1}v_2 = 0$\par
	$R^0_{v_3v_1}v_3 = \frac{-\sqrt{\lambda}}{4}v_2$\par
	$R^0_{v_4v_1}v_4 = 0$\par
	$R^0_{v_1v_2}v_1 = 0$\par
	$R^0_{v_3v_2}v_3 = 0$\par
	$R^0_{v_4v_2}v_4 = 0$\par
	$R^0_{v_1v_3}v_1 = 0$\par
	$R^0_{v_2v_3}v_2 = 0$\par
	$R^0_{v_4v_3}v_4 = 0$\par
	$R^0_{v_1v_4}v_1 = 0$\par
	$R^0_{v_2v_4}v_2 = 0$\par
	$R^0_{v_3v_4}v_3 = 0$
\end{multicols}
The Ricci operator $\operatorname{Ric}^0$ of $\nabla^0$ is represented in the basis $\mathscr{B}$ by the following matrix
$$\operatorname{Ric}^0 = \begin{bmatrix}
0 & 0 & 0 & 0\\
\frac{-\sqrt{\lambda}}{4} & 0 & 0 & 0\\
0 & 0 & 0 & 0\\
0 & 0 & 0 & 0
\end{bmatrix}.$$
\begin{theorem}
$\left(\mathfrak{g}_4, S_3^{\lambda}, J\right)$ is a first kind algebraic Ricci soliton associated to the connection $\nabla^0$ if and only if $c = 0$.
\end{theorem}
\begin{proof}
Assume that $\left(\mathfrak{g}_4, S_3^{\lambda}, J\right)$ is a first kind algebraic Ricci soliton, then there exists a real number $c$ and a derivation $D$ of $\mathfrak{g}_4$ such that $\operatorname{Ric}^0 = cI_4 + D$. From this, it follows that 
$$D = \operatorname{Ric}^0 - cI_4 = \begin{bmatrix}
-c & 0 & 0 & 0\\
\frac{-\sqrt{\lambda}}{4} & -c & 0 & 0\\
0 & 0 & -c & 0\\
0 & 0 & 0 & -c
\end{bmatrix}.$$
Given that $D$ is a derivation of  $\mathfrak{g}_4$, the following condition holds
\begin{eqnarray*}
& [D(v_1), v_3] + [v_1, D(v_3)] = D[v_1, v_3] \\ \Leftrightarrow & \left[-cv_1 - \frac{\sqrt{\lambda}}{4}v_2, v_3\right] + \left[ v_1, -cv_3\right] = \frac{1}{2}\left( D(v_4) - D(v_3)\right)\\
\Leftrightarrow& -2c[v_1, v_3] = \frac{1}{2}\left( D(v_4) - D(v_3)\right) \\
\Leftrightarrow& -2c\frac{1}{2}\left( v_4 - v_3\right) = \frac{1}{2}\left( -cv_4 + cv_3\right) \\
\Leftrightarrow& 2c = c \Leftrightarrow c = 0
\end{eqnarray*}
Now, it is easy to see that $D = \operatorname{Ric}^0$ satisfy all the other conditions 
$$[D(v_i), v_j] + [v_i, D(v_j)] = D[v_i, v_j].$$
\end{proof}
We obtain that $\widetilde{\operatorname{Ric}}^0$ is represented in the basis $\mathscr{B}$ by the following matrix
$$\widetilde{\operatorname{Ric}}^0 = \begin{bmatrix}
0 & \frac{-\sqrt{\lambda}}{8} & 0 & 0\\
\frac{-\sqrt{\lambda}}{8} & 0 & 0 & 0\\
0 & 0 & 0 & 0\\
0 & 0 & 0 & 0
\end{bmatrix}.$$
\begin{theorem}
$\left(\mathfrak{g}_4, S_3^{\lambda}, J\right)$ does not exhibit the properties of a second kind algebraic Ricci soliton associated to the connection $\nabla^0$.
\end{theorem}
\begin{proof}
The proof is similar to that of the previous theorem; the condition
$$[D(v_1), v_3] + [v_1, D(v_3)] = D[v_1, v_3],$$
shows that $c = 0$. Next, $D = \widetilde{\operatorname{Ric}}^0$ satisfies the condition
$$[D(v_3), v_4] + [v_3, D(v_4)] = D[v_3, v_4] \Leftrightarrow 0 = -\sqrt{\lambda}D(v_2) \Leftrightarrow 0 = \frac{\lambda}{8}v_1.$$
This is a contradiction.
\end{proof}
The nonzero terms in the Kobayashi-Nomizu connection $\nabla^1$ of $(\mathfrak{g}_4, S_3^{\lambda}, J)$ are
\begin{multicols}{3}
	$\nabla^1_{v_1} v_4 = \frac{1}{2}v_4$\par
	$\nabla^1_{v_3} v_1 = \frac{1}{2}v_3$\par
	$\nabla^1_{v_3} v_3 = \frac{-1}{2}v_1$\par
	$\nabla^1_{v_4} v_1 = \frac{1}{2}v_3$\par
	$\nabla^1_{v_4} v_3 = \sqrt{\lambda}v_2$\par
\end{multicols}

In order to find to Ricci operator associated with $\nabla^1$, we need the following curvatures
\begin{multicols}{3}
	$R^1_{v_2v_1}v_2 = 0$\par
	$R^1_{v_3v_1}v_3 = \frac{-1}{4}v_1 - \frac{\sqrt{\lambda}}{2}v_2$\par
	$R^1_{v_4v_1}v_4 = 0$\par
	$R^1_{v_1v_2}v_1 = 0$\par
	$R^1_{v_3v_2}v_3 = 0$\par
	$R^1_{v_4v_2}v_4 = 0$\par
	$R^1_{v_1v_3}v_1 = 0$\par
	$R^1_{v_2v_3}v_2 = 0$\par
	$R^1_{v_4v_3}v_4 = 0$\par
	$R^1_{v_1v_4}v_1 = 0$\par
	$R^1_{v_2v_4}v_2 = 0$\par
	$R^1_{v_3v_4}v_3 = \frac{-1}{4}v_3$
\end{multicols}
We see that the Ricci operator $\operatorname{Ric}^1$ of $\nabla^1$ is represented in the basis $\mathscr{B}$ by the following matrix
$$\operatorname{Ric}^1 = \begin{bmatrix}
\frac{-1}{4} & 0 & 0 & 0\\
\frac{-\sqrt{\lambda}}{2} & 0 & 0 & 0\\
0 & 0 & 0 & \frac{-1}{4}\\
0 & 0 & 0 & 0
\end{bmatrix}.$$
\begin{theorem}
$\left(\mathfrak{g}_4, S_3^{\lambda}, J\right)$ is not a first kind algebraic Ricci soliton associated to the connection $\nabla^1$.
\end{theorem}
\begin{proof}
Let us assume that $\left(\mathfrak{g}_4, S_3^{\lambda}, J\right)$ is a first kind algebraic Ricci soliton associated to $\nabla^1$, then $\operatorname{Ric}^1 = cI_4 + D$, where $c \in \mathbb{R}$ and $D$ is a derivation of $\mathfrak{g}_4$. Hence we have 
$$D = \operatorname{Ric}^1 - cI_4 = \begin{bmatrix}
\frac{-1}{4} - c & 0 & 0 & 0\\
\frac{-\sqrt{\lambda}}{2} & -c & 0 & 0\\
0 & 0 & -c & \frac{-1}{4}\\
0 & 0 & 0 & -c
\end{bmatrix}.$$
Since $D$ is a derivation of  $\mathfrak{g}_4$, it satisfies the following condition
\begin{eqnarray*}
& [D(v_1), v_3] + [v_1, D(v_3)] = D[v_1, v_3] \\
\Leftrightarrow& \left[ \left(\frac{-1}{4} - c\right)v_1 - \frac{\sqrt{\lambda}}{2}v_2, v_3\right] + \left[ v_1, -cv_3\right] = \frac{1}{2}\left( D(v_4) - D(v_3)\right)\\
\Leftrightarrow& \left(\frac{-1}{4} - 2c\right)[v_1, v_3] = \frac{1}{2}\left( D(v_4) - D(v_3)\right)\\
\Leftrightarrow& \left(\frac{-1}{4} - 2c\right)\frac{1}{2}\left(v_4 - v_3\right) = \frac{1}{2}\left( \left(\frac{-1}{4} + c\right)v_3 - cv_4\right)
\end{eqnarray*}
By identifying the components of $v_4$ and $v_3$, we obtain 
$$\left\lbrace\begin{array}{lll}
\frac{-1}{4} - 2c = -c \Longrightarrow c = \frac{-1}{4} \\
\frac{1}{4} + 2c = \frac{-1}{4} + c \Longrightarrow c = \frac{-1}{2}
\end{array}\right.$$
This is a contradiction.
\end{proof}
We see that $\widetilde{\operatorname{Ric}}^1$ is represented in the basis $\mathscr{B}$ by
$$\widetilde{\operatorname{Ric}}^1 = \begin{bmatrix}
\frac{-1}{4} & \frac{-\sqrt{\lambda}}{4} & 0 & 0\\
\frac{-\sqrt{\lambda}}{4} & 0 & 0 & 0\\
0 & 0 & 0 & \frac{-1}{8}\\
0 & 0 & \frac{1}{8} & 0
\end{bmatrix}.$$
\begin{theorem}
$\left(\mathfrak{g}_4, S_3^{\lambda}, J\right)$ is not a second kind algebraic Ricci soliton associated to the connection $\nabla^1$.
\end{theorem}
\begin{proof}
We are searching for some $c \in \mathbb{R}$ and a derivation $D$ of $\mathfrak{g}_4$ such that $\widetilde{\operatorname{Ric}}^1 = cI_4 + D$. Then
$$D = \widetilde{\operatorname{Ric}}^1 - cI_4 = \begin{bmatrix}
\frac{-1}{4} - c & \frac{-\sqrt{\lambda}}{4} & 0 & 0\\
\frac{-\sqrt{\lambda}}{4} & -c & 0 & 0\\
0 & 0 & -c & \frac{-1}{8}\\
0 & 0 & \frac{1}{8} & -c
\end{bmatrix}.$$
Given that $D$ is a derivation of  $\mathfrak{g}_4$, the condition below holds
\begin{eqnarray*}
& [D(v_2), v_3] + [v_2, D(v_3)] = D[v_2, v_3] \\
\Leftrightarrow& \left[ \frac{-\sqrt{\lambda}}{4}v_1 - cv_2, v_3\right] + \left[ v_2, -cv_3 + \frac{1}{8}v_4\right] = 0\\
\Leftrightarrow& \frac{-\sqrt{\lambda}}{4}[v_1, v_3] = 0 \Leftrightarrow \frac{-\sqrt{\lambda}}{4}\frac{1}{2}(v_4 - v_3) = 0
\end{eqnarray*}
This is a contradiction because $\lambda > 0$.
\end{proof}
\subsection{The Lorentzian Lie algebra $\left(\mathfrak{g}_4, S_4^{\lambda}\right)$}
Choose the following elements 
$$v_1 = e_1, \qquad v_2 = \frac{1}{\sqrt{\lambda}}e_4, \qquad v_3 = \frac{1}{\sqrt{2}}(e_2 + e_3), \qquad v_4 = \frac{1}{\sqrt{2}}(e_2 - e_3).$$
Then $\mathscr{B} = \left\lbrace v_1, v_2, v_3, v_4\right\rbrace$ is an orthonormal basis of $\left(\mathfrak{g}_4, S_4^{\lambda}\right)$. The bracket in the basis $\mathscr{B}$ is given by
$$[v_1, v_3] = \frac{1}{2}v_3 - \frac{1}{2}v_4 + \sqrt{\frac{\lambda}{2}}v_2, \qquad [v_1, v_4] = \frac{1}{2}v_3 - \frac{1}{2}v_4 - \sqrt{\frac{\lambda}{2}}v_2.$$
By \cite{aitbenhaddou2026lorentzian}, the Levi-Civita connection of $\left(\mathfrak{g}_4, S_4^{\lambda}\right)$ is characterized by the following formulas
\begin{multicols}{2}
$\nabla_{v_1} v_1 = \nabla_{v_2} v_2 = 0 $\par
$\nabla_{v_3} v_3 = \nabla_{v_4} v_4 = \frac{1}{2}v_1 $\par
$\nabla_{v_1} v_2 = \nabla_{v_2} v_1 = \frac{-1}{2}\sqrt{\frac{\lambda}{2}}v_3 - \frac{1}{2}\sqrt{\frac{\lambda}{2}}v_4$\par
$\nabla_{v_1} v_3 = -\nabla_{v_1} v_4 = \frac{1}{2}\sqrt{\frac{\lambda}{2}}v_2$\par
$\nabla_{v_3} v_1 = \frac{-1}{2}\sqrt{\frac{\lambda}{2}}v_2 - \frac{1}{2}v_3 + \frac{1}{2}v_4$\par
$\nabla_{v_4} v_1 = \frac{1}{2}\sqrt{\frac{\lambda}{2}}v_2 - \frac{1}{2}v_3 + \frac{1}{2}v_4$\par
$\nabla_{v_2} v_3 = \nabla_{v_3} v_2 = \frac{1}{2}\sqrt{\frac{\lambda}{2}}v_1$\par
$\nabla_{v_2} v_4 = \nabla_{v_4} v_2 = \frac{-1}{2}\sqrt{\frac{\lambda}{2}}v_1$\par
$\nabla_{v_3} v_4 = \nabla_{v_4} v_3 = \frac{1}{2}v_1$
\end{multicols}
The terms $\nabla_{v_i} (J)v_j$ are given by the following equalities
\begin{multicols}{3}
	$\nabla_{v_1} (J)v_1 = 0$\par
	$\nabla_{v_1} (J)v_2 = -\sqrt{\frac{\lambda}{2}}v_4$\par
	$\nabla_{v_1} (J)v_3 = 0$\par
	$\nabla_{v_1} (J)v_4 = \sqrt{\frac{\lambda}{2}}v_2$\par
	$\nabla_{v_2} (J)v_1 = -\sqrt{\frac{\lambda}{2}}v_4$\par
	$\nabla_{v_2} (J)v_2 = 0$\par
	$\nabla_{v_2} (J)v_3 = 0$\par
	$\nabla_{v_2} (J)v_4 = \sqrt{\frac{\lambda}{2}}v_1$\par
	$\nabla_{v_3} (J)v_1 = v_4$\par
	$\nabla_{v_3} (J)v_2 = 0$\par
	$\nabla_{v_3} (J)v_3 = 0$\par
	$\nabla_{v_3} (J)v_4 = -v_1$\par
	$\nabla_{v_4} (J)v_1 = v_4$\par
	$\nabla_{v_4} (J)v_2 = 0$\par
	$\nabla_{v_4} (J)v_3 = 0$\par
	$\nabla_{v_4} (J)v_4 = -v_1$\par
\end{multicols}

The canonical connection $\nabla^0$ of $\left(\mathfrak{g}_4, S_4^{\lambda}\right)$ is given by
\begin{multicols}{3}
	$\nabla^0_{v_1} v_1 = 0$\par
	$\nabla^0_{v_1} v_2 = \frac{-1}{2}\sqrt{\frac{\lambda}{2}}v_3$\par
	$\nabla^0_{v_1} v_3 = \frac{1}{2}\sqrt{\frac{\lambda}{2}}v_2$\par
	$\nabla^0_{v_1} v_4 = 0$\par
	$\nabla^0_{v_2} v_1 = \frac{-1}{2}\sqrt{\frac{\lambda}{2}}v_3$\par
	$\nabla^0_{v_2} v_2 = 0$\par
	$\nabla^0_{v_2} v_3 = \frac{1}{2}\sqrt{\frac{\lambda}{2}}v_1$\par
	$\nabla^0_{v_2} v_4 = 0$\par
	$\nabla^0_{v_3} v_1 = \frac{-1}{2}\sqrt{\frac{\lambda}{2}}v_2 - \frac{1}{2}v_3$\par
	$\nabla^0_{v_3} v_2 = \frac{1}{2}\sqrt{\frac{\lambda}{2}}v_1$\par
	$\nabla^0_{v_3} v_3 = \frac{1}{2}v_1$\par
	$\nabla^0_{v_3} v_4 = 0$\par
	$\nabla^0_{v_4} v_1 = \frac{1}{2}\sqrt{\frac{\lambda}{2}}v_2 - \frac{1}{2}v_3$\par
	$\nabla^0_{v_4} v_2 = \frac{-1}{2}\sqrt{\frac{\lambda}{2}}v_1$\par
	$\nabla^0_{v_4} v_3 = \frac{1}{2}v_1$\par
	$\nabla^0_{v_4} v_4 = 0$
\end{multicols}
In order to find to Ricci operator associated with $\nabla^0$, we need the following curvatures
\begin{multicols}{3}
	$R^0_{v_2v_1}v_2 = \frac{\lambda}{8}v_1$\par
	$R^0_{v_3v_1}v_3 = \frac{-3\lambda}{8}v_1$\par
	$R^0_{v_4v_1}v_4 = 0$\par
	$R^0_{v_1v_2}v_1 = \frac{\lambda}{8}v_2$\par
	$R^0_{v_3v_2}v_3 = \frac{\lambda}{8}v_2$\par
	$R^0_{v_4v_2}v_4 = 0$\par
	$R^0_{v_1v_3}v_1 = \frac{-1}{4}\sqrt{\frac{\lambda}{2}}v_2 - \frac{3\lambda}{8}v_3$\par
	$R^0_{v_2v_3}v_2 = \frac{\lambda}{8}v_3$\par
	$R^0_{v_4v_3}v_4 = 0$\par
	$R^0_{v_1v_4}v_1 = \frac{-1}{4}\sqrt{\frac{\lambda}{2}}v_2 + \frac{3\lambda}{8}v_3$\par
	$R^0_{v_2v_4}v_2 = \frac{-\lambda}{8}v_3$\par
	$R^0_{v_3v_4}v_3 = \frac{1}{2}\sqrt{\frac{\lambda}{2}}v_2$
\end{multicols}
The Ricci operator $\operatorname{Ric}^0$ of $\nabla^0$ is represented in the basis $\mathscr{B}$ by the following matrix
$$\operatorname{Ric}^0 = \begin{bmatrix}
\frac{-\lambda}{4} & 0 & 0 & 0\\
\\
0 & \frac{\lambda}{4} & \frac{-1}{4}\sqrt{\frac{\lambda}{2}} & \frac{1}{4}\sqrt{\frac{\lambda}{2}}\\
\\
0 & 0 & \frac{-\lambda}{4} & \frac{\lambda}{4}\\
\\
0 & 0 & 0 & 0
\end{bmatrix}.$$
\begin{theorem}
$\left(\mathfrak{g}_4, S_4^{\lambda}, J\right)$ does not qualify as a first kind algebraic Ricci soliton associated to the connection $\nabla^0$.
\end{theorem}
\begin{proof}
Assume that $\left(\mathfrak{g}_4, S_4^{\lambda}, J\right)$ is a first kind algebraic Ricci soliton associated to $\nabla^0$, then there exists a real number $c$ and a derivation $D$ of $\mathfrak{g}_4$ such that $\operatorname{Ric}^0 = cI_4 + D$. This implies that 
$$D = \operatorname{Ric}^0 - cI_4 = \begin{bmatrix}
\frac{-\lambda}{4} - c & 0 & 0 & 0\\
\\
0 & \frac{\lambda}{4} - c & \frac{-1}{4}\sqrt{\frac{\lambda}{2}} & \frac{1}{4}\sqrt{\frac{\lambda}{2}}\\
\\
0 & 0 & \frac{-\lambda}{4} - c & \frac{\lambda}{4}\\
\\
0 & 0 & 0 & - c
\end{bmatrix}.$$
Since $D$ is a derivation of  $\mathfrak{g}_4$, it satisfies the following condition
\begin{eqnarray*}
& [D(v_1), v_3] + [v_1, D(v_3)] = D[v_1, v_3] \\
\Leftrightarrow& \left[ \left(\frac{-\lambda}{4} - c\right)v_1, v_3\right] + \left[ v_1, \frac{-1}{4}\sqrt{\frac{\lambda}{2}}v_2 + \left(\frac{-\lambda}{4} - c\right)v_3\right] = \frac{1}{2}D(v_3) - \frac{1}{2}D(v_4) + \sqrt{\frac{\lambda}{2}}D(v_2)\\
\Leftrightarrow& \left(\frac{-\lambda}{2} - 2c\right)[v_1, v_3] = \frac{1}{2}D(v_3) - \frac{1}{2}D(v_4) + \sqrt{\frac{\lambda}{2}}D(v_2)
\end{eqnarray*}
We develop the following 
\begin{eqnarray*}
& \frac{1}{2}D(v_3) - \frac{1}{2}D(v_4) + \sqrt{\frac{\lambda}{2}}D(v_2) \\
=& \frac{-1}{8}\sqrt{\frac{\lambda}{2}}v_2 + \frac{1}{2}\left( \frac{-\lambda}{4} - c\right)v_3 - \frac{1}{8}\sqrt{\frac{\lambda}{2}}v_2 - \frac{\lambda}{8}v_3 + \frac{c}{2}v_4 + \sqrt{\frac{\lambda}{2}}\left( \frac{\lambda}{4} - c\right)v_2
\end{eqnarray*}
By identifying the component of $v_4$ in the following equality 
$$(\frac{-\lambda}{2} - 2c)[v_1, v_3] = \frac{-1}{8}\sqrt{\frac{\lambda}{2}}v_2 + \frac{1}{2}(\frac{-\lambda}{4} - c)v_3 - \frac{1}{8}\sqrt{\frac{\lambda}{2}}v_2 - \frac{\lambda}{8}v_3 + \frac{c}{2}v_4 + \sqrt{\frac{\lambda}{2}}(\frac{\lambda}{4} - c)v_2$$
we obtain 
$$\frac{\lambda}{4} + c = \frac{c}{2} \Longrightarrow c = \frac{-\lambda}{2}.$$
Similarly, identifying the components of $v_3$ yields
$$\frac{-\lambda}{4} - c = \frac{-\lambda}{8} - \frac{c}{2} - \frac{\lambda}{8} \Longrightarrow c = 0.$$
Hence $c = 0 = \frac{-\lambda}{2}$, which is a contradiction.
\end{proof}
We obtain that $\widetilde{\operatorname{Ric}}^0$ is represented in the basis $\mathscr{B}$ by the following matrix
$$\widetilde{\operatorname{Ric}}^0 = \begin{bmatrix}
\frac{-\lambda}{4} & 0 & 0 & 0\\
\\
0 & \frac{\lambda}{4} & \frac{-1}{8}\sqrt{\frac{\lambda}{2}} & \frac{1}{8}\sqrt{\frac{\lambda}{2}}\\
\\
0 & \frac{-1}{8}\sqrt{\frac{\lambda}{2}} & \frac{-\lambda}{4} & \frac{\lambda}{8}\\
\\
0 & \frac{-1}{8}\sqrt{\frac{\lambda}{2}} & \frac{-\lambda}{8} & 0
\end{bmatrix}.$$
\begin{theorem}
$\left(\mathfrak{g}_4, S_4^{\lambda}, J\right)$ does not qualify as a second kind algebraic Ricci soliton associated to the connection $\nabla^0$.
\end{theorem}
\begin{proof}
Suppose that $\left(\mathfrak{g}_4, S_4^{\lambda}, J\right)$ is a second kind algebraic Ricci soliton, then there exists a real number $c$ and a derivation $D$ of $\mathfrak{g}_4$ such that $\widetilde{\operatorname{Ric}}^0 = cI_4 + D$. Then we obtain that
$$D = \widetilde{\operatorname{Ric}}^0 - cI_4 = \begin{bmatrix}
\frac{-\lambda}{4} - c & 0 & 0 & 0\\
\\
0 & \frac{\lambda}{4} - c & \frac{-1}{8}\sqrt{\frac{\lambda}{2}} & \frac{1}{8}\sqrt{\frac{\lambda}{2}}\\
\\
0 & \frac{-1}{8}\sqrt{\frac{\lambda}{2}} & \frac{-\lambda}{4} - c & \frac{\lambda}{8}\\
\\
0 & \frac{-1}{8}\sqrt{\frac{\lambda}{2}} & \frac{-\lambda}{8} & - c
\end{bmatrix}.$$
As $D$ is a derivation of  $\mathfrak{g}_4$, the rule below holds
\begin{eqnarray*}
& [D(v_1), v_2] + [v_1, D(v_2)] = D[v_1, v_2] \\ \Leftrightarrow & \left[ \left(\frac{-\lambda}{4} - c\right)v_1, v_2\right] + \left[ v_1, \left(\frac{\lambda}{4} - c\right)v_2 - \frac{1}{8}\sqrt{\frac{\lambda}{2}}v_3 - \frac{1}{8}\sqrt{\frac{\lambda}{2}}v_4\right] = 0\\
\Leftrightarrow& \frac{-1}{8}\sqrt{\frac{\lambda}{2}}\left([v_1, v_3] + [v_1, v_4]\right) = 0 \\
\Leftrightarrow& \frac{-1}{8}\sqrt{\frac{\lambda}{2}}\left(v_3 - v_4\right) = 0 
\end{eqnarray*}
This is a contradiction.
\end{proof}
The Kobayashi-Nomizu connection $\nabla^1$ of $(\mathfrak{g}_4, S_4^{\lambda}, J)$ is given by
\begin{multicols}{3}
	$\nabla^1_{v_1} v_1 = 0$\par
	$\nabla^1_{v_1} v_2 = \frac{-1}{2}\sqrt{\frac{\lambda}{2}}v_3$\par
	$\nabla^1_{v_1} v_3 = \frac{1}{2}\sqrt{\frac{\lambda}{2}}v_2$\par
	$\nabla^1_{v_1} v_4 = \frac{-1}{2}v_4$\par
	$\nabla^1_{v_2} v_1 = \frac{-1}{2}\sqrt{\frac{\lambda}{2}}v_3$\par
	$\nabla^1_{v_2} v_2 = 0$\par
	$\nabla^1_{v_2} v_3 = \frac{1}{2}\sqrt{\frac{\lambda}{2}}v_1$\par
	$\nabla^1_{v_2} v_4 = 0$\par
	$\nabla^1_{v_3} v_1 = \frac{-1}{2}\sqrt{\frac{\lambda}{2}}v_2 - \frac{1}{2}v_3$\par
	$\nabla^1_{v_3} v_2 = \frac{1}{2}\sqrt{\frac{\lambda}{2}}v_1$\par
	$\nabla^1_{v_3} v_3 = \frac{1}{2}v_1$\par
	$\nabla^1_{v_3} v_4 = 0$\par
	$\nabla^1_{v_4} v_1 = \sqrt{\frac{\lambda}{2}}v_2 - \frac{1}{2}v_3$\par
	$\nabla^1_{v_4} v_2 = 0$\par
	$\nabla^1_{v_4} v_3 = 0$\par
	$\nabla^1_{v_4} v_4 = 0$
\end{multicols}
To find the Ricci operator associated with $\nabla^1$, we need the following curvatures
\begin{multicols}{3}
	$R^1_{v_2v_1}v_2 = \frac{\lambda}{8}v_1$\par
	$R^1_{v_3v_1}v_3 = \frac{-1}{4}v_1 - \frac{3\lambda}{8}v_1$\par
	$R^1_{v_4v_1}v_4 = 0$\par
	$R^1_{v_1v_2}v_1 = \frac{\lambda}{8}v_2$\par
	$R^1_{v_3v_2}v_3 = \frac{\lambda}{8}v_2$\par
	$R^1_{v_4v_2}v_4 = 0$\par
	$R^1_{v_1v_3}v_1 = \frac{-1}{2}\sqrt{\frac{\lambda}{2}}v_2 - \frac{3\lambda}{8}v_3$\par
	$R^1_{v_2v_3}v_2 = \frac{\lambda}{8}v_3$\par
	$R^1_{v_4v_3}v_4 = 0$\par
	$R^1_{v_1v_4}v_1 = \frac{-1}{2}\sqrt{\frac{\lambda}{2}}v_2 + \frac{\lambda}{2}v_3$\par
	$R^1_{v_2v_4}v_2 = 0$\par
	$R^1_{v_3v_4}v_3 = \frac{1}{2}\sqrt{\frac{\lambda}{2}}v_2 - \frac{1}{4}v_3$
\end{multicols}
The Ricci operator $\operatorname{Ric}^1$ of $\nabla^1$ is represented in the basis $\mathscr{B}$ by the following matrix
$$\operatorname{Ric}^1 = \begin{bmatrix}
\frac{-1 - \lambda}{4} & 0 & 0 & 0\\
\\
0 & \frac{\lambda}{4} & \frac{-1}{2}\sqrt{\frac{\lambda}{2}} & 0\\
\\
0 & 0 & \frac{-\lambda}{4} & \frac{2\lambda - 1}{4}\\
\\
0 & 0 & 0 & 0
\end{bmatrix}.$$
\begin{theorem}
$\left(\mathfrak{g}_4, S_4^{\lambda}, J\right)$ does not exhibit the properties of a first kind algebraic Ricci soliton associated to the connection $\nabla^1$.
\end{theorem}
\begin{proof}
Assume that $\left(\mathfrak{g}_4, S_4^{\lambda}, J\right)$ is a first kind algebraic Ricci soliton associated to $\nabla^1$, then there exists a real number $c$ and a derivation $D$ of $\mathfrak{g}_4$ such that $\operatorname{Ric}^1 = cI_4 + D$. From this, it follows that 
$$D = \operatorname{Ric}^1 - cI_4 = \begin{bmatrix}
\frac{-1 - \lambda}{4} - c & 0 & 0 & 0\\
\\
0 & \frac{\lambda}{4} - c & \frac{-1}{2}\sqrt{\frac{\lambda}{2}} & 0\\
\\
0 & 0 & \frac{-\lambda}{4} - c & \frac{2\lambda - 1}{4}\\
\\
0 & 0 & 0 & - c
\end{bmatrix}.$$
Since $D$ is a derivation of  $\mathfrak{g}_4$, it satisfies the following condition
\begin{eqnarray*}
& [D(v_1), v_3] + [v_1, D(v_3)] = D[v_1, v_3] \\
\Leftrightarrow& \left[ \left(\frac{-1 - \lambda}{4} - c\right)v_1, v_3\right] + \left[ v_1, \frac{-1}{2}\sqrt{\frac{\lambda}{2}}v_2 + \left(\frac{-\lambda}{4} - c\right)v_3\right] = \frac{1}{2}D(v_3) - \frac{1}{2}D(v_4) + \sqrt{\frac{\lambda}{2}}D(v_2)\\
\Leftrightarrow& \left(\frac{-1 - 2\lambda}{4} - 2c\right)[v_1, v_3] = \frac{1}{2}D(v_3) - \frac{1}{2}D(v_4) + \sqrt{\frac{\lambda}{2}}D(v_2)
\end{eqnarray*}
We develop the following 
\begin{eqnarray*}
& \frac{1}{2}D(v_3) - \frac{1}{2}D(v_4) + \sqrt{\frac{\lambda}{2}}D(v_2) \\
=& \frac{-1}{4}\sqrt{\frac{\lambda}{2}}v_2 + \frac{1}{2}\left( \frac{-\lambda}{4} - c\right)v_3 - \frac{2\lambda - 1}{8}v_3 + \frac{c}{2}v_4 + \sqrt{\frac{\lambda}{2}}\left( \frac{\lambda}{4} - c\right)v_2
\end{eqnarray*}
By identifying the component of $v_4$ in the following equality 
$$(\frac{-1 - 2\lambda}{4} - 2c)[v_1, v_3] = \frac{-1}{4}\sqrt{\frac{\lambda}{2}}v_2 + \frac{1}{2}(\frac{-\lambda}{4} - c)v_3 - \frac{2\lambda - 1}{8}v_3 + \frac{c}{2}v_4 + \sqrt{\frac{\lambda}{2}}(\frac{\lambda}{4} - c)v_2$$
we obtain 
$$\frac{1 + 2\lambda}{8} + c = \frac{c}{2} \Longrightarrow c = \frac{-1 - 2\lambda}{4}.$$
Similarly, identifying the components of $v_3$ yields
$$\frac{-1 - 2\lambda}{8} - c = \frac{-\lambda}{8} - \frac{c}{2} - \frac{2\lambda - 1}{8} \Longrightarrow c = \frac{\lambda - 2}{4}.$$
Next, $D$ satisfies the following condition
\begin{eqnarray*}
& [D(v_1), v_4] + [v_1, D(v_4)] = D[v_1, v_4] \\
\Leftrightarrow& \left[ \left(\frac{-1 - \lambda}{4} - c\right)v_1, v_4\right] + \left[ v_1, \frac{2\lambda - 1}{4}v_3 - cv_4\right] = \frac{1}{2}D(v_3) - \frac{1}{2}D(v_4) - \sqrt{\frac{\lambda}{2}}D(v_2)\\
\Leftrightarrow& \left(\frac{-1 - \lambda}{4} - 2c\right)[v_1, v_4] + \frac{2\lambda - 1}{4}[v_1, v_3]  = \frac{1}{2}D(v_3) - \frac{1}{2}D(v_4) - \sqrt{\frac{\lambda}{2}}D(v_2)
\end{eqnarray*}
We develop the last equality as above and obtain
\begin{eqnarray*}
& \left(\frac{-1 - \lambda}{8} - c\right)v_3 + \left(\frac{1 + \lambda}{8} + c\right)v_4 - \sqrt{\frac{\lambda}{2}}\left(\frac{-1 - \lambda}{4} - 2c\right)v_2 + \frac{2\lambda - 1}{8}v_3 - \frac{2\lambda - 1}{8}v_4 + \sqrt{\frac{\lambda}{2}}\frac{2\lambda - 1}{4}v_2 \\
=& \frac{-1}{4}\sqrt{\frac{\lambda}{2}}v_2 + \frac{1}{2}\left( \frac{-\lambda}{4} - c\right)v_3 - \frac{2\lambda - 1}{8}v_3 + \frac{c}{2}v_4 - \sqrt{\frac{\lambda}{2}}\left( \frac{\lambda}{4} - c\right)v_2
\end{eqnarray*}
By identifying the component of $v_2$, we obtain 
$$\frac{1 + \lambda}{4} + 2c + \frac{2\lambda - 1}{4} = \frac{-1}{4} - \frac{\lambda}{4} + c \Longrightarrow c = \frac{-1 - 4\lambda}{4}.$$
Hence
$$c = \frac{-1 - 2\lambda}{4} = \frac{-1 - 4\lambda}{4} \Longrightarrow 2\lambda = 0.$$
This is a contradiction because $\lambda > 0$.
\end{proof}
We obtain that $\widetilde{\operatorname{Ric}}^1$ is represented in the basis $\mathscr{B}$ by
$$\widetilde{\operatorname{Ric}}^1 = \begin{bmatrix}
\frac{-1 - \lambda}{4} & 0 & 0 & 0\\
\\
0 & \frac{\lambda}{4} & \frac{-1}{4}\sqrt{\frac{\lambda}{2}} & 0\\
\\
0 & \frac{-1}{4}\sqrt{\frac{\lambda}{2}} & \frac{-\lambda}{4} & \frac{2\lambda - 1}{8}\\
\\
0 & 0 & \frac{1 - 2\lambda}{8} & 0
\end{bmatrix}.$$
\begin{theorem}
$\left(\mathfrak{g}_4, S_4^{\lambda}, J\right)$ does not exhibit the properties of a second kind algebraic Ricci soliton associated to the connection $\nabla^1$.
\end{theorem}
\begin{proof}
Suppose that $\widetilde{\operatorname{Ric}}^1 = cI_4 + D$, where $c \in \mathbb{R}$ and $D$ is a derivation of $\mathfrak{g}_4$. Thus 
$$D = \widetilde{\operatorname{Ric}}^1 - cI_4 = \begin{bmatrix}
\frac{-1 - \lambda}{4} - c & 0 & 0 & 0\\
\\
0 & \frac{\lambda}{4} - c & \frac{-1}{4}\sqrt{\frac{\lambda}{2}} & 0\\
\\
0 & \frac{-1}{4}\sqrt{\frac{\lambda}{2}} & \frac{-\lambda}{4} - c & \frac{2\lambda - 1}{8}\\
\\
0 & 0 & \frac{1 - 2\lambda}{8} & - c
\end{bmatrix}.$$
Given that $D$ is a derivation of  $\mathfrak{g}_4$, the condition below holds
\begin{eqnarray*}
& [D(v_1), v_2] + [v_1, D(v_2)] = D[v_1, v_2] \\
\Leftrightarrow& \left[ \left(\frac{-1 - \lambda}{4} - c\right)v_1, v_2\right] + \left[ v_1, \left(\frac{\lambda}{4} - c\right)v_2 - \frac{1}{4}\sqrt{\frac{\lambda}{2}}v_3\right] = 0\\
\Leftrightarrow& \frac{-1}{4}\sqrt{\frac{\lambda}{2}}[v_1, v_3] = 0 \Leftrightarrow \frac{-1}{4}\sqrt{\frac{\lambda}{2}}\left(\frac{1}{2}v_3 - \frac{1}{2}v_4 + \sqrt{\frac{\lambda}{2}}v_2\right) = 0
\end{eqnarray*}
This is a contradiction because $\lambda > 0$.
\end{proof}
\subsection{The Lorentzian Lie algebra $\left(\mathfrak{g}_4, S_A^{\pm}\right)$}
\subsubsection{Case of the Lorentzian inner product $S_A^{+}$}
We consider $S_A^{+} = \begin{bmatrix}
1 & 0 & 0 & 0\\
0 & -1 & 0 & 0\\
0 & 0 & a & b\\
0 & 0 & b & c
\end{bmatrix}$ where $A = \begin{bmatrix}
a & b\\
b & c
\end{bmatrix}$ is symmetric positive definite. We set $d = ac - b^2$; an orthonormal basis $\mathscr{B} = \left\lbrace v_1, v_2, v_3, v_4\right\rbrace$ of the Lorentzian Lie algebra $\left(\mathfrak{g}_4, S_A^{+}\right)$ is given by the following
$$v_1 = e_1, \qquad v_2 = \frac{1}{\sqrt{a}}e_3, \qquad v_3 = \frac{-b}{\sqrt{ad}}e_3 + \sqrt{\frac{a}{d}}e_4, \qquad v_4 = e_2.$$
The commutators of the elements of $\mathscr{B}$ are described by
$$[v_1, v_2] = \frac{\sqrt{d}}{a}v_3 + \frac{b}{a}v_2, \qquad [v_1, v_3] = \frac{-b}{a}v_3 - \frac{b^2}{a\sqrt{d}}v_2, \qquad [v_1, v_4] = \sqrt{a}v_2.$$
By \cite{aitbenhaddou2026lorentzian}, the Levi-Civita connection of $\left(\mathfrak{g}_4, S_A^{+}\right)$ is given by
\begin{multicols}{2}
$\nabla_{v_1} v_1 = \nabla_{v_4} v_4 = 0$\par
$\nabla_{v_2} v_2 = -\nabla_{v_3} v_3 = \frac{b}{a}v_1$\par
$\nabla_{v_1} v_2 = \frac{b^2 + d}{2a\sqrt{d}}v_3 + \frac{\sqrt{a}}{2}v_4$\par
$\nabla_{v_2} v_1 = \frac{-b}{a}v_2 + \frac{b^2 - d}{2a\sqrt{d}}v_3 + \frac{\sqrt{a}}{2}v_4$\par
$\nabla_{v_1} v_3 = \frac{-b^2 - d}{2a\sqrt{d}}v_2$\par
$\nabla_{v_3} v_1 = \frac{b^2 - d}{2a\sqrt{d}}v_2 + \frac{b}{a}v_3$\par
$\nabla_{v_1} v_4 = -\nabla_{v_4} v_1 = \frac{\sqrt{a}}{2}v_2 $\par
$\nabla_{v_2} v_3 = \nabla_{v_3} v_2 = \frac{-b^2 + d}{2a\sqrt{d}}v_1 $\par
$\nabla_{v_2} v_4 = \nabla_{v_4} v_2 = \frac{\sqrt{a}}{2}v_1 $\par
$\nabla_{v_3} v_4 = \nabla_{v_4} v_3 = 0$
\end{multicols}
The nonzero terms in $\nabla_{v_i} (J)v_j$ are
$$\nabla_{v_1} (J)v_2 = \nabla_{v_2} (J)v_1 = \sqrt{a}v_4, \qquad \nabla_{v_1} (J)v_4 = -\sqrt{a}v_2, \qquad \nabla_{v_2} (J)v_4 = -\sqrt{a}v_1.$$
The canonical connection $\nabla^0$ of $\left(\mathfrak{g}_4, S_A^{+}\right)$ is given by
\begin{multicols}{3}
	$\nabla^0_{v_1} v_1 = 0$\par
	$\nabla^0_{v_1} v_2 = \frac{b^2 + d}{2a\sqrt{d}}v_3$\par
	$\nabla^0_{v_1} v_3 = \frac{-b^2 - d}{2a\sqrt{d}}v_2$\par
	$\nabla^0_{v_1} v_4 = 0$\par
	$\nabla^0_{v_2} v_1 = \frac{-b}{a}v_2 + \frac{b^2 - d}{2a\sqrt{d}}v_3$\par
	$\nabla^0_{v_2} v_2 = \frac{b}{a}v_1$\par
	$\nabla^0_{v_2} v_3 = \frac{-b^2 + d}{2a\sqrt{d}}v_1$\par
	$\nabla^0_{v_2} v_4 = 0$\par
	$\nabla^0_{v_3} v_1 = \frac{b^2 - d}{2a\sqrt{d}}v_2 + \frac{b}{a}v_3$\par
	$\nabla^0_{v_3} v_2 = \frac{-b^2 + d}{2a\sqrt{d}}v_1$\par
	$\nabla^0_{v_3} v_3 = \frac{-b}{a}v_1$\par
	$\nabla^0_{v_3} v_4 = 0$\par
	$\nabla^0_{v_4} v_1 = \frac{-\sqrt{a}}{2}v_2$\par
	$\nabla^0_{v_4} v_2 = \frac{\sqrt{a}}{2}v_1$\par
	$\nabla^0_{v_4} v_3 = 0$\par
	$\nabla^0_{v_4} v_4 = 0$
\end{multicols}
In order to find to Ricci operator associated with $\nabla^0$, we need the following curvatures
\begin{multicols}{2}
	$R^0_{v_2v_1}v_2 = \frac{b^4 - 2b^2d - 3d^2}{4a^2d}v_1$\par
	$R^0_{v_3v_1}v_3 = \frac{-3b^4 - 2b^2d + d^2}{4a^2d}v_1$\par
	$R^0_{v_4v_1}v_4 = 0$\par
	$R^0_{v_1v_2}v_1 = \frac{-2b^2d - 3d^2 + b^4}{4a^2d}v_2 + \frac{b^3 + bd}{a^2\sqrt{d}}v_3$\par
	$R^0_{v_3v_2}v_3 = \frac{b^4 + 2b^2d + d^2}{4a^2d}v_2$\par
	$R^0_{v_4v_2}v_4 = 0$\par
	$R^0_{v_1v_3}v_1 = \frac{b^3 + bd}{a^2\sqrt{d}}v_2 + \left(\frac{-3b^4 + 2b^2d + d^2}{4a^2d} - \frac{b^2}{a^2}\right)v_3$\par
	$R^0_{v_2v_3}v_2 = \left(\frac{b^4 - 2b^2d + d^2}{4a^2d} + \frac{b^2}{a^2}\right)v_3$\par
	$R^0_{v_4v_3}v_4 = 0$\par
	$R^0_{v_1v_4}v_1 = \frac{-b}{\sqrt{a}}v_2 + \frac{3b^2 - d}{4\sqrt{ad}}v_3$\par
	$R^0_{v_2v_4}v_2 = \frac{d - b^2}{4\sqrt{ad}}v_3$\par
	$R^0_{v_3v_4}v_3 = \frac{b}{2\sqrt{a}}v_2$
\end{multicols}
We recall that $d = ac - b^2$. The Ricci operator $\operatorname{Ric}^0$ of $\nabla^0$ is represented in the basis $\mathscr{B}$ by the following matrix
$$\operatorname{Ric}^0 = \begin{bmatrix}
\frac{-c^2}{2d} & 0 & 0 & 0\\
\\
0 & \frac{c(b^2 - d)}{2ad} & \frac{bc}{a\sqrt{d}} & \frac{-b}{2\sqrt{a}}\\
\\
0 & \frac{bc}{a\sqrt{d}} & \frac{c(d - b^2)}{2ad} & \frac{b^2}{2\sqrt{ad}}\\
\\
0 & 0 & 0 & 0
\end{bmatrix}.$$
\begin{theorem}\label{th1}
$\left(\mathfrak{g}_4, S_A^{+}, J\right)$ does not exhibit the properties of a first kind algebraic Ricci soliton associated to the connection $\nabla^0$.
\end{theorem}
\begin{proof}
Assume that $\left(\mathfrak{g}_4, S_A^{+}, J\right)$ is a first kind algebraic Ricci soliton associated to $\nabla^0$, then there exists a real number $\eta$ and a derivation $D$ of $\mathfrak{g}_4$ such that $\operatorname{Ric}^0 = \eta I_4 + D$. Then we obtain that 
$$D = \operatorname{Ric}^0 - \eta I_4 = \begin{bmatrix}
\frac{-c^2}{2d} - \eta & 0 & 0 & 0\\
\\
0 & \frac{c(b^2 - d)}{2ad} - \eta & \frac{bc}{a\sqrt{d}} & \frac{-b}{2\sqrt{a}}\\
\\
0 & \frac{bc}{a\sqrt{d}} & \frac{c(d - b^2)}{2ad} - \eta & \frac{b^2}{2\sqrt{ad}}\\
\\
0 & 0 & 0 & - \eta
\end{bmatrix}.$$
Given that $D$ is a derivation of $\mathfrak{g}_4$, it satisfies the following condition
\begin{eqnarray*}
& [D(v_1), v_4] + [v_1, D(v_4)] = D[v_1, v_4] \\ \Leftrightarrow & \left( \frac{- c^2}{2d} - \eta\right)\left[v_1, v_4\right] + \left[ v_1, \frac{-b}{2\sqrt{a}}v_2 + \frac{b^2}{2\sqrt{ad}}v_3 - \eta v_4\right] = \sqrt{a}D(v_2)\\
\Leftrightarrow & \left( \frac{-c^2}{2d} - 2\eta\right)\left[v_1, v_4\right] - \frac{b}{2\sqrt{a}} \left[ v_1, v_2\right] + \frac{b^2}{2\sqrt{ad}}\left[v_1, v_3\right] = \sqrt{a}D(v_2)
\end{eqnarray*}
By carefully developing the last equation above, we obtain the following equality
\begin{eqnarray*}
& \left(\frac{-c^2}{2d} - 2\eta\right)\sqrt{a}v_2 - \frac{b\sqrt{d}}{2a\sqrt{a}}v_3 - \frac{b^2}{2a\sqrt{a}}v_2 - \frac{b^3}{2a\sqrt{ad}}v_3 - \frac{b^4}{2a\sqrt{a}d}v_2 \\
& = \left(\frac{c(b^2 - d)}{2ad} - \eta\right)\sqrt{a}v_2 + \frac{bc}{\sqrt{ad}}v_3
\end{eqnarray*}
By identifying the components $v_3$, we get
\begin{eqnarray*}
\frac{-b\sqrt{d}}{2a\sqrt{a}} - \frac{b^3}{2a\sqrt{ad}} = \frac{bc}{\sqrt{ad}} &\Leftrightarrow& \frac{-bd - b^3}{2a\sqrt{ad}} = \frac{2abc}{2a\sqrt{ad}}\\
&\Leftrightarrow& -bd - b^3 = 2abc\\
&\Leftrightarrow& -b(ac - b^2) - b^3 = 2abc\\
&\Leftrightarrow& -abc = 2abc \Leftrightarrow abc = 0
\end{eqnarray*}
Since $A = \begin{bmatrix}
a & b\\
b & c
\end{bmatrix}$ is symmetric positive definite, we have $a > 0$ and $c \neq 0$. Thus $b = 0$. By using this fact and by identifying the components of $v_2$, we get
$$\frac{-c}{2a} - 2\eta = \frac{-c}{2a} - \eta \Longrightarrow \eta = 0.$$
Hence the derivation $D$ becomes
$$D = \operatorname{Ric}^0 = \operatorname{diag}\left\lbrace \frac{-c}{2a}, \frac{-c}{2a}, \frac{c}{2a}, 0\right\rbrace.$$
Next, $D$ satisfies the following condition
\begin{eqnarray*}
[D(v_1), v_2] + [v_1, D(v_2)] = D[v_1, v_2] &\Leftrightarrow & \frac{-c}{2a}[v_1, v_2] - \frac{c}{2a}[v_1, v_2] = D\left(\sqrt{\frac{c}{a}}v_3\right)\\
&\Leftrightarrow & \frac{-c}{a}[v_1, v_2] = \sqrt{\frac{c}{a}}D(v_3)\\ 
&\Leftrightarrow& \frac{-c}{a}\sqrt{\frac{c}{a}}v_3 = \sqrt{\frac{c}{a}}\frac{c}{2a}v_3\\
&\Leftrightarrow& \frac{-c}{a} = \frac{c}{2a} \Leftrightarrow c = 0
\end{eqnarray*}
This is a contradiction.
\end{proof}
We can see that $\widetilde{\operatorname{Ric}}^0$ is represented in the basis $\mathscr{B}$ by the following matrix
$$\widetilde{\operatorname{Ric}}^0 = \begin{bmatrix}
\frac{-c^2}{2d} & 0 & 0 & 0\\
\\
0 & \frac{c(b^2 - d)}{2ad} & \frac{bc}{a\sqrt{d}} & \frac{-b}{4\sqrt{a}}\\
\\
0 & \frac{bc}{a\sqrt{d}} & \frac{c(d - b^2)}{2ad} & \frac{b^2}{4\sqrt{ad}}\\
\\
0 & \frac{b}{4\sqrt{a}} & \frac{-b^2}{4\sqrt{ad}} & 0
\end{bmatrix}.$$
\begin{theorem}
$\left(\mathfrak{g}_4, S_A^{+}, J\right)$ does not qualify as a second kind algebraic Ricci soliton associated to the connection $\nabla^0$.
\end{theorem}
\begin{proof}
Assume that $\widetilde{\operatorname{Ric}}^0 = \eta I_4 + D$, where $\eta \in \mathbb{R}$ and $D$ is a derivation of $\mathfrak{g}_4$. This implies that
$$D = \widetilde{\operatorname{Ric}}^0 - \eta I_4 = \begin{bmatrix}
\frac{-c^2}{2d} - \eta & 0 & 0 & 0\\
\\
0 & \frac{c(b^2 - d)}{2ad} - \eta & \frac{bc}{a\sqrt{d}} & \frac{-b}{4\sqrt{a}}\\
\\
0 & \frac{bc}{a\sqrt{d}} & \frac{c(d - b^2)}{2ad} - \eta & \frac{b^2}{4\sqrt{ad}}\\
\\
0 & \frac{b}{4\sqrt{a}} & \frac{-b^2}{4\sqrt{ad}} & - \eta
\end{bmatrix}.$$
Since $D$ is a derivation of $\mathfrak{g}_4$, the rule below holds
\begin{eqnarray*}
& [D(v_1), v_4] + [v_1, D(v_4)] = D[v_1, v_4] \\ \Leftrightarrow & \left( \frac{- c^2}{2d} - \eta\right)\left[v_1, v_4\right] + \left[ v_1, \frac{-b}{4\sqrt{a}}v_2 + \frac{b^2}{4\sqrt{ad}}v_3 - \eta v_4\right] = \sqrt{a}D(v_2)\\
\Leftrightarrow & \left( \frac{-c^2}{2d} - 2\eta\right)\left[v_1, v_4\right] - \frac{b}{4\sqrt{a}} \left[ v_1, v_2\right] + \frac{b^2}{4\sqrt{ad}}\left[v_1, v_3\right] = \sqrt{a}D(v_2)
\end{eqnarray*}
By carefully developing the last equation above, we obtain the following equality
\begin{eqnarray*}
& \left(\frac{-c^2}{2d} - 2\eta\right)\sqrt{a}v_2 - \frac{b\sqrt{d}}{4a\sqrt{a}}v_3 - \frac{b^2}{4a\sqrt{a}}v_2 - \frac{b^3}{4a\sqrt{ad}}v_3 - \frac{b^4}{4a\sqrt{a}d}v_2 \\
& = \left(\frac{c(b^2 - d)}{2ad} - \eta\right)\sqrt{a}v_2 + \frac{bc}{\sqrt{ad}}v_3 + \frac{b}{4\sqrt{a}}v_4
\end{eqnarray*}
By identifying the components $v_4$, we get $b = 0$. Now, the same reasoning as in the previous theorem shows that $\eta = 0$ and $c = 0$, which is a contradiction.
\end{proof}
The Kobayashi-Nomizu connection $\nabla^1$ of $(\mathfrak{g}_4, S_A^{+}, J)$ is given by
\begin{multicols}{3}
	$\nabla^1_{v_1} v_1 = 0$\par
	$\nabla^1_{v_1} v_2 = \frac{b^2 + d}{2a\sqrt{d}}v_3$\par
	$\nabla^1_{v_1} v_3 = \frac{-b^2 - d}{2a\sqrt{d}}v_2$\par
	$\nabla^1_{v_1} v_4 = 0$\par
	$\nabla^1_{v_2} v_1 = \frac{-b}{a}v_2 + \frac{b^2 - d}{2a\sqrt{d}}v_3$\par
	$\nabla^1_{v_2} v_2 = \frac{b}{a}v_1$\par
	$\nabla^1_{v_2} v_3 = \frac{-b^2 + d}{2a\sqrt{d}}v_1$\par
	$\nabla^1_{v_2} v_4 = 0$\par
	$\nabla^1_{v_3} v_1 = \frac{b^2 - d}{2a\sqrt{d}}v_2 + \frac{b}{a}v_3$\par
	$\nabla^1_{v_3} v_2 = \frac{-b^2 + d}{2a\sqrt{d}}v_1$\par
	$\nabla^1_{v_3} v_3 = \frac{-b}{a}v_1$\par
	$\nabla^1_{v_3} v_4 = 0$\par
	$\nabla^1_{v_4} v_1 = -\sqrt{a}v_2$\par
	$\nabla^1_{v_4} v_2 = 0$\par
	$\nabla^1_{v_4} v_3 = 0$\par
	$\nabla^1_{v_4} v_4 = 0$
\end{multicols}
In order to find to Ricci operator associated with $\nabla^1$, we need the following curvatures (note that the calculations are close to those related to $\nabla^0$, the only terms that changed are $\nabla^1_{v_4}v_1$ and $\nabla^1_{v_4}v_2$, so the only curvatures that are changed are the last three curvatures)
\begin{multicols}{2}
	$R^1_{v_2v_1}v_2 = \frac{b^4 - 2b^2d - 3d^2}{4a^2d}v_1$\par
	$R^1_{v_3v_1}v_3 = \frac{-3b^4 - 2b^2d + d^2}{4a^2d}v_1$\par
	$R^1_{v_4v_1}v_4 = 0$\par
	$R^1_{v_1v_2}v_1 = \frac{-2b^2d - 3d^2 + b^4}{4a^2d}v_2 + \frac{b^3 + bd}{a^2\sqrt{d}}v_3$\par
	$R^1_{v_3v_2}v_3 = \frac{b^4 + 2b^2d + d^2}{4a^2d}v_2$\par
	$R^1_{v_4v_2}v_4 = 0$\par
	$R^1_{v_1v_3}v_1 = \frac{b^3 + bd}{a^2\sqrt{d}}v_2 + \left(\frac{-3b^4 + 2b^2d + d^2}{4a^2d} - \frac{b^2}{a^2}\right)v_3$\par
	$R^1_{v_2v_3}v_2 = \left(\frac{b^4 - 2b^2d + d^2}{4a^2d} + \frac{b^2}{a^2}\right)v_3$\par
	$R^1_{v_4v_3}v_4 = 0$\par
	$R^1_{v_1v_4}v_1 = \frac{-b}{\sqrt{a}}v_2 + \frac{b^2}{\sqrt{ad}}v_3$\par
	$R^1_{v_2v_4}v_2 = \frac{-b}{\sqrt{a}}v_2$\par
	$R^1_{v_3v_4}v_3 = \frac{b}{\sqrt{a}}v_2$
\end{multicols}
The Ricci operator $\operatorname{Ric}^1$ of $\nabla^1$ is represented in the basis $\mathscr{B}$ by the following matrix
$$\operatorname{Ric}^1 = \begin{bmatrix}
\frac{-c^2}{2d} & 0 & 0 & 0\\
\\
0 & \frac{c(b^2 - d)}{2ad} & \frac{bc}{a\sqrt{d}} & \frac{-b}{\sqrt{a}}\\
\\
0 & \frac{bc}{a\sqrt{d}} & \frac{c(d - b^2)}{2ad} & \frac{b^2}{\sqrt{ad}}\\
\\
0 & 0 & 0 & 0
\end{bmatrix}.$$
\begin{theorem}
$\left(\mathfrak{g}_4, S_A^{+}, J\right)$ is not a first kind algebraic Ricci soliton associated to the connection $\nabla^1$.
\end{theorem}
\begin{proof}
Suppose that $\operatorname{Ric}^1 = \eta I_4 + D$, where $\eta \in \mathbb{R}$ and $D$ is a derivation of $\mathfrak{g}_4$. Then we obtain that 
$$D = \operatorname{Ric}^1 - \eta I_4 = \begin{bmatrix}
\frac{-c^2}{2d} - \eta & 0 & 0 & 0\\
\\
0 & \frac{c(b^2 - d)}{2ad} - \eta & \frac{bc}{a\sqrt{d}} & \frac{-b}{\sqrt{a}}\\
\\
0 & \frac{bc}{a\sqrt{d}} & \frac{c(d - b^2)}{2ad} - \eta & \frac{b^2}{\sqrt{ad}}\\
\\
0 & 0 & 0 & - \eta
\end{bmatrix}.$$
Given that $D$ is a derivation of $\mathfrak{g}_4$, it satisfies the following condition
\begin{eqnarray*}
& [D(v_1), v_4] + [v_1, D(v_4)] = D[v_1, v_4] \\ \Leftrightarrow & \left( \frac{- c^2}{2d} - \eta\right)\left[v_1, v_4\right] + \left[ v_1, \frac{-b}{\sqrt{a}}v_2 + \frac{b^2}{\sqrt{ad}}v_3 - \eta v_4\right] = \sqrt{a}D(v_2)\\
\Leftrightarrow & \left( \frac{-c^2}{2d} - 2\eta\right)\left[v_1, v_4\right] - \frac{b}{\sqrt{a}} \left[ v_1, v_2\right] + \frac{b^2}{\sqrt{ad}}\left[v_1, v_3\right] = \sqrt{a}D(v_2)
\end{eqnarray*}
By carefully developing the last equation above, we obtain the following equality
\begin{eqnarray*}
& \left(\frac{-c^2}{2d} - 2\eta\right)\sqrt{a}v_2 - \frac{b\sqrt{d}}{a\sqrt{a}}v_3 - \frac{b^2}{a\sqrt{a}}v_2 - \frac{b^3}{a\sqrt{ad}}v_3 - \frac{b^4}{a\sqrt{a}d}v_2 \\
& = \left(\frac{c(b^2 - d)}{2ad} - \eta\right)\sqrt{a}v_2 + \frac{bc}{\sqrt{ad}}v_3
\end{eqnarray*}
By identifying the components $v_3$, we get
\begin{eqnarray*}
\frac{-b\sqrt{d}}{a\sqrt{a}} - \frac{b^3}{a\sqrt{ad}} = \frac{bc}{\sqrt{ad}} &\Leftrightarrow& \frac{-bd - b^3}{a\sqrt{ad}} = \frac{abc}{a\sqrt{ad}}\\
&\Leftrightarrow& -bd - b^3 = abc\\
&\Leftrightarrow& -b(ac - b^2) - b^3 = abc\\
&\Leftrightarrow& -abc = abc \Leftrightarrow abc = 0
\end{eqnarray*}
Since $A = \begin{bmatrix}
a & b\\
b & c
\end{bmatrix}$ is symmetric positive definite, we have $a > 0$ and $c \neq 0$. Thus $b = 0$. Next, the same reasoning as in theorem \ref{th1} shows that $\eta = 0$ and $c = 0$, which is a contradiction.
\end{proof}
We obtain that $\widetilde{\operatorname{Ric}}^1$ is represented in the basis $\mathscr{B}$ by
$$\widetilde{\operatorname{Ric}}^1 = \begin{bmatrix}
\frac{-c^2}{2d} & 0 & 0 & 0\\
\\
0 & \frac{c(b^2 - d)}{2ad} & \frac{bc}{a\sqrt{d}} & \frac{-b}{2\sqrt{a}}\\
\\
0 & \frac{bc}{a\sqrt{d}} & \frac{c(d - b^2)}{2ad} & \frac{b^2}{2\sqrt{ad}}\\
\\
0 & \frac{b}{2\sqrt{a}} & \frac{-b^2}{2\sqrt{ad}} & 0
\end{bmatrix}.$$
\begin{theorem}
$\left(\mathfrak{g}_4, S_A^{+}, J\right)$ is not a second kind algebraic Ricci soliton associated to the connection $\nabla^1$.
\end{theorem}
\begin{proof}
Assume that $\widetilde{\operatorname{Ric}}^1 = \eta I_4 + D$, where $\eta \in \mathbb{R}$ and $D$ is a derivation of $\mathfrak{g}_4$. Thus 
$$D = \widetilde{\operatorname{Ric}}^1 - \eta I_4 = \begin{bmatrix}
\frac{-c^2}{2d} - \eta & 0 & 0 & 0\\
\\
0 & \frac{c(b^2 - d)}{2ad} - \eta & \frac{bc}{a\sqrt{d}} & \frac{-b}{2\sqrt{a}}\\
\\
0 & \frac{bc}{a\sqrt{d}} & \frac{c(d - b^2)}{2ad} - \eta & \frac{b^2}{2\sqrt{ad}}\\
\\
0 & \frac{b}{2\sqrt{a}} & \frac{-b^2}{2\sqrt{ad}} & - \eta
\end{bmatrix}.$$
As $D$ is a derivation of $\mathfrak{g}_4$, the rule below is satisfied
\begin{eqnarray*}
& [D(v_1), v_4] + [v_1, D(v_4)] = D[v_1, v_4] \\ \Leftrightarrow & \left( \frac{- c^2}{2d} - \eta\right)\left[v_1, v_4\right] + \left[ v_1, \frac{-b}{2\sqrt{a}}v_2 + \frac{b^2}{2\sqrt{ad}}v_3 - \eta v_4\right] = \sqrt{a}D(v_2)\\
\Leftrightarrow & \left( \frac{-c^2}{2d} - 2\eta\right)\left[v_1, v_4\right] - \frac{b}{2\sqrt{a}} \left[ v_1, v_2\right] + \frac{b^2}{2\sqrt{ad}}\left[v_1, v_3\right] = \sqrt{a}D(v_2)
\end{eqnarray*}
By carefully developing the last equation above, we obtain the following equality
\begin{eqnarray*}
& \left(\frac{-c^2}{2d} - 2\eta\right)\sqrt{a}v_2 - \frac{b\sqrt{d}}{2a\sqrt{a}}v_3 - \frac{b^2}{2a\sqrt{a}}v_2 - \frac{b^3}{2a\sqrt{ad}}v_3 - \frac{b^4}{2a\sqrt{a}d}v_2 \\
& = \left(\frac{c(b^2 - d)}{2ad} - \eta\right)\sqrt{a}v_2 + \frac{bc}{\sqrt{ad}}v_3 + \frac{b}{2}v_4
\end{eqnarray*}
By identifying the components $v_4$, we get $b = 0$. Now, the same reasoning as in theorem \ref{th1} shows that $\eta = 0$ and $c = 0$, which is a contradiction.
\end{proof}
\subsubsection{Case of the Lorentzian inner product $S_A^{-}$}
We consider $S_A^{-} = \begin{bmatrix}
-1 & 0 & 0 & 0\\
0 & 1 & 0 & 0\\
0 & 0 & a & b\\
0 & 0 & b & c
\end{bmatrix}$ where $A = \begin{bmatrix}
a & b\\
b & c
\end{bmatrix}$ is symmetric positive definite. We set $d = ac - b^2$; an orthonormal basis $\mathscr{B} = \left\lbrace v_1, v_2, v_3, v_4\right\rbrace$ of the Lorentzian Lie algebra $\left(\mathfrak{g}_4, S_A^{-}\right)$ is given by the following
$$v_1 = e_2, \qquad v_2 = \frac{1}{\sqrt{a}}e_3, \qquad v_3 = \frac{-b}{\sqrt{ad}}e_3 + \sqrt{\frac{a}{d}}e_4, \qquad v_4 = e_1.$$
The bracket is given in the basis $\mathscr{B}$ by
$$[v_1, v_4] = -\sqrt{a}v_2, \qquad [v_2, v_4] = \frac{-\sqrt{d}}{a}v_3 - \frac{b}{a}v_2, \qquad [v_3, v_4] = \frac{b}{a}v_3 + \frac{b^2}{a\sqrt{d}}v_2.$$
By \cite{aitbenhaddou2026lorentzian}, the Levi-Civita connection of $\left(\mathfrak{g}_4, S_A^{-}\right)$ is described by
\begin{multicols}{2}
$\nabla_{v_1} v_1 = \nabla_{v_4} v_4 = 0$\par
$\nabla_{v_2} v_2 = -\nabla_{v_3} v_3 = \frac{-b}{a}v_4$\par
$\nabla_{v_1} v_2 = \nabla_{v_2} v_1 = \frac{-\sqrt{a}}{2}v_4$\par
$\nabla_{v_1} v_3 = \nabla_{v_3} v_1 = 0$\par
$\nabla_{v_1} v_4= -\nabla_{v_4} v_1 = \frac{-\sqrt{a}}{2}v_2$\par
$\nabla_{v_2} v_3 = \nabla_{v_3} v_2 = \frac{b^2 - d}{2a\sqrt{d}}v_4$\par
$\nabla_{v_2} v_4 = \frac{-\sqrt{a}}{2}v_1 - \frac{b}{a}v_2 +  \frac{b^2 - d}{2a\sqrt{d}}v_3$\par
$\nabla_{v_4} v_2 = \frac{-\sqrt{a}}{2}v_1 + \frac{b^2 + d}{2a\sqrt{d}}v_3$\par
$\nabla_{v_3} v_4 = \frac{b^2 - d}{2a\sqrt{d}}v_2 + \frac{b}{a}v_3$\par
$\nabla_{v_4} v_3 = \frac{-b^2 - d}{2a\sqrt{d}}v_2$
\end{multicols}
The terms $\nabla_{v_i} (J)v_j$ are given by the following equalities
\begin{multicols}{2}
	$\nabla_{v_1} (J)v_1 = 0$\par
	$\nabla_{v_1} (J)v_2 = -\sqrt{a}v_4$\par
	$\nabla_{v_1} (J)v_3 = 0$\par
	$\nabla_{v_1} (J)v_4 = \sqrt{a}v_2$\par
	$\nabla_{v_2} (J)v_1 = -\sqrt{a}v_4$\par
	$\nabla_{v_2} (J)v_2 = \frac{-2b}{a}v_4$\par
	$\nabla_{v_2} (J)v_3 = \frac{b^2 - d}{a\sqrt{d}}v_4$\par
	$\nabla_{v_2} (J)v_4 = \sqrt{a}v_1 + \frac{2b}{a}v_2 -  \frac{b^2 - d}{a\sqrt{d}}v_3$\par
	$\nabla_{v_3} (J)v_1 = 0$\par
	$\nabla_{v_3} (J)v_2 = \frac{b^2 - d}{a\sqrt{d}}v_4$\par
	$\nabla_{v_3} (J)v_3 =  \frac{2b}{a}v_4$\par
	$\nabla_{v_3} (J)v_4 = \frac{d - b^2}{a\sqrt{d}}v_2 - \frac{2b}{a}v_3$\par
	$\nabla_{v_4} (J)v_1 = 0$\par
	$\nabla_{v_4} (J)v_2 = 0$\par
	$\nabla_{v_4} (J)v_3 = 0$\par
	$\nabla_{v_4} (J)v_4 = 0$\par
\end{multicols}
The nonvanishing terms in the canonical connection $\nabla^0$ of $\left(\mathfrak{g}_4, S_A^{-}\right)$ are
$$\nabla^0_{v_4} v_1 = \frac{\sqrt{a}}{2}v_2, \qquad \nabla^0_{v_4} v_2 = \frac{-\sqrt{a}}{2}v_1 + \frac{b^2 + d}{2a\sqrt{d}}v_3, \qquad \nabla^0_{v_4} v_3 = \frac{-b^2 - d}{2a\sqrt{d}}v_2.$$
\begin{theorem}\label{th2}
The Lorentzian Lie algebra $\left(\mathfrak{g}_4, S_A^{-}, J\right)$ is flat with respect to the canonical connection $\nabla^0$.
\end{theorem}
\begin{proof}
The curvature tensor $R^0$ of $\nabla^0$ is given by
$$R^0_{v_iv_j}v_k = \nabla^0_{[v_i, v_j]}v_k - \nabla^0_{v_i}\nabla^0_{v_j}v_k + \nabla^0_{v_j}\nabla^0_{v_i}v_k.$$
The bracket $[v_i, v_j]$ is expressed in terms of $v_2$ and $v_3$ and we have
$$\nabla^0_{v_2}v_k = \nabla^0_{v_3}v_k = 0 \quad \forall k = 1, \ldots, 4.$$
So the first term $\nabla^0_{[v_i, v_j]}v_k$ is always zero. Next, for the term $\nabla^0_{v_j}v_k$ to be nonzero, $j$ must equal $4$. Again, for the term $\nabla^0_{v_i}\nabla^0_{v_j}v_k$ to be nonzero, $i$ must equal $4$. So, $i = j = 4$. The same reasoning applies to the other term $\nabla^0_{v_j}\nabla^0_{v_i}v_k$. But when $i = j$, we have $R^0_{v_iv_j}v_k = 0$.
\end{proof}
The nonvanishing terms in the Kobayashi-Nomizu connection $\nabla^1$ of $(\mathfrak{g}_4, S_A^{-}, J)$ are
$$\nabla^1_{v_4} v_1 = \sqrt{a}v_2, \qquad \nabla^1_{v_4} v_2 = \frac{b}{a}v_2 + \frac{\sqrt{d}}{a}v_3, \qquad \nabla^1_{v_4} v_3 = \frac{-b^2}{a\sqrt{d}}v_2 - \frac{b}{a}v_3.$$
\begin{theorem}
The Lorentzian Lie algebra $\left(\mathfrak{g}_4, S_A^{-}, J\right)$ is flat with respect to the Kobayashi-Nomizu connection $\nabla^1$.
\end{theorem}
\begin{proof}
Similar to the proof of theorem \ref{th2}.
\end{proof}
\subsection{The Lorentzian Lie algebra $\left(\mathfrak{g}_4, S_A\right)$}
The Lorentzian inner product $S_A$ is given by $S_A = \begin{bmatrix}
1 & 0 & 0 & 0\\
0 & 1 & 0 & 0\\
0 & 0 & a & b\\
0 & 0 & b & c
\end{bmatrix}$, where the matrix $A = \begin{bmatrix}
a & b\\
b & c
\end{bmatrix}$ is Lorentzian. Consider the following orthonormal basis of $\left(\mathfrak{g}_4, S_A\right)$
$$v_1 = e_1, \qquad v_2 = e_2, \qquad v_3 = \alpha e_3 + \beta e_4, \qquad v_4 = \mu e_3 + \nu e_4$$
where the vector $v_4$ is timelike. This means that
$$S_A(v_4, v_4) = a\mu^2 + 2b\mu\nu + c\nu^2 = -1.$$
Since $\mathscr{B} = \left\lbrace v_1, v_2, v_3, v_4\right\rbrace$ is a basis of $\mathfrak{g_4}$, then $d := \operatorname{det}\left\lbrace v_1, v_2, v_3, v_4\right\rbrace = \alpha\nu - \beta\mu \neq 0$. We have
\begin{equation} \label{system1}
\left\lbrace\begin{array}{lll}
\nu v_3 = \alpha\nu e_3 + \beta\nu e_4 \\
\beta v_4 = \beta\mu e_3 + \beta\nu e_4 
\end{array}\right.
\end{equation}
By subtracting the two equations of the system (\ref{system1}), we obtain that $\nu v_3 - \beta v_4 = de_3$.
Similarly, we have the following system
\begin{equation} \label{system2}
\left\lbrace\begin{array}{lll}
\mu v_3 = \alpha\mu e_3 + \beta\mu e_4 \\
\alpha v_4 = \alpha\mu e_3 + \alpha\nu e_4 
\end{array}\right.
\end{equation}
By subtracting the two equations of the system (\ref{system2}), we obtain that $\mu v_3 - \alpha v_4 = -de_4$. Hence
$$\left\lbrace\begin{array}{lll}
e_3 = \frac{\nu}{d}v_3 - \frac{\beta}{d}v_4 \\
e_4 = \frac{-\mu}{d}v_3 + \frac{\alpha}{d}v_4
\end{array}\right.$$
Now one can obtain that the bracket in the basis $\mathscr{B}$ is described by
$$\left[v_1, v_2\right] = \frac{\nu}{d}v_3 - \frac{\beta}{d}v_4, \qquad
\left[v_1, v_3\right] = \frac{-\alpha\mu}{d}v_3 + \frac{\alpha^2}{d}v_4, \qquad
\left[v_1, v_4\right] = \frac{-\mu^2}{d}v_3 + \frac{\alpha\mu}{d}v_4.$$
By \cite{aitbenhaddou2026lorentzian}, the Levi-Civita connection of $\left(\mathfrak{g}_4, S_A\right)$ is given by
\begin{multicols}{2}
$\nabla_{v_1} v_1 = \nabla_{v_2} v_2 = 0$\par
$\nabla_{v_3} v_3 = \nabla_{v_4} v_4 = \frac{-\alpha\mu}{d}v_1$\par
$\nabla_{v_1} v_2 = \frac{\nu}{2d}v_3 - \frac{\beta}{2d}v_4$\par
$\nabla_{v_2} v_1 = \frac{-\nu}{2d}v_3 + \frac{\beta}{2d}v_4$\par
$\nabla_{v_1} v_3 = \frac{-\nu}{2d}v_2 - \frac{\mu^2 - \alpha^2}{2d}v_4$\par
$\nabla_{v_3} v_1 = \frac{-\nu}{2d}v_2 + \frac{\alpha\mu}{d}v_3 - \frac{\alpha^2 + \mu^2}{2d}v_4$\par
$\nabla_{v_1} v_4 = \frac{-\beta}{2d}v_2 + \frac{\alpha^2 - \mu^2}{2d}v_3 $\par
$\nabla_{v_4} v_1 = \frac{-\beta}{2d}v_2 + \frac{\alpha^2 + \mu^2}{2d}v_3 - \frac{\alpha\mu}{d}v_4 $\par
$\nabla_{v_2} v_3 = \nabla_{v_3} v_2 = \frac{\nu}{2d}v_1 $\par
$\nabla_{v_2} v_4 = \nabla_{v_4} v_2 = \frac{\beta}{2d}v_1 $\par
$\nabla_{v_3} v_4 = \nabla_{v_4} v_3 = -\frac{\alpha^2 + \mu^2}{2d}v_1$
\end{multicols}
The terms $\nabla_{v_i} (J)v_j$ are described by the following equalities
\begin{multicols}{2}
	$\nabla_{v_1} (J)v_1 = 0$\par
	$\nabla_{v_1} (J)v_2 = \frac{-\beta}{d}v_4$\par
	$\nabla_{v_1} (J)v_3 = \frac{\alpha^2 - \mu^2}{d}v_4$\par
	$\nabla_{v_1} (J)v_4 = \frac{\beta}{d}v_2 - \frac{\alpha^2 - \mu^2}{d}v_3$\par
	$\nabla_{v_2} (J)v_1 = \frac{\beta}{d}v_4$\par
	$\nabla_{v_2} (J)v_2 = 0$\par
	$\nabla_{v_2} (J)v_3 = 0$\par
	$\nabla_{v_2} (J)v_4 = \frac{-\beta}{d}v_1$\par
	$\nabla_{v_3} (J)v_1 = -\frac{\alpha^2 + \mu^2}{d}v_4$\par
	$\nabla_{v_3} (J)v_2 = 0$\par
	$\nabla_{v_3} (J)v_3 =  0$\par
	$\nabla_{v_3} (J)v_4 = \frac{\alpha^2 + \mu^2}{d}v_1$\par
	$\nabla_{v_4} (J)v_1 = \frac{-2\alpha\mu}{d}v_4$\par
	$\nabla_{v_4} (J)v_2 = 0$\par
	$\nabla_{v_4} (J)v_3 = 0$\par
	$\nabla_{v_4} (J)v_4 = \frac{2\alpha\mu}{d}v_1$\par
\end{multicols}
The canonical connection $\nabla^0$ of $\left(\mathfrak{g}_4, S_A\right)$ is given by
\begin{multicols}{3}
	$\nabla^0_{v_1} v_1 = 0$\par
	$\nabla^0_{v_1} v_2 = \frac{\nu}{2d}v_3$\par
	$\nabla^0_{v_1} v_3 = \frac{-\nu}{2d}v_2$\par
	$\nabla^0_{v_1} v_4 = 0$\par
	$\nabla^0_{v_2} v_1 = \frac{-\nu}{2d}v_3$\par
	$\nabla^0_{v_2} v_2 = 0$\par
	$\nabla^0_{v_2} v_3 = \frac{\nu}{2d}v_1$\par
	$\nabla^0_{v_2} v_4 = 0$\par
	$\nabla^0_{v_3} v_1 = \frac{-\nu}{2d}v_2 + \frac{\alpha\mu}{d}v_3$\par
	$\nabla^0_{v_3} v_2 = \frac{\nu}{2d}v_1$\par
	$\nabla^0_{v_3} v_3 = \frac{-\alpha\mu}{d}v_1$\par
	$\nabla^0_{v_3} v_4 = 0$\par
	$\nabla^0_{v_4} v_1 = \frac{-\beta}{2d}v_2 + \frac{\alpha^2 + \mu^2}{2d}v_3$\par
	$\nabla^0_{v_4} v_2 = \frac{\beta}{2d}v_1$\par
	$\nabla^0_{v_4} v_3 = -\frac{\alpha^2 + \mu^2}{2d}v_1$\par
	$\nabla^0_{v_4} v_4 = 0$
\end{multicols}
In order to find to Ricci operator associated with $\nabla^0$, we need the following curvatures
\begin{multicols}{2}
	$R^0_{v_2v_1}v_2 = \left(\frac{-3\nu^2}{4d^2} + \frac{\beta^2}{2d^2}\right)v_1$\par
	$R^0_{v_3v_1}v_3 = \left(\frac{\alpha^4 - \alpha^2\mu^2}{2d^2} + \frac{\nu^2}{4d^2}\right)v_1$\par
	$R^0_{v_4v_1}v_4 = 0$\par
	$R^0_{v_1v_2}v_1 = \frac{-3\nu^2 + 2\beta^2}{4d^2}v_2 - \frac{\alpha^2\beta + \mu^2\beta - 2\alpha\mu\nu}{2d^2}v_3$\par
	$R^0_{v_3v_2}v_3 = \frac{\nu^2}{4d^2}v_2$\par
	$R^0_{v_4v_2}v_4 = 0$\par
	$R^0_{v_1v_3}v_1 = \frac{2\alpha\mu\nu - \alpha^2\beta}{2d^2}v_2 + \left(\frac{\alpha^4 - \alpha^2\mu^2}{2d^2} + \frac{\nu^2}{4d^2}\right)v_3$\par
	$R^0_{v_2v_3}v_2 = \frac{\nu^2}{4d^2}v_3$\par
	$R^0_{v_4v_3}v_4 = 0$\par
	$R^0_{v_1v_4}v_1 = \left(\frac{\mu^2\nu - \alpha\beta\mu}{2d^2} + \frac{\alpha^2\nu + \mu^2\nu}{4d^2}\right)v_2 + \left(\frac{\alpha^3\mu - \alpha\mu^3}{2d^2} + \frac{\beta\nu}{4d^2}\right)v_3$\par
	$R^0_{v_2v_4}v_2 = \frac{\beta\nu}{4d^2}v_3$\par
	$R^0_{v_3v_4}v_3 = \left(-\frac{\alpha^2\nu + \mu^2\nu}{4d^2} + \frac{\alpha\beta\mu}{2d^2}\right)v_2$
\end{multicols}
The Ricci operator $\operatorname{Ric}^0$ of $\nabla^0$ is represented in the basis $\mathscr{B}$ by the matrix below
$$\operatorname{Ric}^0 = \begin{bmatrix}
\frac{\alpha^4 + \beta^2 - \nu^2 - \alpha^2\mu^2}{2d^2} & 0 & 0 & 0\\
\\
0 & \frac{\beta^2 - \nu^2}{2d^2} & \frac{2\alpha\mu\nu - \alpha^2\beta}{2d^2} & \frac{\mu^2\nu}{2d^2}\\
\\
0 & \frac{2\alpha\mu\nu - \alpha^2\beta - \mu^2\beta}{2d^2} & \frac{\alpha^4 + \nu^2 - \alpha^2\mu^2}{2d^2} & \frac{\alpha^3\mu - \alpha\mu^3 + \beta\nu}{2d^2}\\
\\
0 & 0 & 0 & 0
\end{bmatrix}.$$
To simplify the expression of $\operatorname{Ric}^0$, we adopt the following notation:
\begin{multicols}{3}
$x = \frac{\alpha^4 + \beta^2 - \nu^2 - \alpha^2\mu^2}{2d^2}$\par
$y = \frac{\beta^2 - \nu^2}{2d^2}$\par
$z = \frac{2\alpha\mu\nu - \alpha^2\beta - \mu^2\beta}{2d^2}$\par
$u = \frac{2\alpha\mu\nu - \alpha^2\beta}{2d^2}$\par
$v = \frac{\alpha^4 + \nu^2 - \alpha^2\mu^2}{2d^2}$\par
$w = \frac{\mu^2\nu}{2d^2}$\par
$t = \frac{\alpha^3\mu - \alpha\mu^3 + \beta\nu}{2d^2}$
\end{multicols}
Hence the simple form of $\operatorname{Ric}^0$ is given by:
$$\operatorname{Ric}^0 = \begin{bmatrix}
x & 0 & 0 & 0\\
0 & y & u & w\\
0 & z & v & t\\
0 & 0 & 0 & 0
\end{bmatrix}.$$
\begin{theorem}\label{th3.25}
$\left(\mathfrak{g}_4, S_A, J\right)$ does not exhibit the properties of a first kind algebraic Ricci soliton associated to the connection $\nabla^0$.
\end{theorem}
\begin{proof}
Assume that $\left(\mathfrak{g}_4, S_A, J\right)$ is a first kind algebraic Ricci soliton associated to $\nabla^0$, then there exists a real number $\eta$ and a derivation $D$ of $\mathfrak{g}_4$ such that $\operatorname{Ric}^0 = \eta I_4 + D$. Then we obtain that 
$$D = \operatorname{Ric}^0 - \eta I_4 = \begin{bmatrix}
x - \eta & 0 & 0 & 0\\
0 & y - \eta & u & w\\
0 & z & v - \eta & t\\
0 & 0 & 0 & -\eta
\end{bmatrix}.$$
Let us analyze the following first condition
\begin{eqnarray*}
& [D(v_1), v_2] + [v_1, D(v_2)] = D[v_1, v_2] \\ \Leftrightarrow & \left[\left(x - \eta\right)v_1, v_2\right] + \left[v_1, (y - \eta)v_2 + zv_3\right] = D[v_1, v_2]\\
\Leftrightarrow & \left(x + y - 2\eta\right)\left[v_1, v_2\right] + z[v_1, v_3] = \frac{\nu}{d}D(v_3) - \frac{\beta}{d}D(v_4)
\end{eqnarray*}
By carefully developing the last equation above, we obtain the following equality
\begin{eqnarray*}
& \left(x + y - 2\eta\right)\frac{\nu}{d}v_3 - \left(x + y - 2\eta\right)\frac{\beta}{d}v_4 - \frac{\alpha\mu z}{d}v_3 + \frac{\alpha^2z}{d}v_4 \\
& = \frac{\nu u}{d}v_2 + \left(v - \eta\right)\frac{\nu}{d}v_3 - \frac{\beta w}{d}v_2 - \frac{\beta t}{d}v_3 + \frac{\beta}{d}\eta v_4
\end{eqnarray*}
By identifying the components of $v_i$ we obtain the following system
$$\left\lbrace\begin{array}{lll}
\nu u - \beta w = 0 \\
\\
(x + y - 2\eta)\nu - \alpha\mu z = (v - \eta)\nu - \beta t\\
\\
-(x + y - 2\eta)\beta + \alpha^2z = \beta\eta
\end{array}\right.$$
Similarly, the condition $[D(v_1), v_3] + [v_1, D(v_3)] = D[v_1, v_3]$ gives the following system
$$\left\lbrace\begin{array}{lll}
\alpha^2w - \alpha\mu u = 0\\
\\
-\alpha\mu(x + v - 2\eta) + \nu u = \alpha^2 t - \alpha\mu(v - \eta)\\
\\
\alpha^2(x + v - 2\eta) - \beta u = -\alpha^2\eta
\end{array}\right.$$
The same for the condition $[D(v_1), v_4] + [v_1, D(v_4)] = D[v_1, v_4]$, it gives the following system
$$\left\lbrace\begin{array}{lll}
-\mu^2 u + \alpha\mu w = 0\\
\\
-\mu^2(x - 2\eta) + \nu w - \alpha\mu t = -\mu^2(v - \eta) + \alpha\mu t\\
\\
\alpha\mu(x - 2\eta) - \beta w + \alpha^2 t = -\alpha\mu\eta
\end{array}\right.$$
Then the fact that $D$ is a derivation of $\mathfrak{g}_4$ with respect to the orthonormal basis $\left\lbrace v_1, v_2, v_3, v_4\right\rbrace$ gives the following equations
\begin{align}
\nu u - \beta w &= 0 \label{eq1} \\
\alpha^2w - \alpha\mu u &= 0 \label{eq2} \\
-\mu^2 u + \alpha\mu w &= 0 \label{eq3} \\
(x + y - 2\eta)\nu - \alpha\mu z &= (v - \eta)\nu - \beta t \label{eq4} \\
-(x + y - 2\eta)\beta + \alpha^2z &= \beta\eta \label{eq5} \\
-\alpha\mu(x + v - 2\eta) + \nu u &= \alpha^2 t - \alpha\mu(v - \eta) \label{eq6} \\
\alpha^2(x + v - 2\eta) - \beta u &= -\alpha^2\eta \label{eq7} \\
-\mu^2(x - 2\eta) + \nu w - \alpha\mu t &= -\mu^2(v - \eta) + \alpha\mu t \label{eq8} \\
\alpha\mu(x - 2\eta) - \beta w + \alpha^2 t &= -\alpha\mu\eta \label{eq9}
\end{align}
We distinguish four cases depending on whether $\alpha$ and $\mu$ are zero or nonzero:\\
\textbf{Case 1}. If $\alpha = \mu = 0$, this case is clearly impossible bacause $d = \alpha\nu - \beta\mu \neq 0$.\\
\textbf{Case 2}. If $\alpha \neq 0$ and $\mu = 0$, then $\nu \neq 0$ because $d = \alpha\nu \neq 0$. By equations (\ref{eq1}) and (\ref{eq2}), we have $w = u = 0$. Since $\mu = 0$, then $u = \frac{-\alpha^2\beta}{2d^2} = 0$, hence $\beta = 0$. We replace these elements in the nine equations above, we obtain that
$$x + y - 2\eta = v - \eta, \qquad z = 0, \qquad t = 0, \qquad x + v - 2\eta = - \eta.$$
Hence $ \eta = x + y - v = x + v$. Note that with the conditions as above, the elements $x$, $y$ and $v$ becomes
$$x = \frac{\alpha^4 - \nu^2}{2\alpha^2\nu^2}, \qquad y = \frac{-\nu^2}{2\alpha^2\nu^2}, \qquad v = \frac{\alpha^4 + \nu^2}{2\alpha^2\nu^2}.$$
Thus $\eta = \frac{-3\nu^2}{2\alpha^2\nu^2} = \frac{2\alpha^4}{2\alpha^2\nu^2}$. This is a contradiction.\\
\textbf{Case 3}. If $\alpha = 0$ and $\mu \neq 0$, then $\beta \neq 0$ because $d = -\beta\mu \neq 0$. Since $\alpha = 0$, then $u = 0$. By equation (\ref{eq1}), we have $w = 0$ which implies that $\nu = 0$. Since $\alpha = \nu = 0$, then $v = t = 0$. Note that $\nu = 0 \Longrightarrow y = \frac{\beta^2}{2d^2}$. By equations (\ref{eq5}) and (\ref{eq8}) we obtain that: 
$$\eta = x + y = x \Longrightarrow y = 0 \Longrightarrow \beta = 0.$$
This is a contradiction.\\
\textbf{Case 4}. If $\alpha, \mu \neq 0$, we have $\nu \neq 0$ or $\beta \neq 0$ because $d = \alpha\nu - \beta\mu \neq 0$. The equations (\ref{eq1}) and (\ref{eq2}) can be rewritten as
$$\begin{bmatrix}
\nu & -\beta\\
-\alpha\mu & \alpha^2
\end{bmatrix}\begin{bmatrix}
u\\
w
\end{bmatrix} = \begin{bmatrix}
0\\
0
\end{bmatrix}$$
We have $\operatorname{det}\begin{bmatrix}
\nu & -\beta\\
-\alpha\mu & \alpha^2
\end{bmatrix} = \alpha d \neq 0$, then $u = w = 0$. We see that $w = 0$ implies that $\nu = 0$, and then $\beta \neq 0$. Next, $\nu = 0 \Longrightarrow u = \frac{-\alpha^2\beta}{2d^2} = 0$, this implies that $\alpha = 0$ or $\beta = 0$. This is a contradiction.
\end{proof}
We can see that $\widetilde{\operatorname{Ric}}^0$ is represented in the basis $\mathscr{B}$ by the following matrix
$$\widetilde{\operatorname{Ric}}^0 = \begin{bmatrix}
\frac{\alpha^4 + \beta^2 - \nu^2 - \alpha^2\mu^2}{2d^2} & 0 & 0 & 0\\
\\
0 & \frac{\beta^2 - \nu^2}{2d^2} & \frac{4\alpha\mu\nu - 2\alpha^2\beta - \mu^2\beta}{4d^2} & \frac{\mu^2\nu}{4d^2}\\
\\
0 & \frac{4\alpha\mu\nu - 2\alpha^2\beta - \mu^2\beta}{4d^2} & \frac{\alpha^4 + \nu^2 - \alpha^2\mu^2}{2d^2} & \frac{\alpha^3\mu - \alpha\mu^3 + \beta\nu}{4d^2}\\
\\
0 & \frac{-\mu^2\nu}{4d^2} & -\frac{\alpha^3\mu - \alpha\mu^3 + \beta\nu}{4d^2} & 0
\end{bmatrix}.$$
To simplify the expression of $\widetilde{\operatorname{Ric}}^0$, we adopt the following notation:
\begin{multicols}{3}
$x = \frac{\alpha^4 + \beta^2 - \nu^2 - \alpha^2\mu^2}{2d^2}$\par
$y = \frac{\beta^2 - \nu^2}{2d^2}$\par
$z = \frac{4\alpha\mu\nu - 2\alpha^2\beta - \mu^2\beta}{4d^2}$\par
$u = \frac{-\mu^2\nu}{4d^2}$\par
$v = \frac{\alpha^4 + \nu^2 - \alpha^2\mu^2}{2d^2}$\par
$w = -\frac{\alpha^3\mu - \alpha\mu^3 + \beta\nu}{4d^2}$
\end{multicols}
Hence the simple form of $\widetilde{\operatorname{Ric}}^0$ is given by:
$$\widetilde{\operatorname{Ric}}^0 = \begin{bmatrix}
x & 0 & 0 & 0\\
0 & y & z & -u\\
0 & z & v & -w\\
0 & u & w & 0
\end{bmatrix}.$$
\begin{theorem}\label{th3.26}
$\left(\mathfrak{g}_4, S_A, J\right)$ does not exhibit the properties of a second kind algebraic Ricci soliton associated to the connection $\nabla^0$.
\end{theorem}
\begin{proof}
Assume that $\left(\mathfrak{g}_4, S_A, J\right)$ is a second kind algebraic Ricci soliton associated to $\nabla^0$, then there exists a real number $\eta$ and a derivation $D$ of $\mathfrak{g}_4$ such that $\widetilde{\operatorname{Ric}}^0 = \eta I_4 + D$. Then we obtain that 
$$D = \widetilde{\operatorname{Ric}}^0 - \eta I_4 = \begin{bmatrix}
x - \eta & 0 & 0 & 0\\
0 & y - \eta & z & -u\\
0 & z & v - \eta & -w\\
0 & u & w & -\eta
\end{bmatrix}.$$
Let us analyze the following first condition
\begin{eqnarray*}
	& [D(v_1), v_2] + [v_1, D(v_2)] = D[v_1, v_2] \\ \Leftrightarrow & \left[\left(x - \eta\right)v_1, v_2\right] + \left[v_1, \left(y - \eta\right)v_2 + zv_3 + uv_4\right] = D[v_1, v_2]\\
	\Leftrightarrow & \left(x - \eta\right)\left[v_1, v_2\right] + \left(y - \eta\right)\left[v_1, v_2\right] + z[v_1, v_3] + u[v_1, v_4] = \frac{\nu}{d}D(v_3) - \frac{\beta}{d}D(v_4)\\
	\Leftrightarrow & \left(x + y - 2\eta\right)\left[v_1, v_2\right] + z[v_1, v_3] + u[v_1, v_4] = \frac{\nu}{d}D(v_3) - \frac{\beta}{d}D(v_4)
\end{eqnarray*}
By carefully developing the last equation above, we obtain the following equality
\begin{eqnarray*}
& \left(x + y - 2\eta\right)\frac{\nu}{d}v_3 - \left(x + y - 2\eta\right)\frac{\beta}{d}v_4 - \frac{\alpha\mu z}{d}v_3 + \frac{\alpha^2z}{d}v_4 - \frac{\mu^2u}{d}v_3 + \frac{\alpha\mu u}{d}v_4 \\
& = \frac{\nu z}{d}v_2 + \left(v - \eta\right)\frac{\nu}{d}v_3 + \frac{\nu w}{d}v_4 + \frac{\beta u}{d}v_2 + \frac{\beta w}{d}v_3 + \frac{\beta}{d}\eta v_4
\end{eqnarray*}
By identifying the components of $v_i$ we obtain the following system
$$\left\lbrace\begin{array}{lll}
\nu z + \beta u = 0 \\
\\
(x + y - 2\eta)\nu - \alpha\mu z - \mu^2u = (v - \eta)\nu + \beta w\\
\\
-(x + y - 2\eta)\beta + \alpha^2 z + \alpha\mu u = \nu w + \beta\eta
\end{array}\right.$$
Similarly, the condition $[D(v_1), v_3] + [v_1, D(v_3)] = D[v_1, v_3]$ gives the following system
$$\left\lbrace\begin{array}{lll}
\alpha\mu z + \alpha^2u = 0\\
\\
-\alpha\mu(x + v - 2\eta) + \nu z - \mu^2w = -\alpha\mu(v - \eta) - \alpha^2 w\\
\\
\alpha^2(x + v - 2\eta) - \beta z + \alpha\mu w = -\alpha\mu w - \alpha^2\eta
\end{array}\right.$$
The same for the condition $[D(v_1), v_4] + [v_1, D(v_4)] = D[v_1, v_4]$, it gives the following system
$$\left\lbrace\begin{array}{lll}
\mu^2 z + \alpha\mu u = 0\\
\\
-\mu^2(x - 2\eta) - u\nu + \alpha\mu w = -\mu^2(v - \eta) -\alpha\mu w\\
\\
(x - 2\eta)\alpha\mu +\beta u - \alpha^2 w = -\mu^2 w - \alpha\mu\eta
\end{array}\right.$$
Then the fact that $D$ is a derivation of $\mathfrak{g}_4$ with respect to the orthonormal basis $\left\lbrace v_1, v_2, v_3, v_4\right\rbrace$ gives the following equations
\begin{align}
\nu z + \beta u &= 0 \label{eq11} \\
\alpha\mu z + \alpha^2u &= 0 \label{eq21} \\
\mu^2 z + \alpha\mu u &= 0 \label{eq31} \\
(x + y - 2\eta)\nu - \alpha\mu z - \mu^2u &= (v - \eta)\nu + \beta w \label{eq41} \\
-(x + y - 2\eta)\beta + \alpha^2 z + \alpha\mu u &= \nu w + \beta\eta \label{eq51} \\
-\alpha\mu(x + v - 2\eta) + \nu z - \mu^2 w &= -\alpha\mu(v - \eta) - \alpha^2 w \label{eq61} \\
\alpha^2(x + v - 2\eta) - \beta z + \alpha\mu w &= -\alpha\mu w - \alpha^2\eta \label{eq71} \\
-\mu^2(x - 2\eta) - u\nu + \alpha\mu w &= -\mu^2(v - \eta) -\alpha\mu w \label{eq81} \\
(x - 2\eta)\alpha\mu +\beta u - \alpha^2 w &= -\mu^2 w - \alpha\mu\eta \label{eq91}
\end{align}
We distinguish four cases depending on whether $\alpha$ and $\mu$ are zero or nonzero:\\
\textbf{Case 1}. If $\alpha = \mu = 0$, this case is clearly impossible bacause $d = \alpha\nu - \beta\mu \neq 0$.\\
\textbf{Case 2}. If $\alpha \neq 0$ and $\mu = 0$, then $\nu \neq 0$ because $d = \alpha\nu \neq 0$. By equations (\ref{eq11}) and (\ref{eq21}), we have $z = u = 0$. Since $\mu = 0$, then $z = \frac{-\alpha^2\beta}{2d^2} = 0$, hence $\beta = 0$. We replace these elements in the nine equations above, we obtain that
$$x + y - 2\eta = v - \eta, \qquad w = 0, \qquad x + v - 2\eta = -\eta.$$
Hence $ \eta = x + y - v = x + v$. Note that with the conditions as above, the elements $x, y$ and $v$ becomes
$$x = \frac{\alpha^4 - \nu^2}{2\alpha^2\nu^2}, \qquad y = \frac{-\nu^2}{2\alpha^2\nu^2}, \qquad v = \frac{\alpha^4 + \nu^2}{2\alpha^2\nu^2}.$$
Thus $\eta = \frac{-3\nu^2}{2\alpha^2\nu^2} = \frac{2\alpha^4}{2\alpha^2\nu^2}$. This is a contradiction.\\
\textbf{Case 3}. If $\alpha = 0$ and $\mu \neq 0$, then $\beta \neq 0$ because $d = -\beta\mu \neq 0$. Since $\alpha = 0$, then $z = \frac{-\mu^2\beta}{4d^2}$. By equation (\ref{eq31}), we have $z = 0$, which implies that $\mu = 0$ or $\beta = 0$. This is a contradiction.\\
\textbf{Case 4}. If $\alpha, \mu \neq 0$, we have $\nu \neq 0$ or $\beta \neq 0$ because $d = \alpha\nu - \beta\mu \neq 0$. The equations (\ref{eq11}) and (\ref{eq21}) can be rewritten as
$$\begin{bmatrix}
\nu & \beta\\
\alpha\mu & \alpha^2
\end{bmatrix}\begin{bmatrix}
z\\
u
\end{bmatrix} = \begin{bmatrix}
0\\
0
\end{bmatrix}$$
We have $\operatorname{det}\begin{bmatrix}
\nu & \beta\\
\alpha\mu & \alpha^2
\end{bmatrix} = \alpha d \neq 0$, then $z = u = 0$. Note that $u = 0 \Longrightarrow \nu = 0$. Next, we can see that
$$z = \frac{-\beta(\mu^2 + 2\alpha^2)}{4d^2} = 0.$$
Then $\beta = 0$ or $\mu^2 + 2\alpha^2 = 0$. This is a contradiction.
\end{proof}
The Kobayashi-Nomizu connection $\nabla^1$ of $(\mathfrak{g}_4, S_A, J)$ is given by
\begin{multicols}{3}
	$\nabla^1_{v_1} v_1 = 0$\par
	$\nabla^1_{v_1} v_2 = \frac{\nu}{2d}v_3$\par
	$\nabla^1_{v_1} v_3 = \frac{-\nu}{2d}v_2$\par
	$\nabla^1_{v_1} v_4 = \frac{\alpha\mu}{d}v_4$\par
	$\nabla^1_{v_2} v_1 = \frac{-\nu}{2d}v_3$\par
	$\nabla^1_{v_2} v_2 = 0$\par
	$\nabla^1_{v_2} v_3 = \frac{\nu}{2d}v_1$\par
	$\nabla^1_{v_2} v_4 = 0$\par
	$\nabla^1_{v_3} v_1 = \frac{-\nu}{2d}v_2 + \frac{\alpha\mu}{d}v_3$\par
	$\nabla^1_{v_3} v_2 = \frac{\nu}{2d}v_1$\par
	$\nabla^1_{v_3} v_3 = \frac{-\alpha\mu}{d}v_1$\par
	$\nabla^1_{v_3} v_4 = 0$\par
	$\nabla^1_{v_4} v_1 = \frac{\mu^2}{d}v_3$\par
	$\nabla^1_{v_4} v_2 = 0$\par
	$\nabla^1_{v_4} v_3 = 0$\par
	$\nabla^1_{v_4} v_4 = 0$
\end{multicols}
In order to find to Ricci operator associated with $\nabla^1$, we need the following curvatures
\begin{multicols}{2}
	$R^1_{v_2v_1}v_2 = \frac{-3\nu^2}{4d^2}v_1$\par
	$R^1_{v_3v_1}v_3 = \left(\frac{-\alpha^2\mu^2}{d^2} + \frac{\nu^2}{4d^2}\right)v_1$\par
	$R^1_{v_4v_1}v_4 = 0$\par
	$R^1_{v_1v_2}v_1 = \frac{-3\nu^2}{4d^2}v_2 + \frac{\alpha\mu\nu - \mu^2\beta}{d^2}v_3$\par
	$R^1_{v_3v_2}v_3 = \frac{\nu^2}{4d^2}v_2$\par
	$R^1_{v_4v_2}v_4 = 0$\par
	$R^1_{v_1v_3}v_1 = \frac{\alpha\mu\nu}{d^2}v_2 + \frac{\nu^2}{4d^2}v_3$\par
	$R^1_{v_2v_3}v_2 = \frac{\nu^2}{4d^2}v_3$\par
	$R^1_{v_4v_3}v_4 = 0$\par
	$R^1_{v_1v_4}v_1 = \frac{\mu^2\nu}{d^2}v_2$\par
	$R^1_{v_2v_4}v_2 = 0$\par
	$R^1_{v_3v_4}v_3 = \frac{-\alpha\mu^3}{d^2}v_3$
\end{multicols}
The Ricci operator $\operatorname{Ric}^1$ of $\nabla^1$ is represented in the basis $\mathscr{B}$ by the following matrix
$$\operatorname{Ric}^1 = \begin{bmatrix}
\frac{-2\alpha^2\mu^2 - \nu^2}{2d^2} & 0 & 0 & 0\\
\\
0 & \frac{-\nu^2}{2d^2} & \frac{\alpha\mu\nu}{d^2} & \frac{\mu^2\nu}{d^2}\\
\\
0 & \frac{\alpha\mu\nu - \mu^2\beta}{d^2} & \frac{\nu^2}{2d^2} & \frac{-\alpha\mu^3}{d^2}\\
\\
0 & 0 & 0 & 0
\end{bmatrix}.$$
\begin{theorem}
$\left(\mathfrak{g}_4, S_A, J\right)$ is a first kind algebraic Ricci soliton associated to the connection $\nabla^1$ if and only if $\alpha = \nu = 0$. In this case we have 
$$\operatorname{Ric}^1 = \begin{bmatrix}
0 & 0 & 0 & 0\\
0 & 0 & 0 & 0\\
0 & \frac{-1}{\beta} & 0 & 0\\
0 & 0 & 0 & 0
\end{bmatrix}$$
is a derivation of $\mathfrak{g}_4$ with respect to the corresponding orthonormal basis $\left\lbrace v_1, v_2, v_3, v_4\right\rbrace$.
\end{theorem}
\begin{proof}
Assume that $\left(\mathfrak{g}_4, S_A, J\right)$ is a first kind algebraic Ricci soliton associated to $\nabla^1$, then there exists a real number $\eta$ and a derivation $D$ of $\mathfrak{g}_4$ such that $\operatorname{Ric}^1 = \eta I_4 + D$. Then we obtain that 
$$D = \operatorname{Ric}^1 - \eta I_4 = \begin{bmatrix}
\frac{-2\alpha^2\mu^2 - \nu^2}{2d^2} - \eta & 0 & 0 & 0\\
\\
0 & \frac{-\nu^2}{2d^2} - \eta & \frac{\alpha\mu\nu}{d^2} & \frac{\mu^2\nu}{d^2}\\
\\
0 & \frac{\alpha\mu\nu - \mu^2\beta}{d^2} & \frac{\nu^2}{2d^2} - \eta & \frac{-\alpha\mu^3}{d^2}\\
\\
0 & 0 & 0 & -\eta
\end{bmatrix}.$$
We distinguish two cases
\begin{enumerate}
\item If $\nu = 0$, then $\beta, \mu \neq 0$ because $d = -\beta\mu \neq 0$. In this case the Ricci operator is given by
$$\operatorname{Ric}^1 = \begin{bmatrix}
\frac{-\alpha^2}{\beta^2} & 0 & 0 & 0\\
0 & 0 & 0 & 0\\
0 & \frac{-1}{\beta} & 0 & \frac{-\alpha\mu}{\beta^2}\\
0 & 0 & 0 & 0
\end{bmatrix}.$$
The derivation $D$ is given by
$$D = \operatorname{Ric}^1 - \eta I_4 = \begin{bmatrix}
\frac{-\alpha^2}{\beta^2} - \eta & 0 & 0 & 0\\
0 & -\eta & 0 & 0\\
0 & \frac{-1}{\beta} & -\eta & \frac{-\alpha\mu}{\beta^2}\\
0 & 0 & 0 & -\eta
\end{bmatrix}.$$
The bracket in the basis $\mathscr{B}$ is as follows
$$[v_1, v_2] = \frac{1}{\mu}v_4, \qquad [v_1, v_3] = \frac{\alpha}{\beta}v_3 - \frac{\alpha^2}{\beta\mu}v_4, \qquad [v_1, v_4] = \frac{\mu}{\beta}v_3 - \frac{\alpha}{\beta}v_4.$$
Given that $D$ is a derivation of $\mathfrak{g}_4$, is satisfies the condition
\begin{eqnarray*}
& [D(v_1), v_2] + [v_1, D(v_2)] = D[v_1, v_2]\\ \Leftrightarrow& [(\frac{-\alpha^2}{\beta^2} - \eta)v_1, v_2] + [v_1, -\eta v_2 - \frac{1}{\beta}v_3] = D[v_1, v_2]\\
\Leftrightarrow& (\frac{-\alpha^2}{\beta^2} - 2\eta)[v_1, v_2] - \frac{1}{\beta}[v_1, v_3] = \frac{1}{\mu}D(v_4)\\
\Leftrightarrow& (\frac{-\alpha^2}{\beta^2} - 2\eta)\frac{1}{\mu}v_4 - \frac{\alpha}{\beta^2}v_3 + \frac{\alpha^2}{\beta^2\mu}v_4 = \frac{-\alpha}{\beta^2}v_3 - \frac{\eta}{\mu}v_4
\end{eqnarray*}
By identifying the components of $v_4$ in the last equality as above, we obtain that $2\eta = \eta$ and then $\eta = 0$. Next, $D = \operatorname{Ric}^1$ satisfies the condition
\begin{eqnarray*}
& [D(v_1), v_3] + [v_1, D(v_3)] = D[v_1, v_3]\\ \Leftrightarrow& \frac{-\alpha^2}{\beta^2}[v_1, v_3] = \frac{-\alpha^2}{\beta\mu}D(v_4)\\
\Leftrightarrow& \frac{-\alpha^3}{\beta^3}v_3 + \frac{\alpha^4}{\beta^3\mu}v_4 = \frac{\alpha^3}{\beta^3}v_3 \Leftrightarrow \alpha = 0
\end{eqnarray*}
Since $\alpha = 0$, the Ricci operator $\operatorname{Ric}^1$ becomes
$$\operatorname{Ric}^1 = \begin{bmatrix}
0 & 0 & 0 & 0\\
0 & 0 & 0 & 0\\
0 & \frac{-1}{\beta} & 0 & 0\\
0 & 0 & 0 & 0
\end{bmatrix}.$$
Now, it is easy to check that $\operatorname{Ric}^1$ satisfies the remaining derivation conditions. An example of this case occurs when the Lorentzian inner product $S_A$ satisfies $a = -1$, $c = 1$ and $b = 0$; moreover $S_A$ is represented in the basis $\left\lbrace e_1, e_2, e_3, e_4\right\rbrace$ by
$$S_A = \begin{bmatrix}
1 & 0 & 0 & 0\\
0 & 1 & 0 & 0\\
0 & 0 & -1 & 0\\
0 & 0 & 0 & 1
\end{bmatrix}.$$
Then $\alpha = \nu = 0$ and $\beta = \mu = 1$. It follows that $S_A$ is a first kind algebraic Ricci soliton with respect to the Kobayashi-Nomizu connection $\nabla^1$, where $\eta = 0$ and 
$$D = \operatorname{Ric}^1 = \begin{bmatrix}
0 & 0 & 0 & 0\\
0 & 0 & 0 & 0\\
0 & -1 & 0 & 0\\
0 & 0 & 0 & 0
\end{bmatrix}$$
is a derivation of $\mathfrak{g}_4$ with respect to the corresponding orthonormal basis $\left\lbrace v_1, v_2, v_3, v_4\right\rbrace$.
\item If $\nu \neq 0$, let us first simplify the expression for $\operatorname{Ric}^1$ by introducing the following notation:
\begin{multicols}{3}
$x = \frac{-2\alpha^2\mu^2 - \nu^2}{2d^2}$\par
$y = \frac{-\nu^2}{2d^2}$\par
$z = \frac{\alpha\mu\nu - \mu^2\beta}{d^2}$\par
$u = \frac{\alpha\mu\nu}{d^2}$\par
$v = \frac{\nu^2}{2d^2}$\par
$w = \frac{\mu^2\nu}{d^2}$\par
$t = \frac{-\alpha\mu^3}{d^2}$
\end{multicols}
Hence the simple form of $\operatorname{Ric}^1$ is given by:
$$\operatorname{Ric}^1 = \begin{bmatrix}
x & 0 & 0 & 0\\
0 & y & u & w\\
0 & z & v & t\\
0 & 0 & 0 & 0
\end{bmatrix}.$$
The derivation $D$ is given by
$$D = \operatorname{Ric}^1 - \eta I_4 = \begin{bmatrix}
x - \eta & 0 & 0 & 0\\
0 & y - \eta & u & w\\
0 & z & v - \eta & t\\
0 & 0 & 0 & -\eta
\end{bmatrix}.$$
By the same reasoning as in Theorem \ref{th3.25}, the fact that $D$ is a derivation of $\mathfrak{g}_4$ with respect to the orthonormal basis $\left\lbrace v_1, v_2, v_3, v_4\right\rbrace$ yields the following equations:
\begin{align}
\nu u - \beta w &= 0 \label{eq12} \\
\alpha^2w - \alpha\mu u &= 0 \label{eq22} \\
-\mu^2 u + \alpha\mu w &= 0 \label{eq32} \\
(x + y - 2\eta)\nu - \alpha\mu z &= (v - \eta)\nu - \beta t \label{eq42} \\
-(x + y - 2\eta)\beta + \alpha^2z &= \beta\eta \label{eq52} \\
-\alpha\mu(x + v - 2\eta) + \nu u &= \alpha^2 t - \alpha\mu(v - \eta) \label{eq62} \\
\alpha^2(x + v - 2\eta) - \beta u &= -\alpha^2\eta \label{eq72} \\
-\mu^2(x - 2\eta) + \nu w - \alpha\mu t &= -\mu^2(v - \eta) + \alpha\mu t \label{eq82} \\
\alpha\mu(x - 2\eta) - \beta w + \alpha^2 t &= -\alpha\mu\eta \label{eq92}
\end{align}
We distinguish four cases depending on whether $\alpha$ and $\mu$ are zero or nonzero:\\
\textbf{Case 1}. If $\alpha = \mu = 0$, this case is clearly impossible bacause $d = \alpha\nu - \beta\mu \neq 0$.\\
\textbf{Case 2}. If $\alpha \neq 0$ and $\mu = 0$, then $w = u = z = t = 0$. By equations (\ref{eq42}) and (\ref{eq72}), we have
$$x + y - 2\eta = v - \eta, \qquad x + v - 2\eta = - \eta.$$
Hence $ \eta = x + y - v = x + v$. Since $\mu = 0$, the elements $x$, $y$ and $v$ becomes
$$x = \frac{-1}{2\alpha^2}, \qquad y = \frac{-1}{2\alpha^2}, \qquad v = \frac{1}{2\alpha^2}.$$
Thus $\eta = \frac{-3}{2\alpha^2} = 0$. This is a contradiction.\\
\textbf{Case 3}. If $\alpha = 0$ and $\mu \neq 0$, then $\beta \neq 0$ because $d = -\beta\mu \neq 0$. Since $\alpha = 0$, then $u = 0$. By equation (\ref{eq12}), we have $w = 0$ which implies that $\mu = 0$ or $\nu = 0$. This is a contradiction, since we are considering the case where $\mu$ and $\nu$ are both nonzero.\\
\textbf{Case 4}. If $\alpha, \mu \neq 0$. The equations (\ref{eq12}) and (\ref{eq22}) can be rewritten as
$$\begin{bmatrix}
\nu & -\beta\\
-\alpha\mu & \alpha^2
\end{bmatrix}\begin{bmatrix}
u\\
w
\end{bmatrix} = \begin{bmatrix}
0\\
0
\end{bmatrix}$$
We have $\operatorname{det}\begin{bmatrix}
\nu & -\beta\\
-\alpha\mu & \alpha^2
\end{bmatrix} = \alpha d \neq 0$, then $u = w = 0$. Since $\alpha, \mu, \nu \neq 0$, then $u$ and $w$ are nonzero. This is a contradiction.
\end{enumerate}
\end{proof}
We can see that $\widetilde{\operatorname{Ric}}^1$ is represented in the basis $\mathscr{B}$ by
$$\widetilde{\operatorname{Ric}}^1 = \begin{bmatrix}
\frac{-2\alpha^2\mu^2 - \nu^2}{2d^2} & 0 & 0 & 0\\
\\
0 & \frac{-\nu^2}{2d^2} & \frac{2\alpha\mu\nu - \mu^2\beta}{2d^2} & \frac{\mu^2\nu}{2d^2}\\
\\
0 & \frac{2\alpha\mu\nu - \mu^2\beta}{2d^2} & \frac{\nu^2}{2d^2} & \frac{-\alpha\mu^3}{2d^2}\\
\\
0 & \frac{-\mu^2\nu}{2d^2} & \frac{\alpha\mu^3}{2d^2} & 0
\end{bmatrix}.$$
To simplify the expression of $\widetilde{\operatorname{Ric}}^1$, we consider the following notation:
\begin{multicols}{3}
$x = \frac{-2\alpha^2\mu^2 - \nu^2}{2d^2}$\par
$y = \frac{-\nu^2}{2d^2}$\par
$z = \frac{2\alpha\mu\nu - \mu^2\beta}{2d^2}$\par
$u = \frac{-\mu^2\nu}{2d^2}$\par
$v = \frac{\nu^2}{2d^2}$\par
$w = \frac{\alpha\mu^3}{2d^2}$
\end{multicols}
Hence the simple form of $\widetilde{\operatorname{Ric}}^1$ is given by:
$$\widetilde{\operatorname{Ric}}^1 = \begin{bmatrix}
x & 0 & 0 & 0\\
0 & y & z & -u\\
0 & z & v & -w\\
0 & u & w & 0
\end{bmatrix}.$$
\begin{theorem}
$\left(\mathfrak{g}_4, S_A, J\right)$ does not exhibit the properties of a second kind algebraic Ricci soliton associated to the connection $\nabla^1$.
\end{theorem}
\begin{proof}
Assume that $\left(\mathfrak{g}_4, S_A, J\right)$ is a second kind algebraic Ricci soliton associated to $\nabla^1$, then there exists a real number $\eta$ and a derivation $D$ of $\mathfrak{g}_4$ such that $\widetilde{\operatorname{Ric}}^1 = \eta I_4 + D$. Then we obtain that 
$$D = \widetilde{\operatorname{Ric}}^1 - \eta I_4 = \begin{bmatrix}
x - \eta & 0 & 0 & 0\\
0 & y - \eta & z & -u\\
0 & z & v - \eta & -w\\
0 & u & w & -\eta
\end{bmatrix}.$$
By the same reasoning as in Theorem \ref{th3.26}, the fact that $D$ is a derivation of $\mathfrak{g}_4$ with respect to the orthonormal basis $\left\lbrace v_1, v_2, v_3, v_4\right\rbrace$ yields the following equations:
\begin{align}
\nu z + \beta u &= 0 \label{eq13} \\
\alpha\mu z + \alpha^2u &= 0 \label{eq23} \\
\mu^2 z + \alpha\mu u &= 0 \label{eq33} \\
(x + y - 2\eta)\nu - \alpha\mu z - \mu^2u &= (v - \eta)\nu + \beta w \label{eq43} \\
-(x + y - 2\eta)\beta + \alpha^2 z + \alpha\mu u &= \nu w + \beta\eta \label{eq53} \\
-\alpha\mu(x + v - 2\eta) + \nu z - \mu^2 w &= -\alpha\mu(v - \eta) - \alpha^2 w \label{eq63} \\
\alpha^2(x + v - 2\eta) - \beta z + \alpha\mu w &= -\alpha\mu w - \alpha^2\eta \label{eq73} \\
-\mu^2(x - 2\eta) - u\nu + \alpha\mu w &= -\mu^2(v - \eta) -\alpha\mu w \label{eq83} \\
(x - 2\eta)\alpha\mu +\beta u - \alpha^2 w &= -\mu^2 w - \alpha\mu\eta \label{eq93}
\end{align}
We distinguish four cases depending on whether $\alpha$ and $\mu$ are zero or nonzero:\\
\textbf{Case 1}. If $\alpha = \mu = 0$, this case is clearly impossible bacause $d = \alpha\nu - \beta\mu \neq 0$.\\
\textbf{Case 2}. If $\alpha \neq 0$ and $\mu = 0$, then $\nu \neq 0$ because $d = \alpha\nu \neq 0$. Since $\mu = 0$, then $z = u = w = 0$. By equations (\ref{eq43}) and (\ref{eq73}), we have
$$x + y - 2\eta = v - \eta, \qquad x + v - 2\eta = - \eta.$$
Hence $ \eta = x + y - v = x + v$. Since $\mu = 0$, the elements $x$, $y$ and $v$ becomes
$$x = \frac{-1}{2\alpha^2}, \qquad y = \frac{-1}{2\alpha^2}, \qquad v = \frac{1}{2\alpha^2}.$$
Thus $\eta = \frac{-3}{2\alpha^2} = 0$. This is a contradiction.\\
\textbf{Case 3}. If $\alpha = 0$ and $\mu \neq 0$, then $\beta \neq 0$ because $d = -\beta\mu \neq 0$. Since $\alpha = 0$, then $z = \frac{-\mu^2\beta}{2d^2}$. By equation (\ref{eq33}), we have $z = 0$, which implies that $\mu = 0$ or $\beta = 0$. This is a contradiction.\\
\textbf{Case 4}. If $\alpha, \mu \neq 0$, we have $\nu \neq 0$ or $\beta \neq 0$ because $d = \alpha\nu - \beta\mu \neq 0$. The equations (\ref{eq13}) and (\ref{eq23}) can be rewritten as
$$\begin{bmatrix}
\nu & \beta\\
\alpha\mu & \alpha^2
\end{bmatrix}\begin{bmatrix}
z\\
u
\end{bmatrix} = \begin{bmatrix}
0\\
0
\end{bmatrix}$$
We have $\operatorname{det}\begin{bmatrix}
\nu & \beta\\
\alpha\mu & \alpha^2
\end{bmatrix} = \alpha d \neq 0$, then $z = u = 0$. Note that $u = 0 \Longrightarrow \nu = 0$, and then $\beta \neq 0$. Since $\nu = 0$, then $y = v = 0$. By equation (\ref{eq43}), we obtain $w = 0$, which implies that $\alpha = 0$ or $\mu = 0$. This is a contradiction.
\end{proof}
It is well known that the Lie group $G_4$ admits a flat Lorentzian metric, although this metric was overlooked in the classification obtained in \cite{bokan2015lorentz}. This was proved in \cite{aitbenhaddou2012left,bajo2024classification}. In fact, the Lorentzian inner product inducing this flat Lorentzian metric is represented, with respect to the basis $\left\lbrace e_1, e_2, e_3, e_4\right\rbrace$ of $\mathfrak{g}_4$, by the following matrix:
$$S = \begin{bmatrix}
1 & 0 & 0 & 0\\
0 & 0 & 0 & -1\\
0 & 0 & 1 & 0\\
0 & -1 & 0 & 0
\end{bmatrix}.$$
\subsection{The Lorentzian Lie algebra $\left(\mathfrak{g}_4, S\right)$}
An orthonormal basis of $\left(\mathfrak{g}_4, S\right)$ is given by
$$v_1 = e_1, \qquad v_2 = e_3, \qquad v_3 = \frac{1}{\sqrt{2}}(e_2 - e_4), \qquad v_4 = \frac{1}{\sqrt{2}}(e_2 + e_4).$$
The bracket in the orthonormal basis $\mathscr{B} = \left\lbrace v_1, v_2, v_3, v_4\right\rbrace$ satisfies
$$[v_1, v_2] = \frac{1}{\sqrt{2}}(v_4 - v_3), \qquad [v_1, v_3] = [v_1, v_4] = \frac{1}{\sqrt{2}}v_2.$$
Hence, the structure constants of $\left(\mathfrak{g}_4, S\right)$ are:
\begin{multicols}{4}
$\xi_{123} = \frac{-1}{\sqrt{2}}$\par
$\xi_{213} = \frac{1}{\sqrt{2}}$\par
$\xi_{124} = \frac{-1}{\sqrt{2}}$\par
$\xi_{214} = \frac{1}{\sqrt{2}}$\par
$\xi_{132} = \frac{1}{\sqrt{2}}$\par
$\xi_{312} = \frac{-1}{\sqrt{2}}$\par
$\xi_{142} = \frac{1}{\sqrt{2}}$\par
$\xi_{412} = \frac{-1}{\sqrt{2}}$
\end{multicols}
Using (\ref{Levi}), we find that the Levi-Civita connection of $\left(\mathfrak{g}_4, S\right)$ is as follows
\begin{multicols}{3}
$\nabla_{v_i}v_i = 0,\ \forall i=1,\ldots,4$\par
$\nabla_{v_1}v_2 = \frac{1}{\sqrt{2}}(v_4 - v_3)$\par
$\nabla_{v_2}v_1 = 0$\par
$\nabla_{v_1}v_3 = \nabla_{v_1}v_4 = \frac{1}{\sqrt{2}}v_2$\par
$\nabla_{v_3}v_1 = \nabla_{v_4}v_1 = 0$\par
$\nabla_{v_2}v_3 = \nabla_{v_3}v_2 = 0$\par
$\nabla_{v_2}v_4 = \nabla_{v_4}v_2 = 0$\par
$\nabla_{v_3}v_4 = \nabla_{v_4}v_3=0$
\end{multicols}
It is straightforward to check that the curvature tensor $R$ is zero and then $\left(\mathfrak{g}_4, S\right)$ is flat. For example, we have
\begin{eqnarray*}
R_{v_1v_2}v_1 &=& \nabla_{[v_1, v_2]}v_1 - \nabla_{v_1}\nabla_{v_2}v_1 + \nabla_{v_2}\nabla_{v_1}v_1\\
&=& \frac{1}{\sqrt{2}}\nabla_{v_4}v_1 - \frac{1}{\sqrt{2}}\nabla_{v_3}v_1 = 0
\end{eqnarray*}
and the same reasoning applies to the other terms.

The nonzero terms in $\nabla_{v_i}(J)v_j$ are
$$\nabla_{v_1}(J)v_2 = \sqrt{2}v_4, \qquad \nabla_{v_1}(J)v_4 = -\sqrt{2}v_2.$$
The nonvanishing terms in $\nabla^0$ of $\left(\mathfrak{g}_4, S, J\right)$ are
$$\nabla^0_{v_1}v_2 = \frac{-1}{\sqrt{2}}v_3, \qquad \nabla^0_{v_1}v_3 = \frac{1}{\sqrt{2}}v_2.$$
\begin{theorem}
The Lorentzian Lie algebra $\left(\mathfrak{g}_4, S, J\right)$ is still flat with respect to $\nabla^0$.
\end{theorem}
\begin{proof}
The curvature $R^0$ is given by
$$R^0_{v_iv_j}v_k = \nabla^0_{[v_i, v_j]}v_k - \nabla^0_{v_i}\nabla^0_{v_j}v_k + \nabla^0_{v_j}\nabla^0_{v_i}v_k.$$
The bracket $[v_i, v_j]$ is expressed in terms of $v_2, v_3$ and $v_4$ and we have
$$\nabla^0_{v_2} = \nabla^0_{v_3} = \nabla^0_{v_4} = 0.$$
So the first term $\nabla^0_{[v_i, v_j]}v_k$ is always zero. Next, for the term $\nabla^0_{v_j}v_k$ to be nonzero, $j$ must equal $1$. Again, for the term $\nabla^0_{v_i}\nabla^0_{v_j}v_k$ to be nonzero, $i$ must equal $1$. So, $i = j = 1$. The same reasoning applies to the other term $\nabla^0_{v_j}\nabla^0_{v_i}v_k$. But when $i = j$, we have $R^0_{v_iv_j}v_k = 0$.
\end{proof}
The nonvanishing terms in $\nabla^1$ of $\left(\mathfrak{g}_4, S, J\right)$ are
$$\nabla^1_{v_1}v_2 = \frac{-1}{\sqrt{2}}v_3, \qquad \nabla^1_{v_1}v_3 = \frac{1}{\sqrt{2}}v_2, \qquad \nabla^1_{v_4}v_1 = \frac{-1}{\sqrt{2}}v_2.$$
In order to find to Ricci operator associated with $\nabla^1$, we need the following curvatures
\begin{multicols}{4}
	$R^1_{v_2v_1}v_2 = 0$\par
	$R^1_{v_3v_1}v_3 = 0$\par
	$R^1_{v_4v_1}v_4 = 0$\par
	$R^1_{v_1v_2}v_1 = \frac{-1}{2}v_2$\par
	$R^1_{v_3v_2}v_3 = 0$\par
	$R^1_{v_4v_2}v_4 = 0$\par
	$R^1_{v_1v_3}v_1 = 0$\par
	$R^1_{v_2v_3}v_2 = 0$\par
	$R^1_{v_4v_3}v_4 = 0$\par
	$R^1_{v_1v_4}v_1 = \frac{-1}{2}v_3$\par
	$R^1_{v_2v_4}v_2 = 0$\par
	$R^1_{v_3v_4}v_3 = 0$
\end{multicols}
The Ricci operator $\operatorname{Ric}^1$ of $\nabla^1$ is represented in the basis $\mathscr{B}$ by the following matrix
$$\operatorname{Ric}^1 = \begin{bmatrix}
0 & 0 & 0 & 0\\
0 & \frac{-1}{2} & 0 & 0\\
0 & 0 & 0 & \frac{-1}{2}\\
0 & 0 & 0 & 0
\end{bmatrix}.$$
Therefore, $\left(\mathfrak{g}_4, S, J\right)$ is not Ricci-flat and then not flat.
\begin{theorem}
$\left(\mathfrak{g}_4, S, J\right)$ does not exhibit the properties of a first kind algebraic Ricci soliton associated to the connection $\nabla^1$.
\end{theorem}
\begin{proof}
Assume that $\operatorname{Ric}^1 = cI_4 + D$, where $c \in \mathbb{R}$ and $D$ is a derivation of $\mathfrak{g}_4$. Thus 
$$D = \operatorname{Ric}^1 - cI_4 = \begin{bmatrix}
-c & 0 & 0 & 0\\
0 & \frac{-1}{2} - c & 0 & 0\\
0 & 0 & -c & \frac{-1}{2}\\
0 & 0 & 0 & -c
\end{bmatrix}.$$
Since $D$ is a derivation of $\mathfrak{g}_4$, it satisfies the rule
\begin{eqnarray*}
& [D(v_1), v_2] + [v_1, D(v_2)] = D[v_1, v_2]\\ \Leftrightarrow& -c[v_1, v_2] + (\frac{-1}{2} - c)[v_1, v_2] = D[v_1, v_2]\\
\Leftrightarrow& (\frac{-1}{2} - 2c)[v_1, v_2] = \frac{1}{\sqrt{2}}(D(v_4) - D(v_3))\\
\Leftrightarrow& (\frac{-1}{2} - 2c)\frac{1}{\sqrt{2}}(v_4 - v_3) = \frac{1}{\sqrt{2}}(-cv_4 + (\frac{-1}{2} + c)v_3)
\end{eqnarray*}
By identifying the components of $v_4$ and $v_3$ respectively, we obtain
$$\left\lbrace\begin{array}{lll}
\frac{-1}{2} - 2c = -c \Longrightarrow c = \frac{-1}{2} \\
\frac{1}{2} + 2c = \frac{-1}{2} + c \Longrightarrow c = -1
\end{array}\right.$$
This is a contradiction.
\end{proof}
We can see that $\widetilde{\operatorname{Ric}}^1$ is represented in the basis $\mathscr{B}$ by
$$\widetilde{\operatorname{Ric}}^1 = \begin{bmatrix}
0 & 0 & 0 & 0\\
0 & \frac{-1}{2} & 0 & 0\\
0 & 0 & 0 & \frac{-1}{4}\\
0 & 0 & \frac{1}{4} & 0
\end{bmatrix}.$$
\begin{theorem}
$\left(\mathfrak{g}_4, S_A, J\right)$ is not a second kind algebraic Ricci soliton associated to the connection $\nabla^1$.
\end{theorem}
\begin{proof}
Similar to the proof of the precedent theorem.
\end{proof}
\section{Lorentzian algebraic Ricci solitons on $\mathfrak{h}_3 \oplus \mathbb{R}$}
\begin{theorem} \cite{bokan2015lorentz}
The set $\mathcal{S}^L\left(\mathfrak{h}_3 \oplus \mathbb{R}\right)$ of nonequivalent Lorentz inner product on algebra $\mathfrak{h}_3 \oplus \mathbb{R}$ in basis $\left\lbrace e_1, e_2, e_3, e_4\right\rbrace$ is represented by the following matrices
\begin{multicols}{3}
$S_{\mu} = \begin{bmatrix}
1 & 0 & 0 & 0\\
0 & -1 & 0 & 0\\
0 & 0 & \mu & 0\\
0 & 0 & 0 & 1
\end{bmatrix}$\par
$S_{\lambda}^{+} = \begin{bmatrix}
1 & 0 & 0 & 0\\
0 & 1 & 0 & 0\\
0 & 0 & \lambda & 0\\
0 & 0 & 0 & -1
\end{bmatrix}$\par
$S_{\lambda}^{-} = \begin{bmatrix}
1 & 0 & 0 & 0\\
0 & 1 & 0 & 0\\
0 & 0 & -\lambda & 0\\
0 & 0 & 0 & 1
\end{bmatrix}$\par
$S_0^{1} = \begin{bmatrix}
1 & 0 & 0 & 0\\
0 & 1 & 0 & 0\\
0 & 0 & 0 & 1\\
0 & 0 & 1 & 0
\end{bmatrix}$\par
$S_0^{2} = \begin{bmatrix}
1 & 0 & 0 & 0\\
0 & 0 & 1 & 0\\
0 & 1 & 0 & 0\\
0 & 0 & 0 & \lambda
\end{bmatrix}$\par
$S_0^{3} = \begin{bmatrix}
1 & 0 & 0 & 0\\
0 & 0 & 0 & 1\\
0 & 0 & 1 & 0\\
0 & 1 & 0 & 0
\end{bmatrix}$
\end{multicols}
with $\lambda, \mu > 0$.
\end{theorem}
\begin{remark}
Note that as a vector space, $\mathfrak{h}_3 \oplus \mathbb{R}$ is identified with $\mathbb{R}^4$ and then the elements $e_1$, $e_2$, $e_3$ and $e_4$ are the canonical basis elements of $\mathbb{R}^4$.
\end{remark}
\subsection{The Lorentzian Lie algebra $\left(\mathfrak{h}_3 \oplus \mathbb{R}, S_{\mu}\right)$}
An orthonormal basis of $\left(\mathfrak{h}_3 \oplus \mathbb{R}, S_{\mu}\right)$ is given by
$$v_1 = e_1, \qquad v_2 = \frac{1}{\sqrt{\mu}}e_3, \qquad v_3 = e_4, \qquad v_4 = e_2.$$
The nonzero bracket in this basis is $[v_1, v_4] = \sqrt{\mu}v_2$, thus the nonzero structure constants are $\xi_{142} = -\xi_{412} = \sqrt{\mu}$. Using (\ref{Levi}), we find that the nonvanishing terms in the Levi-Civita connection of $\left(\mathfrak{h}_3 \oplus \mathbb{R}, S_{\mu}\right)$ are the following
$$\nabla_{v_1}v_2 = \nabla_{v_2}v_1 = \frac{\sqrt{\mu}}{2}v_4, \qquad \nabla_{v_1}v_4 = -\nabla_{v_4}v_1 = \frac{\sqrt{\mu}}{2}v_2, \qquad \nabla_{v_2}v_4 = \nabla_{v_4}v_2 = \frac{\sqrt{\mu}}{2}v_1.$$
The nonzero terms in $\nabla_{v_i}(J)v_j$ are
$$\nabla_{v_1}(J)v_2 = \nabla_{v_2}(J)v_1 = \sqrt{\mu}v_4, \qquad \nabla_{v_1}(J)v_4 = -\sqrt{\mu}v_2, \qquad \nabla_{v_2}(J)v_4 = -\sqrt{\mu}v_1.$$
The nonvanishing terms in $\nabla^0$ of $\left(\mathfrak{h}_3 \oplus \mathbb{R}, S_{\mu}\right)$ are
$$\nabla^0_{v_4}v_1 = \frac{-\sqrt{\mu}}{2}v_2, \qquad \nabla^0_{v_4}v_2 = \frac{\sqrt{\mu}}{2}v_1.$$
\begin{theorem}
The Lorentzian Lie algebra $\left(\mathfrak{h}_3 \oplus \mathbb{R}, S_{\mu}, J\right)$ is flat with respect to $\nabla^0$.
\end{theorem}
\begin{proof}
The curvature $R^0$ is given by
$$R^0_{v_iv_j}v_k = \nabla^0_{[v_i, v_j]}v_k - \nabla^0_{v_i}\nabla^0_{v_j}v_k + \nabla^0_{v_j}\nabla^0_{v_i}v_k.$$
The bracket $[v_i, v_j]$ is expressed in term of $v_2$ and we have $\nabla^0_{v_2} = 0$. So the first term $\nabla^0_{[v_i, v_j]}v_k$ is always zero. Next, for the term $\nabla^0_{v_j}v_k$ to be nonzero, $j$ must equal $4$. Again, for the term $\nabla^0_{v_i}\nabla^0_{v_j}v_k$ to be nonzero, $i$ must equal $4$. So, $i = j = 4$. The same reasoning applies to the other term $\nabla^0_{v_j}\nabla^0_{v_i}v_k$. But when $i = j$, we have $R^0_{v_iv_j}v_k = 0$.
\end{proof}
The nonzero term in $\nabla^1$ of $\left(\mathfrak{h}_3 \oplus \mathbb{R}, S_{\mu}\right)$ is $\nabla^1_{v_4}v_1 = -\sqrt{\mu}v_2.$
\begin{theorem}
The Lorentzian Lie algebra $\left(\mathfrak{h}_3 \oplus \mathbb{R}, S_{\mu}, J\right)$ is flat with respect to $\nabla^1$.
\end{theorem}
\begin{proof}
Similar to the proof of the precedent theorem.
\end{proof}
\subsection{The Lorentzian Lie algebra $\left(\mathfrak{h}_3 \oplus \mathbb{R}, S_{\lambda}^{+}\right)$}
An orthonormal basis of $\left(\mathfrak{h}_3 \oplus \mathbb{R}, S_{\lambda}^{+}\right)$ is given by
$$v_1 = e_1, \qquad v_2 = e_2, \qquad v_3 = \frac{1}{\sqrt{\lambda}}e_3, \qquad v_4 = e_4.$$
The nonzero bracket in this basis is $[v_1, v_2] = \sqrt{\lambda}v_3$. The nonvanishing terms in the Levi-Civita connection of $\left(\mathfrak{h}_3 \oplus \mathbb{R}, S_{\lambda}^{+}\right)$ are
$$\nabla_{v_1}v_2 = -\nabla_{v_2}v_1 = \frac{\sqrt{\lambda}}{2}v_3, \qquad \nabla_{v_1}v_3 = \nabla_{v_3}v_1 = \frac{-\sqrt{\lambda}}{2}v_2, \qquad \nabla_{v_2}v_3 = \nabla_{v_3}v_2 = \frac{\sqrt{\lambda}}{2}v_1.$$
We can see that
$$\nabla_{v_i}(J)v_j = 0 \quad \forall i, j \in \left\lbrace 1, \ldots, 4\right\rbrace.$$
Thus $\nabla^1 = \nabla^0 = \nabla$. By \cite{aitbenhaddou2026lorentzian}, the Ricci operator is represented in the orthonormal basis $\left\lbrace v_1, v_2, v_3, v_4\right\rbrace$ by
$$\operatorname{Ric}^1 = \widetilde{\operatorname{Ric}}^1 = \operatorname{Ric}^0 = \widetilde{\operatorname{Ric}}^0 = \operatorname{Ric} = \operatorname{diag}\left\lbrace \frac{-\lambda}{2}, \frac{-\lambda}{2}, \frac{\lambda}{2}, 0\right\rbrace,$$
and the algebraic Ricci soliton equation is expressed as
$$\operatorname{Ric} = \frac{-3\lambda}{2}I_4 + D, \quad \text{where}\; D = \operatorname{diag}\left\lbrace \lambda, \lambda, 2\lambda, \frac{3\lambda}{2}\right\rbrace $$
is the derivation of $\mathfrak{h}_3 \oplus \mathbb{R}$ with respect to the orthonormal basis $\left\lbrace v_1, v_2, v_3, v_4\right\rbrace$.
\subsection{The Lorentzian Lie algebra $\left(\mathfrak{h}_3 \oplus \mathbb{R}, S_{\lambda}^{-}\right)$}
An orthonormal basis of $\left(\mathfrak{h}_3 \oplus \mathbb{R}, S_{\lambda}^{-}\right)$ is given by
$$v_1 = e_1, \qquad v_2 = e_2, \qquad v_3 = e_4, \qquad v_4 = \frac{1}{\sqrt{\lambda}}e_3$$
The nonzero bracket in this basis is $[v_1, v_2] = \sqrt{\lambda}v_4$. The nonvanishing terms in the Levi-Civita connection of $\left(\mathfrak{h}_3 \oplus \mathbb{R}, S_{\lambda}^{-}\right)$ are
$$\nabla_{v_1}v_2 = -\nabla_{v_2}v_1 = \frac{\sqrt{\lambda}}{2}v_4, \qquad \nabla_{v_1}v_4 = \nabla_{v_4}v_1 = \frac{\sqrt{\lambda}}{2}v_2, \qquad \nabla_{v_2}v_4 = \nabla_{v_4}v_2 = \frac{-\sqrt{\lambda}}{2}v_1.$$
The nonzero terms in $\nabla_{v_i}(J)v_j$ are
$$\nabla_{v_1}(J)v_2 = -\nabla_{v_2}(J)v_1 = \sqrt{\lambda}v_4, \qquad \nabla_{v_1}(J)v_4 = -\sqrt{\lambda}v_2, \qquad \nabla_{v_2}(J)v_4 = \sqrt{\lambda}v_1.$$
The nonvanishing terms in $\nabla^0$ of $\left(\mathfrak{h}_3 \oplus \mathbb{R}, S_{\lambda}^{-}\right)$ are
$$\nabla^0_{v_4}v_1 = \frac{\sqrt{\lambda}}{2}v_2, \qquad \nabla^0_{v_4}v_2 = \frac{-\sqrt{\lambda}}{2}v_1.$$
Computing the necessary curvatures yields the following matrix representation of the Ricci operator $\operatorname{Ric}^0$ associated with $\nabla^0$ with respect to the basis $\left\lbrace v_1, v_2, v_3, v_4\right\rbrace$
$$\operatorname{Ric}^0 = \operatorname{diag}\left\lbrace \frac{\lambda}{2}, \frac{\lambda}{2}, 0, 0\right\rbrace.$$
Since $\operatorname{Ric}^0$ is diagonal, then $\operatorname{Ric}^0 = \widetilde{\operatorname{Ric}}^0$.
\begin{theorem}
$\left(\mathfrak{h}_3 \oplus \mathbb{R}, S_{\lambda}^{-}, J\right)$ is a first and second kind algebraic Ricci soliton with respect to $\nabla^0$ where $c = \lambda$ and $D = \operatorname{diag}\left\lbrace \frac{-\lambda}{2}, \frac{-\lambda}{2}, -\lambda, -\lambda\right\rbrace$ is the derivation of $\mathfrak{h}_3 \oplus \mathbb{R}$ with respect to the orthonormal basis $\left\lbrace v_1, v_2, v_3, v_4\right\rbrace$.
\end{theorem}
\begin{proof}
Suppose that $\operatorname{Ric}^0 = \widetilde{\operatorname{Ric}}^0 = cI_4 + D$, where $c \in \mathbb{R}$ and $D$ is a derivation of $\mathfrak{h}_3 \oplus \mathbb{R}$. From this, we get
$$D = \operatorname{Ric}^0 - cI_4 = \operatorname{diag}\left\lbrace \frac{\lambda}{2} - c, \frac{\lambda}{2} - c, -c, -c\right\rbrace.$$
Since $D$ is a derivation of $\mathfrak{h}_3 \oplus \mathbb{R}$, the condition below holds
\begin{eqnarray*}
[D(v_1), v_2] + [v_1, D(v_2)] = D[v_1, v_2] &\Leftrightarrow& [(\frac{\lambda}{2} - c)v_1, v_2] + [v_1, (\frac{\lambda}{2} - c)v_2] = \sqrt{\lambda}D(v_4)\\
&\Leftrightarrow & (\lambda - 2c)[v_1, v_2] = \sqrt{\lambda}D(v_4)\\
&\Leftrightarrow & (\lambda - 2c)\sqrt{\lambda}v_4 = -\sqrt{\lambda}cv_4\\
&\Leftrightarrow & \lambda - 2c = -c \Leftrightarrow c = \lambda.
\end{eqnarray*}
Hence $D = \operatorname{diag}\left\lbrace \frac{-\lambda}{2}, \frac{-\lambda}{2}, -\lambda, -\lambda\right\rbrace$. Since $D$ is diagonal, it obviously satisfies
$$[D(v_i), v_j] + [v_i, D(v_j)] = D[v_i, v_j]$$
for all vanishing brackets. 
\end{proof}
The Kobayashi-Nomizu connection $\nabla^1$ of $\left(\mathfrak{h}_3 \oplus \mathbb{R}, S_{\lambda}^{-}, J\right)$ is null, i.e. $\nabla^1 = 0$.
\subsection{The Lorentzian Lie algebra $\left(\mathfrak{h}_3 \oplus \mathbb{R}, S_{0}^{1}\right)$}
An orthonormal basis of $\left(\mathfrak{h}_3 \oplus \mathbb{R}, S_{0}^{1}\right)$ is given by
$$v_1 = e_1, \qquad v_2 = e_2, \qquad v_3 = \frac{1}{\sqrt{2}}(e_3 + e_4), \qquad v_4 = \frac{1}{\sqrt{2}}(e_3 - e_4)$$
The nonzero bracket in this basis is $[v_1, v_2] = \frac{1}{\sqrt{2}}v_3 + \frac{1}{\sqrt{2}}v_4$. In view of (\ref{Levi}), the Levi-Civita connection of $\left(\mathfrak{h}_3 \oplus \mathbb{R}, S_{0}^{1}\right)$ is described by
\begin{multicols}{3}
	$\nabla_{v_i}v_i = 0,\ \forall i=1,\ldots,4$\par
	$\nabla_{v_1}v_2 = \frac{1}{2\sqrt{2}}v_3 + \frac{1}{2\sqrt{2}}v_4$\par
	$\nabla_{v_2}v_1 = \frac{-1}{2\sqrt{2}}v_3 - \frac{1}{2\sqrt{2}}v_4$\par
	$\nabla_{v_1}v_3 = \nabla_{v_3}v_1 = \frac{-1}{2\sqrt{2}}v_2$\par
	$\nabla_{v_1}v_4 = \nabla_{v_4}v_1 = \frac{1}{2\sqrt{2}}v_2$\par
	$\nabla_{v_2}v_3 = \nabla_{v_3}v_2 = \frac{1}{2\sqrt{2}}v_1$\par
	$\nabla_{v_2}v_4 = \nabla_{v_4}v_2 = \frac{-1}{2\sqrt{2}}v_1$\par
	$\nabla_{v_3}v_4 = \nabla_{v_4}v_3=0$
\end{multicols}
The nonvanishing terms in $\nabla_{v_i}(J)v_j$ are
$$\nabla_{v_1}(J)v_2 = -\nabla_{v_2}(J)v_1 = \frac{1}{\sqrt{2}}v_4, \qquad \nabla_{v_1}(J)v_4 = \frac{-1}{\sqrt{2}}v_2, \qquad \nabla_{v_2}(J)v_4 = \frac{1}{\sqrt{2}}v_1.$$
The canonical connection $\nabla^0$ of $\left(\mathfrak{h}_3 \oplus \mathbb{R}, S_{0}^{1}\right)$ is expressed as
\begin{multicols}{3}
	$\nabla^0_{v_i}v_i = 0,\ \forall i=1,\ldots,4$\par
	$\nabla^0_{v_1}v_2 = \frac{1}{2\sqrt{2}}v_3$\par
	$\nabla^0_{v_2}v_1 = \frac{-1}{2\sqrt{2}}v_3$\par
	$\nabla^0_{v_1}v_3 = \nabla^0_{v_3}v_1 = \frac{-1}{2\sqrt{2}}v_2$\par
	$\nabla^0_{v_1}v_4 = \nabla^0_{v_2}v_4 = 0$\par
	$\nabla^0_{v_4}v_1 = \frac{1}{2\sqrt{2}}v_2$\par
	$\nabla^0_{v_2}v_3 = \nabla^0_{v_3}v_2 = \frac{1}{2\sqrt{2}}v_1$\par
	$\nabla^0_{v_4}v_2 = \frac{-1}{2\sqrt{2}}v_1$\par
	$\nabla^0_{v_3}v_4 = \nabla^0_{v_4}v_3 = 0$
\end{multicols}
In order to find to Ricci operator associated with $\nabla^0$, the following curvatures are needed
\begin{multicols}{4}
	$R^0_{v_2v_1}v_2 = \frac{-1}{8}v_1$\par
	$R^0_{v_3v_1}v_3 = \frac{1}{8}v_1$\par
	$R^0_{v_4v_1}v_4 = 0$\par
	$R^0_{v_1v_2}v_1 = \frac{-1}{8}v_2$\par
	$R^0_{v_3v_2}v_3 = \frac{1}{8}v_2$\par
	$R^0_{v_4v_2}v_4 = 0$\par
	$R^0_{v_1v_3}v_1 = \frac{1}{8}v_3$\par
	$R^0_{v_2v_3}v_2 = \frac{1}{8}v_3$\par
	$R^0_{v_4v_3}v_4 = 0$\par
	$R^0_{v_1v_4}v_1 = \frac{-1}{8}v_3$\par
	$R^0_{v_2v_4}v_2 = \frac{-1}{8}v_3$\par
	$R^0_{v_3v_4}v_3 = 0$
\end{multicols}
The Ricci operator $\operatorname{Ric}^0$ of $\nabla^0$ is represented in the basis $\mathscr{B}$ by the following matrix
$$\operatorname{Ric}^0 = \begin{bmatrix}
0 & 0 & 0 & 0\\
0 & 0 & 0 & 0\\
0 & 0 & \frac{1}{4} & \frac{-1}{4}\\
0 & 0 & 0 & 0
\end{bmatrix}.$$
\begin{theorem}
$\left(\mathfrak{h}_3 \oplus \mathbb{R}, S_{0}^{1}, J\right)$ is a first kind algebraic Ricci soliton with respect to $\nabla^0$ where $c = 0$ and $D = \operatorname{Ric}^0$ is the derivation of $\mathfrak{h}_3 \oplus \mathbb{R}$ with respect to the orthonormal basis $\left\lbrace v_1, v_2, v_3, v_4\right\rbrace$.
\end{theorem}
\begin{proof}
Assume that $\operatorname{Ric}^0 = cI_4 + D$, where $c \in \mathbb{R}$ and $D$ is a derivation of $\mathfrak{h}_3 \oplus \mathbb{R}$. Hence
$$D = \operatorname{Ric}^0 - cI_4 = \begin{bmatrix}
- c & 0 & 0 & 0\\
0 & - c & 0 & 0\\
0 & 0 & \frac{1}{4} - c & \frac{-1}{4}\\
0 & 0 & 0 & - c
\end{bmatrix}.$$
Since $D$ is a derivation of $\mathfrak{h}_3 \oplus \mathbb{R}$, it satisfies the following condition
\begin{eqnarray*}
[D(v_1), v_2] + [v_1, D(v_2)] = D[v_1, v_2] &\Leftrightarrow & - 2c\left[v_1, v_2\right] = \frac{1}{\sqrt{2}}\left(D(v_3) + D(v_4)\right)\\
&\Leftrightarrow& - 2c\frac{1}{\sqrt{2}}(v_3 + v_4) = \frac{1}{\sqrt{2}}\left(D(v_3) + D(v_4)\right)\\
&\Leftrightarrow& - 2c(v_3 + v_4) = -c(v_3 + v_4) \Leftrightarrow 2c = c \Leftrightarrow c = 0.
\end{eqnarray*}
Hence $c = 0$ and $ D = \operatorname{Ric}^0$. A straightforward verification shows that $D = \operatorname{Ric}^0$ satisfies the following condition for all vanishing brackets
$$[D(v_i), v_j] + [v_i, D(v_j)] = D[v_i, v_j].$$ 
Therefore, $\left(\mathfrak{h}_3 \oplus \mathbb{R}, S_{0}^{1}, J\right)$ is a first kind algebraic Ricci soliton where $c = 0$ and $D = \operatorname{Ric}^0$ is a derivation of $\mathfrak{h}_3 \oplus \mathbb{R}$ with respect to the orthonormal basis $\left\lbrace v_1, v_2, v_3, v_4\right\rbrace$.
\end{proof}
We obtain that $\widetilde{\operatorname{Ric}}^0$ is represented in the basis $\mathscr{B}$ by
$$\widetilde{\operatorname{Ric}}^0 = \begin{bmatrix}
0 & 0 & 0 & 0\\
0 & 0 & 0 & 0\\
0 & 0 & \frac{1}{4} & \frac{-1}{8}\\
0 & 0 & \frac{1}{8} & 0
\end{bmatrix}.$$
\begin{theorem}
$\left(\mathfrak{h}_3 \oplus \mathbb{R}, S_{0}^{1}, J\right)$ is a second kind algebraic Ricci soliton with respect to $\nabla^0$ where $c = \frac{-1}{8}$ and 
$$D = \begin{bmatrix}
\frac{1}{8} & 0 & 0 & 0\\
0 & \frac{1}{8} & 0 & 0\\
0 & 0 & \frac{3}{8} & \frac{-1}{8}\\
0 & 0 & \frac{1}{8} & \frac{1}{8}
\end{bmatrix}$$
is the derivation of $\mathfrak{h}_3 \oplus \mathbb{R}$ with respect to the orthonormal basis $\left\lbrace v_1, v_2, v_3, v_4\right\rbrace$.
\end{theorem}
\begin{proof}
Suppose that $\widetilde{\operatorname{Ric}}^0 = cI_4 + D$, where $c \in \mathbb{R}$ and $D$ is a derivation of $\mathfrak{h}_3 \oplus \mathbb{R}$. Then
$$D = \widetilde{\operatorname{Ric}}^0 - cI_4 = \begin{bmatrix}
-c & 0 & 0 & 0\\
0 & -c & 0 & 0\\
0 & 0 & \frac{1}{4} - c & \frac{-1}{8}\\
0 & 0 & \frac{1}{8} & -c
\end{bmatrix}.$$
Given that $D$ is a derivation of $\mathfrak{h}_3 \oplus \mathbb{R}$, the rule below holds
\begin{eqnarray*}
[D(v_1), v_2] + [v_1, D(v_2)] = D[v_1, v_2] &\Leftrightarrow & - 2c\left[v_1, v_2\right] = \frac{1}{\sqrt{2}}\left(D(v_3) + D(v_4)\right)\\
&\Leftrightarrow& - 2c\frac{1}{\sqrt{2}}(v_3 + v_4) = \frac{1}{\sqrt{2}}\left(D(v_3) + D(v_4)\right)\\
&\Leftrightarrow& - 2c(v_3 + v_4) = (\frac{1}{8} - c)(v_3 + v_4)\\
&\Leftrightarrow& -2c = \frac{1}{8} - c \Leftrightarrow c = \frac{-1}{8}.
\end{eqnarray*}
Hence $c = \frac{-1}{8}$ and $D$ is given by
$$D = \begin{bmatrix}
\frac{1}{8} & 0 & 0 & 0\\
0 & \frac{1}{8} & 0 & 0\\
0 & 0 & \frac{3}{8} & \frac{-1}{8}\\
0 & 0 & \frac{1}{8} & \frac{1}{8}
\end{bmatrix}.$$
It is straightforward to check that $D$ as above satisfies the following Leibniz rule for all vanishing brackets
$$[D(v_i), v_j] + [v_i, D(v_j)] = D[v_i, v_j].$$ 
\end{proof}
The Kobayashi-Nomizu connection $\nabla^1$ of $\left(\mathfrak{h}_3 \oplus \mathbb{R}, S_{0}^{1}\right)$ is expressed as
\begin{multicols}{3}
	$\nabla^1_{v_i}v_i = 0,\ \forall i=1,\ldots,4$\par
	$\nabla^1_{v_1}v_2 = \frac{1}{2\sqrt{2}}v_3$\par
	$\nabla^1_{v_2}v_1 = \frac{-1}{2\sqrt{2}}v_3$\par
	$\nabla^1_{v_1}v_3 = \nabla^1_{v_3}v_1 = \frac{-1}{2\sqrt{2}}v_2$\par
	$\nabla^1_{v_1}v_4 = \nabla^1_{v_2}v_4 = 0$\par
	$\nabla^1_{v_4}v_1 = 0$\par
	$\nabla^1_{v_2}v_3 = \nabla^1_{v_3}v_2 = \frac{1}{2\sqrt{2}}v_1$\par
	$\nabla^1_{v_4}v_2 = 0$\par
	$\nabla^1_{v_3}v_4 = \nabla^1_{v_4}v_3 = 0$
\end{multicols}
To find the Ricci operator associated with $\nabla^1$, the curvatures below are needed
\begin{multicols}{4}
	$R^1_{v_2v_1}v_2 = \frac{-3}{8}v_1$\par
	$R^1_{v_3v_1}v_3 = \frac{1}{8}v_1$\par
	$R^1_{v_4v_1}v_4 = 0$\par
	$R^1_{v_1v_2}v_1 = \frac{-3}{8}v_2$\par
	$R^1_{v_3v_2}v_3 = \frac{1}{8}v_2$\par
	$R^1_{v_4v_2}v_4 = 0$\par
	$R^1_{v_1v_3}v_1 = \frac{1}{8}v_3$\par
	$R^1_{v_2v_3}v_2 = \frac{1}{8}v_3$\par
	$R^1_{v_4v_3}v_4 = 0$\par
	$R^1_{v_1v_4}v_1 = 0$\par
	$R^1_{v_2v_4}v_2 = 0$\par
	$R^1_{v_3v_4}v_3 = 0$
\end{multicols}
The Ricci operator $\operatorname{Ric}^1$ is represented in the basis $\left\lbrace v_1, v_2, v_3, v_4\right\rbrace$ by the matrix below
$$\operatorname{Ric}^1 = \operatorname{diag}\left\lbrace \frac{-1}{4}, \frac{-1}{4}, \frac{1}{4}, 0\right\rbrace.$$
Since $\operatorname{Ric}^1$ is diagonal, then $\operatorname{Ric}^1 = \widetilde{\operatorname{Ric}}^1$.
\begin{theorem}
$\left(\mathfrak{h}_3 \oplus \mathbb{R}, S_{0}^{1}, J\right)$ does not exhibit the properties of a first nor second kind algebraic Ricci soliton with respect to $\nabla^1$.
\end{theorem}
\begin{proof}
Suppose that $\operatorname{Ric}^1 = \widetilde{\operatorname{Ric}}^1 = cI_4 + D$, where $c \in \mathbb{R}$ and $D$ is a derivation of $\mathfrak{h}_3 \oplus \mathbb{R}$. From this, we get
$$D = \operatorname{Ric}^1 - cI_4 = \operatorname{diag}\left\lbrace \frac{-1}{4} - c, \frac{-1}{4} - c, \frac{1}{4} - c, -c\right\rbrace.$$
Since $D$ is a derivation of $\mathfrak{h}_3 \oplus \mathbb{R}$, the condition below holds
\begin{eqnarray*}
[D(v_1), v_2] + [v_1, D(v_2)] = D[v_1, v_2] &\Leftrightarrow& [(\frac{-1}{4} - c)v_1, v_2] + [v_1, (\frac{-1}{4} - c)v_2] = D[v_1, v_2]\\
&\Leftrightarrow & (\frac{-1}{2} - 2c)[v_1, v_2] = \frac{1}{\sqrt{2}}(D(v_3) + D(v_4))\\
&\Leftrightarrow & (\frac{-1}{2} - 2c)\frac{1}{\sqrt{2}}(v_3 + v_4) = \frac{1}{\sqrt{2}}((\frac{1}{4} - c)v_3 - cv_4)
\end{eqnarray*}
By identifying the components of $v_3$ and $v_4$, we obtain
$$\left\lbrace\begin{array}{lll}
\frac{-1}{2} - 2c = \frac{1}{4} - c \Longrightarrow c = \frac{-3}{4} \\
\frac{-1}{2} - 2c = -c \Longrightarrow c = \frac{-1}{2}
\end{array}\right.$$
This is a contradiction.
\end{proof}
\subsection{The Lorentzian Lie algebra $\left(\mathfrak{h}_3 \oplus \mathbb{R}, S_{0}^{2}\right)$}
The following basis $\mathscr{B} = \left\lbrace v_1, v_2, v_3, v_4\right\rbrace $ where
$$v_1 = e_1, \quad v_2 = e_4, \quad v_3 = \frac{1}{\sqrt{2}}(e_2 + e_3), \quad v_4 = \frac{1}{\sqrt{2}}(e_2 - e_3),$$
is an orthonormal basis of $\mathfrak{h}_3 \oplus \mathbb{R}$ with respect to our metric $S_{0}^{2}$. The bracket in the basis $\mathscr{B}$ is given by
$$[v_1, v_3] = [v_1, v_4] = \frac{1}{2}(v_3 - v_4).$$
The Levi-Civita connection of $\left(\mathfrak{h}_3 \oplus \mathbb{R}, S_{0}^{2}\right)$ is described by
\begin{multicols}{3}
	$\nabla_{v_1}v_1 = \nabla_{v_2}v_2 = 0$\par
	$\nabla_{v_3}v_3 = \nabla_{v_4}v_4 = \frac{1}{2}v_1$\par
	$\nabla_{v_1}v_2 = \nabla_{v_2}v_1 = 0$\par
	$\nabla_{v_1}v_3 = \nabla_{v_1}v_4 = 0$\par
	$\nabla_{v_3}v_1 = \frac{-1}{2}v_3 + \frac{1}{2}v_4$\par
	$\nabla_{v_4}v_1 = \frac{-1}{2}v_3 + \frac{1}{2}v_4$\par
	$\nabla_{v_2}v_3 = \nabla_{v_3}v_2 = 0$\par
	$\nabla_{v_2}v_4 = \nabla_{v_4}v_2 = 0$\par
	$\nabla_{v_3}v_4 = \nabla_{v_4}v_3 = \frac{1}{2}v_1$
\end{multicols}
The nonvanishing terms in $\nabla_{v_i}(J)v_j$ are
$$\nabla_{v_3}(J)v_1 = \nabla_{v_4}(J)v_1 = v_4, \qquad \nabla_{v_3}(J)v_4 = \nabla_{v_4}(J)v_4 = -v_1.$$
The canonical connection $\nabla^0$ of $\left(\mathfrak{h}_3 \oplus \mathbb{R}, S_{0}^{2}\right)$ is expressed as
\begin{multicols}{3}
	$\nabla^0_{v_1}v_1 = \nabla^0_{v_2}v_2 = 0$\par
	$\nabla^0_{v_3}v_3 = \frac{1}{2}v_1$\par
	$\nabla^0_{v_1}v_2 = \nabla^0_{v_2}v_1 = 0$\par
	$\nabla^0_{v_1}v_3 = \nabla^0_{v_1}v_4 = 0$\par
	$\nabla^0_{v_3}v_1 = \nabla^0_{v_4}v_1 = \frac{-1}{2}v_3$\par
	$\nabla^0_{v_2}v_3 = \nabla^0_{v_3}v_2 = 0$\par
	$\nabla^0_{v_2}v_4 = \nabla^0_{v_4}v_2 = 0$\par
	$\nabla^0_{v_3}v_4 = \nabla^0_{v_4}v_4 = 0$\par
	$\nabla^0_{v_4}v_3 = \frac{1}{2}v_1$
\end{multicols}
It follows from \cite{aitbenhaddou2026lorentzian} that $\left(\mathfrak{h}_3 \oplus \mathbb{R}, S_{0}^{2}\right)$ is flat with respect to the Levi-Civita connection $\nabla$. The following theorem establishes that it is also flat with respect to the canonical connection $\nabla^0$.
\begin{theorem}
$\left(\mathfrak{h}_3 \oplus \mathbb{R}, S_{0}^{2}\right)$ is flat with respect to $\nabla^0$.
\end{theorem}
\begin{proof}
A routine calculation shows that the following curvature transformations vanish:
$$R^0_{v_1v_2} = R^0_{v_1v_3} = R^0_{v_1v_4} = R^0_{v_2v_3} = R^0_{v_2v_4} = R^0_{v_3v_4} = 0,$$
where 
$$R^0_{v_iv_j}v_k = \nabla^0_{[v_i, v_j]}v_k - \nabla^0_{v_i}\nabla^0_{v_j}v_k + \nabla^0_{v_j}\nabla^0_{v_i}v_k.$$
\end{proof}
The Kobayashi-Nomizu connection $\nabla^1$ of $\left(\mathfrak{h}_3 \oplus \mathbb{R}, S_{0}^{2}\right)$ is expressed as
\begin{multicols}{3}
	$\nabla^1_{v_1}v_1 = \nabla^1_{v_2}v_2 = 0$\par
	$\nabla^1_{v_3}v_3 = \frac{1}{2}v_1$\par
	$\nabla^1_{v_1}v_2 = \nabla^1_{v_2}v_1 = 0$\par
	$\nabla^1_{v_1}v_3 = 0$\par
	$\nabla^1_{v_1}v_4 = \frac{-1}{2}v_4$\par
	$\nabla^1_{v_3}v_1 = \nabla^1_{v_4}v_1 = \frac{-1}{2}v_3$\par
	$\nabla^1_{v_2}v_3 = \nabla^1_{v_3}v_2 = 0$\par
	$\nabla^1_{v_2}v_4 = \nabla^1_{v_4}v_2 = 0$\par
	$\nabla^1_{v_3}v_4 = \nabla^1_{v_4}v_4 = 0$\par
	$\nabla^1_{v_4}v_3 = 0$
\end{multicols}
The Ricci operator $\operatorname{Ric}^1$ of $\nabla^1$ is represented in the basis $\mathscr{B}$ by the matrix
$$\operatorname{Ric}^1 = \begin{bmatrix}
\frac{-1}{4} & 0 & 0 & 0\\
0 & 0 & 0 & 0\\
0 & 0 & 0 & \frac{-1}{4}\\
0 & 0 & 0 & 0
\end{bmatrix}.$$
\begin{theorem}
$\left(\mathfrak{h}_3 \oplus \mathbb{R}, S_{0}^{2}, J\right)$ does not exhibit the properties of a first kind algebraic Ricci soliton with respect to $\nabla^1$.
\end{theorem}
\begin{proof}
Assume that $\operatorname{Ric}^1 = cI_4 + D$, where $c \in \mathbb{R}$ and $D$ is a derivation of $\mathfrak{h}_3 \oplus \mathbb{R}$. Then
$$D = \operatorname{Ric}^1 - cI_4 = \begin{bmatrix}
\frac{-1}{4} - c & 0 & 0 & 0\\
0 & -c & 0 & 0\\
0 & 0 & -c & \frac{-1}{4}\\
0 & 0 & 0 & -c
\end{bmatrix}.$$
Since $D$ is a derivation of $\mathfrak{h}_3 \oplus \mathbb{R}$, the condition below holds
\begin{eqnarray*}
[D(v_1), v_3] + [v_1, D(v_3)] = D[v_1, v_3] &\Leftrightarrow& [(\frac{-1}{4} - c)v_1, v_3] + [v_1, -cv_3] = D[v_1, v_3]\\
&\Leftrightarrow & (\frac{-1}{4} - 2c)[v_1, v_3] = \frac{1}{2}(D(v_3) - D(v_4))\\
&\Leftrightarrow & (\frac{-1}{4} - 2c)\frac{1}{2}(v_3 - v_4) = \frac{1}{2}((\frac{1}{4} - c)v_3 + cv_4)
\end{eqnarray*}
By identifying the components of $v_3$ and $v_4$, we obtain
$$\left\lbrace\begin{array}{lll}
\frac{-1}{4} - 2c = \frac{1}{4} - c \Longrightarrow c = \frac{-1}{2} \\
\frac{1}{4} + 2c = c \Longrightarrow c = \frac{-1}{4}
\end{array}\right.$$
This is a contradiction.
\end{proof}
We obtain that $\widetilde{\operatorname{Ric}}^1$ is represented in the basis $\mathscr{B}$ by the matrix
$$\widetilde{\operatorname{Ric}}^1 = \begin{bmatrix}
\frac{-1}{4} & 0 & 0 & 0\\
0 & 0 & 0 & 0\\
0 & 0 & 0 & \frac{-1}{8}\\
0 & 0 & \frac{1}{8} & 0
\end{bmatrix}.$$
\begin{theorem}
$\left(\mathfrak{h}_3 \oplus \mathbb{R}, S_{0}^{2}, J\right)$ does not exhibit the properties of a second kind algebraic Ricci soliton with respect to $\nabla^1$.
\end{theorem}
\begin{proof}
The proof is similar to that of the precedent theorem, we obtain that $c = \frac{-1}{4} = 0$, which is a contradiction.
\end{proof}
\subsection{The Lorentzian Lie algebra $\left(\mathfrak{h}_3 \oplus \mathbb{R}, S_{0}^{3}\right)$}
An orthonormal basis of $\left(\mathfrak{h}_3 \oplus \mathbb{R}, S_{0}^{3}\right)$ is given by
$$v_1 = e_1, \qquad v_2 = e_3, \qquad v_3 = \frac{1}{\sqrt{2}}(e_2 + e_4), \qquad v_4 = \frac{1}{\sqrt{2}}(e_2 - e_4).$$ 
The nonzero brackets in this basis are: $[v_1, v_3] = [v_1, v_4] = \frac{1}{\sqrt{2}}v_2$. According to (\ref{Levi}), the Levi-Civita connection of $\left(\mathfrak{h}_3 \oplus \mathbb{R}, S_{0}^{3}\right)$ is described by
\begin{multicols}{3}
	$\nabla_{v_i}v_i = 0,\ \forall i=1,\ldots,4$\par
	$\nabla_{v_1}v_2 = \frac{-1}{2\sqrt{2}}v_3 + \frac{1}{2\sqrt{2}}v_4$\par
	$\nabla_{v_2}v_1 = \frac{-1}{2\sqrt{2}}v_3 + \frac{1}{2\sqrt{2}}v_4$\par
	$\nabla_{v_1}v_3 = \nabla_{v_1}v_4 = \frac{1}{2\sqrt{2}}v_2$\par
	$\nabla_{v_3}v_1 = \nabla_{v_4}v_1 = \frac{-1}{2\sqrt{2}}v_2$\par
	$\nabla_{v_2}v_3 = \nabla_{v_3}v_2 = \frac{1}{2\sqrt{2}}v_1$\par
	$\nabla_{v_2}v_4 = \nabla_{v_4}v_2 = \frac{1}{2\sqrt{2}}v_1$\par
	$\nabla_{v_3}v_4 = \nabla_{v_4}v_3 = 0$
\end{multicols}
The nonvanishing terms in $\nabla_{v_i}(J)v_j$ are
$$\nabla_{v_1}(J)v_2 = \nabla_{v_2}(J)v_1 = \frac{1}{\sqrt{2}}v_4, \qquad \nabla_{v_1}(J)v_4 = \frac{-1}{\sqrt{2}}v_2, \qquad \nabla_{v_2}(J)v_4 = \frac{-1}{\sqrt{2}}v_1.$$
The canonical connection $\nabla^0$ of $\left(\mathfrak{h}_3 \oplus \mathbb{R}, S_{0}^{3}\right)$ is expressed as
\begin{multicols}{3}
	$\nabla^0_{v_i}v_i = 0,\ \forall i=1,\ldots,4$\par
	$\nabla^0_{v_1}v_2 = \nabla^0_{v_2}v_1 = \frac{-1}{2\sqrt{2}}v_3$\par
	$\nabla^0_{v_1}v_3 = \frac{1}{2\sqrt{2}}v_2$\par
	$\nabla^0_{v_1}v_4 = 0$\par
	$\nabla^0_{v_3}v_1 = \nabla^0_{v_4}v_1 = \frac{-1}{2\sqrt{2}}v_2$\par
	$\nabla^0_{v_2}v_3 = \nabla^0_{v_3}v_2 = \frac{1}{2\sqrt{2}}v_1$\par
	$\nabla^0_{v_2}v_4 = 0$\par
	$\nabla^0_{v_4}v_2 = \frac{1}{2\sqrt{2}}v_1$\par
	$\nabla^0_{v_3}v_4 = \nabla^0_{v_4}v_3 = 0$
\end{multicols}
The Ricci operator $\operatorname{Ric}^0$ of $\nabla^0$ is represented in the basis $\mathscr{B} = \left\lbrace v_1, v_2, v_3, v_4\right\rbrace$ by the following matrix
$$\operatorname{Ric}^0 = \begin{bmatrix}
\frac{-1}{4} & 0 & 0 & 0\\
0 & \frac{1}{4} & 0 & 0\\
0 & 0 & \frac{-1}{4} & \frac{-1}{4}\\
0 & 0 & 0 & 0
\end{bmatrix}.$$
\begin{theorem}
$\left(\mathfrak{h}_3 \oplus \mathbb{R}, S_{0}^{3}, J\right)$ is a first kind algebraic Ricci soliton with respect to $\nabla^0$ where $c = \frac{-3}{4}$ and 
$$D = \begin{bmatrix}
\frac{1}{2} & 0 & 0 & 0\\
0 & 1 & 0 & 0\\
0 & 0 & \frac{1}{2} & \frac{-1}{4}\\
0 & 0 & 0 & \frac{3}{4}
\end{bmatrix}$$ 
is the derivation of $\mathfrak{h}_3 \oplus \mathbb{R}$ with respect to the orthonormal basis $\left\lbrace v_1, v_2, v_3, v_4\right\rbrace$.
\end{theorem}
\begin{proof}
Assume that $\operatorname{Ric}^0 = cI_4 + D$, where $c \in \mathbb{R}$ and $D$ is a derivation of $\mathfrak{h}_3 \oplus \mathbb{R}$. Hence
$$D = \operatorname{Ric}^0 - cI_4 = \begin{bmatrix}
\frac{-1}{4} - c & 0 & 0 & 0\\
0 & \frac{1}{4} - c & 0 & 0\\
0 & 0 & \frac{-1}{4} - c & \frac{-1}{4}\\
0 & 0 & 0 & -c
\end{bmatrix}.$$
It easy to verify that
$$[D(v_1), v_3] + [v_1, D(v_3)] = D[v_1, v_3] \Leftrightarrow c = \frac{-3}{4},$$
$$[D(v_1), v_4] + [v_1, D(v_4)] = D[v_1, v_4] \Leftrightarrow c = \frac{-3}{4}.$$
Hence $D$ is given by 
$$D = \begin{bmatrix}
\frac{1}{2} & 0 & 0 & 0\\
0 & 1 & 0 & 0\\
0 & 0 & \frac{1}{2} & \frac{-1}{4}\\
0 & 0 & 0 & \frac{3}{4}
\end{bmatrix}.$$
Now, it is straightforward to check that $D$ satisfies the derivation condition for all vanishing brackets.
\end{proof}
We obtain that $\widetilde{\operatorname{Ric}}^0$ is represented in the basis $\mathscr{B}$ by
$$\widetilde{\operatorname{Ric}}^0 = \begin{bmatrix}
\frac{-1}{4} & 0 & 0 & 0\\
0 & \frac{1}{4} & 0 & 0\\
0 & 0 & \frac{-1}{4} & \frac{-1}{8}\\
0 & 0 & \frac{1}{8} & 0
\end{bmatrix}.$$
\begin{theorem}
$\left(\mathfrak{h}_3 \oplus \mathbb{R}, S_{0}^{3}, J\right)$ is a second kind algebraic Ricci soliton with respect to $\nabla^0$ where $c = \frac{-5}{8}$ and 
$$D = \begin{bmatrix}
\frac{3}{8} & 0 & 0 & 0\\
0 & \frac{7}{8} & 0 & 0\\
0 & 0 & \frac{3}{8} & \frac{-1}{8}\\
0 & 0 & \frac{1}{8} & \frac{5}{8}
\end{bmatrix}$$
is the derivation of $\mathfrak{h}_3 \oplus \mathbb{R}$ with respect to the orthonormal basis $\left\lbrace v_1, v_2, v_3, v_4\right\rbrace$.
\end{theorem}
\begin{proof}
Let us assume that $\widetilde{\operatorname{Ric}}^0 = cI_4 + D$, where $c \in \mathbb{R}$ and $D$ is a derivation of $\mathfrak{h}_3 \oplus \mathbb{R}$. Thus
$$D = \widetilde{\operatorname{Ric}}^0 - cI_4 = \begin{bmatrix}
\frac{-1}{4} - c & 0 & 0 & 0\\
0 & \frac{1}{4} - c & 0 & 0\\
0 & 0 & \frac{-1}{4} - c & \frac{-1}{8}\\
0 & 0 & \frac{1}{8} & -c
\end{bmatrix}.$$
It easy to verify that
$$[D(v_1), v_3] + [v_1, D(v_3)] = D[v_1, v_3] \Leftrightarrow c = \frac{-5}{8},$$
$$[D(v_1), v_4] + [v_1, D(v_4)] = D[v_1, v_4] \Leftrightarrow c = \frac{-5}{8}.$$
Hence $D$ is given by 
$$D = \begin{bmatrix}
\frac{3}{8} & 0 & 0 & 0\\
0 & \frac{7}{8} & 0 & 0\\
0 & 0 & \frac{3}{8} & \frac{-1}{8}\\
0 & 0 & \frac{1}{8} & \frac{5}{8}
\end{bmatrix}.$$
Next, it is straightforward to check that $D$ satisfies the derivation condition for all vanishing brackets.
\end{proof}
The Kobayashi-Nomizu connection $\nabla^1$ of $\left(\mathfrak{h}_3 \oplus \mathbb{R}, S_{0}^{3}\right)$ is expressed as
\begin{multicols}{3}
	$\nabla^1_{v_i}v_i = 0,\ \forall i=1,\ldots,4$\par
	$\nabla^1_{v_1}v_2 = \nabla^1_{v_2}v_1 = \frac{-1}{2\sqrt{2}}v_3$\par
	$\nabla^1_{v_1}v_3 = \frac{1}{2\sqrt{2}}v_2$\par
	$\nabla^1_{v_1}v_4 = 0$\par
	$\nabla^1_{v_4}v_1 = \frac{-1}{\sqrt{2}}v_2$\par
	$\nabla^1_{v_3}v_1 = \frac{-1}{2\sqrt{2}}v_2$\par
	$\nabla^1_{v_2}v_3 = \nabla^1_{v_3}v_2 = \frac{1}{2\sqrt{2}}v_1$\par
	$\nabla^1_{v_2}v_4 = \nabla^1_{v_4}v_2 = 0$\par
	$\nabla^1_{v_3}v_4 = \nabla^1_{v_4}v_3 = 0$
\end{multicols}
The Ricci operator $\operatorname{Ric}^1$ of $\nabla^1$ is represented in the basis $\mathscr{B} = \left\lbrace v_1, v_2, v_3, v_4\right\rbrace$ by the following matrix
$$\operatorname{Ric}^1 = \begin{bmatrix}
\frac{-1}{4} & 0 & 0 & 0\\
0 & \frac{1}{4} & 0 & 0\\
0 & 0 & \frac{-1}{4} & \frac{-1}{2}\\
0 & 0 & 0 & 0
\end{bmatrix}.$$
\begin{theorem}
$\left(\mathfrak{h}_3 \oplus \mathbb{R}, S_{0}^{3}, J\right)$ is not a first kind algebraic Ricci soliton with respect to $\nabla^1$.
\end{theorem}
\begin{proof}
Assume that $\left(\mathfrak{h}_3 \oplus \mathbb{R}, S_{0}^{3}, J\right)$ is a first kind algebraic Ricci soliton associated to $\nabla^1$, then there exists a real number $c$ and a derivation $D$ of $\mathfrak{h}_3 \oplus \mathbb{R}$ such that $\operatorname{Ric}^1 = cI_4 + D$. Then we obtain that
$$D = \operatorname{Ric}^1 - cI_4 = \begin{bmatrix}
\frac{-1}{4} - c & 0 & 0 & 0\\
0 & \frac{1}{4} - c & 0 & 0\\
0 & 0 & \frac{-1}{4} - c & \frac{-1}{2}\\
0 & 0 & 0 & -c
\end{bmatrix}.$$
It easy to verify that
$$[D(v_1), v_3] + [v_1, D(v_3)] = D[v_1, v_3] \Leftrightarrow c = \frac{-3}{4},$$
$$[D(v_1), v_4] + [v_1, D(v_4)] = D[v_1, v_4] \Leftrightarrow c = -1.$$
Hence $c = \frac{-3}{4} = -1$, which is a contradiction.
\end{proof}
We obtain that $\widetilde{\operatorname{Ric}}^1$ is represented in the basis $\mathscr{B}$ by
$$\widetilde{\operatorname{Ric}}^1 = \begin{bmatrix}
\frac{-1}{4} & 0 & 0 & 0\\
0 & \frac{1}{4} & 0 & 0\\
0 & 0 & \frac{-1}{4} & \frac{-1}{4}\\
0 & 0 & \frac{1}{4} & 0
\end{bmatrix}.$$
\begin{theorem}
$\left(\mathfrak{h}_3 \oplus \mathbb{R}, S_{0}^{3}, J\right)$ does not qualify as a second kind algebraic Ricci soliton with respect to $\nabla^1$.
\end{theorem}
\begin{proof}
Suppose that $\widetilde{\operatorname{Ric}}^1 = cI_4 + D$, where $c \in \mathbb{R}$ and $D$ is a derivation of $\mathfrak{h}_3 \oplus \mathbb{R}$. Hence
$$D = \widetilde{\operatorname{Ric}}^1 - cI_4 = \begin{bmatrix}
\frac{-1}{4} - c & 0 & 0 & 0\\
0 & \frac{1}{4} - c & 0 & 0\\
0 & 0 & \frac{-1}{4} - c & \frac{-1}{4}\\
0 & 0 & \frac{1}{4} & -c
\end{bmatrix}.$$
It easy to verify that
$$[D(v_1), v_3] + [v_1, D(v_3)] = D[v_1, v_3] \Leftrightarrow c = \frac{-1}{2},$$
$$[D(v_1), v_4] + [v_1, D(v_4)] = D[v_1, v_4] \Leftrightarrow c = \frac{-3}{4}.$$
Hence $c = \frac{-1}{2} = \frac{-3}{4}$, which is a contradiction.
\end{proof}


\begin{thebibliography}{99}
	
	\bibitem{aitbenhaddou2026lorentzian} M. Ait Ben Haddou and Y. Ayad, {\it {Lorentzian} Algebraic {Ricci} Solitons on Four-Dimensional Nilpotent {Lie} Groups}, Math. Notes, {\bf 119} (2026), no. 1, 133--158.\\
	\url{https://doi.org/10.1134/s0001434625604630}
	
	\bibitem{aitbenhaddou2012left} M. Ait Ben Haddou, M. Boucetta and H. Lebzioui, {\it Left-invariant {Lorentzian} flat metrics on {Lie} groups}, J. Lie Theory, {\bf 22} (2012), no. 1, 269--289.\\
	\url{https://doi.org/10.5802/jolt.669}
	
	\bibitem{bajo2024classification} I. Bajo, S. Benayadi and H. Lebzioui, {\it Classification of flat {Lorentzian} nilpotent {Lie} algebras}, Bull. Lond. Math. Soc., {\bf 56} (2024), no. 6, 2132--2149.\\
	\url{https://doi.org/10.1112/blms.13047}
	
	\bibitem{batat2017algebraic} W. Batat and K. Onda, {\it Algebraic {Ricci} solitons of three-dimensional {Lorentzian} {Lie} groups}, J. Geom. Phys., {\bf 114} (2017), 138--152.\\
	\url{https://doi.org/10.1016/j.geomphys.2016.11.018}
	
	\bibitem{bokan2015lorentz} N. Bokan, T. {\v{S}}ukilovi{\'c}, and S. Vukmirovi{\'c}, {\it Lorentz geometry of 4-dimensional nilpotent {Lie} groups}, Geom. Dedicata, {\bf 177} (2015), 83--102.\\
	\url{https://doi.org/10.1007/s10711-014-9980-4}
	
	\bibitem{etayo2016distinguished} F. Etayo and R. Santamar{\'{\i}}a, {\it Distinguished connections on {{\((J^2=\pm 1)\)}}-metric manifolds}, Arch. Math. (Brno), {\bf 52} (2016), no. 3, 159--203.\\
	\url{https://doi.org/10.5817/am2016-3-159}
	
	\bibitem{garcia} E. Garc{\'i}a‐R{\'i}o, R. Rodriguez-Gigirey and R. V{\'a}zquez-Lorenzo, {\it Four-dimensional {Lorentzian} algebraic {Ricci} solitons}, Proc. R. Soc. Edinb., Sect. A, Math. (2026), 1--56. (To appear)
	
	\bibitem{goze1996nilpotent} M. Goze and Y. Khakimdjanov, {\it Nilpotent {Lie} algebras}, Math. Appl., Dordr., {\bf 361} (1996).\\
	\url{https://doi.org/10.1007/978-94-017-2432-6}
	
	\bibitem{jablonski2011concerning} M. Jablonski, {\it Concerning the existence of {Einstein} and {Ricci} soliton metrics on solvable {Lie} groups}, Geom. Topol., {\bf 15} (2011), no. 2, 735--764.\\
	\url{https://doi.org/10.2140/gt.2011.15.735}
	
	\bibitem{jablonski2014homogeneous} M. Jablonski, {\it Homogeneous {Ricci} solitons are algebraic}, Geom. Topol., {\bf 18} (2014), no. 4, 2477--2486.\\ \url{https://doi.org/10.2140/gt.2014.18.2477}
	
	\bibitem{jablonski2015homogeneous} M. Jablonski, {\it Homogeneous {Ricci} solitons}, J. Reine Angew. Math., {\bf 699} (2015), 159--182.\\ \url{https://doi.org/10.1515/crelle-2013-0044}
	
	\bibitem{lauret2001ricci} J. Lauret, {\it {Ricci} soliton homogeneous nilmanifolds}, Math. Ann., {\bf 319} (2001), no. 4, 715--733.\\ \url{https://doi.org/10.1007/pl00004456}
	
	\bibitem{lauret2011ricci} J. Lauret, {\it {Ricci} soliton solvmanifolds}, J. Reine Angew. Math., {\bf 650} (2011), 1--21.\\
	\url{https://doi.org/10.1515/crelle.2011.001}
	
	\bibitem{magnin1986algebres} L. Magnin, {\it Sur les algebres de {Lie} nilpotentes de dimension $\leq$ 7}, J. Geom. Phys., {\bf 3} (1986), no. 1, 119--144.\\
	\url{https://doi.org/10.1016/0393-0440(86)90005-7}
	
	\bibitem{onda2014examples} K. Onda, {\it Examples of Algebraic {Ricci} Solitons in the Pseudo-Riemannian Case}, Acta Math. Hung., {\bf 144} (2014), no. 1, 247--265.\\ \url{https://doi.org/10.1007/s10474-014-0426-0}
	
	\bibitem{wang2022canonical} Y. Wang, {\it Canonical connections and algebraic {Ricci} solitons of three-dimensional {Lorentzian} {Lie} groups}, Chin. Ann. Math., Ser. B, {\bf 43} (2022), no. 3, 443--458.\\
	\url{https://doi.org/10.1007/s11401-022-0334-5}
	
	\bibitem{warner1983foundations} F. W. Warner, {\it Foundations of differentiable manifolds and {Lie} groups}, Glenview, {Illinois}-{London}: {Scott}, {Foresman} \& {Comp}. 270 p. {{\textsterling}} 4.24 (1971).
	
\end{thebibliography}
\end{document}